%% file: DiscrepancyIntegration.tex
\documentclass{scrreprt}

\usepackage[sumlimits]{amsmath}
\usepackage{amssymb,amsfonts,amsthm}
\usepackage[english]{babel}
\usepackage[latin1]{inputenc}
\usepackage[T1]{fontenc}
\usepackage{lmodern}
\usepackage{graphicx}
\usepackage{scrpage2}
\usepackage{setspace}
\usepackage{textcomp}
\usepackage{enumerate}
\usepackage{xfrac}
\usepackage{helvet}
\usepackage{bbm}
\usepackage{multirow}
\usepackage{bigstrut}
\usepackage{booktabs}
\usepackage{ifthen}
\usepackage{hyperref}

\newtheorem{thm}{Theorem}[chapter]
\newtheorem{prp}[thm]{Proposition}
\newtheorem{lem}[thm]{Lemma}
\newtheorem{cor}[thm]{Corollary}

\theoremstyle{definition}
\newtheorem{df}[thm]{Definition}
\newtheorem{rem}[thm]{Remark}

\def\Cx{\mathbb{C}}
\def\R{\mathbb{R}}
\def\N{\mathbb{N}}
\def\Q{[0,1)}
\def\Z{\mathbb{Z}}

\def\Dd{\mathbb{D}}
\def\B{\mathbb{B}}
\def\Fb{\mathbb{F}}
\def\1{\mathbbm{1}}
\def\Rn{\mathcal{R}_n}
\def\P{\mathcal{P}}
\def\C{\mathcal{C}}
\def\CS{\mathcal{CS}}
\def\S{\mathcal{S}}
\def\F{\mathcal{F}}
\def\A{\mathcal{A}}

\def\D{\mathcal{D}}
\def\V{\mathcal{V}}
\def\Dn{\mathfrak{D}}
\newcommand\dint{{\rm d}}
\newcommand{\im}{{\rm i}}
\DeclareMathOperator{\supp}{supp}

\DeclareMathOperator{\modulo}{mod}
\DeclareMathOperator{\wal}{wal}

\DeclareMathOperator{\dg}{deg}
\DeclareMathOperator{\err}{err}
\DeclareMathOperator*{\esssup}{ess-sup}
\DeclareMathOperator{\e}{e}

\begin{document}



\title{Discrepancy and integration in function spaces with dominating mixed smoothness}

\author{\textbf{Lev Markhasin} \\ Institut f\"ur Stochastik und Anwendungen \\ Fachbereich Mathematik \\ Universit\"at Stuttgart \\ Pfaffenwaldring 57 \\ 70569 Stuttgart \\ Germany \\ Email: lev.markhasin@mathematik.uni-stuttgart.de}

\maketitle

\tableofcontents
\newpage
\thispagestyle{empty}
\section*{Abstract}
Optimal lower bounds for discrepancy in Besov spaces with dominating mixed smoothness are known from the work of Triebel. Hinrichs proved upper bounds in the plane. In this work we systematically analyse the problem, starting with a survey of discrepancy results and the calculation of the best known constant in Roth's Theorem. We give a larger class of point sets satisfying the optimal upper bounds than already known from Hinrichs for the plane and solve the problem in arbitrary dimension for certain parameters considering a celebrated constructions by Chen and Skriganov which are known to achieve optimal $L_2$-norm of the discrepancy function. Since those constructions are $b$-adic, we give $b$-adic characterizations of the spaces. Finally results for Triebel-Lizorkin and Sobolev spaces with dominating mixed smoothness and for the integration error are concluded.\\[5mm]
\noindent{\footnotesize {\it 2010 Mathematics Subject Classification.} Primary 11K06,11K38,42C10,46E35,65C05.\\
{\it Key words and phrases.} discrepancy, numerical integration, quasi-Monte Carlo algorithms, Haar system, Hammersley point set, Chen-Skriganov point set, dominating mixed smoothness.}\\[5mm]
\textit{Acknowledgement:} The author wants to thank Aicke Hinrichs, Hans Triebel, Dmitriy Bilyk and Tino Ullrich. Research of the author was supported by Gerhard C. Starck Stiftung and Ernst Ludwig Ehrlich Studienwerk.

\chapter*{Introduction}
\markboth{Introduction}{Introduction}
\addcontentsline{toc}{chapter}{Introduction}
\input{Introduction}

\chapter{Preliminaries}
\input{BasicNotation}
\input{DiscrepancyFunction}
\input{NumericalIntegration}
\input{FunctionSpacesOfDominatingMixedSmoothness}
\input{bAdicBases}
\input{DigitalNets}
\input{DualityTheory}

\chapter{Characterization of $S_{pq}^r B([0,1)^d)$-spaces with $b$-adic Haar bases}
\input{HaarBases}
\input{Characterizations}

\chapter{$L_p$-discrepancy}
\input{L2Discrepancy}
\input{LpDiscrepancy}
\input{StarDiscrepancy}
\input{L1Discrepancy}
\input{Conclusion}

\chapter{Discrepancy in spaces with dominating mixed smoothness}
\input{LowerUpperBoundsSpqrB}
\input{DiscrepancyGeneralizedHammersleySpqrB}
\input{BestDiscrepancySpqrB}
\input{ConclusionDominatingMixed}

\chapter{Integration errors} \label{chpt_int_err}
\input{IntegrationErrors}

\input{Bibliography}
\end{document}

%% file: Introduction.tex
The analysis of uniformity of point distributions and the search for very well-distributed point sets play an important role in the context of so-called quasi-Monte Carlo methods. In numerical integration, point sets with low discrepancy can sometimes provide a significant improvement over so-called Monte Carlo methods, which generate point sets randomly. The discrepancy function measures deviation of a concrete given point set from a hypothetical perfectly uniform distribution. Low discrepancy guarantees a small integration error, as can be established by Koksma-Hlawka type inequalities.

Since the much celebrated result by Klaus Roth, the discrepancy theory has become a very popular subject of study. While only $L_p$-spaces have been studied initially, results for other function spaces (BMO, weighted $L_p$-spaces etc.) are emerging now. Nevertheless there is still much work to do for the classical spaces, especially on $L_1$-discrepancy and star discrepancy.

In this work we concentrate on spaces with dominating mixed smoothness, namely Besov spaces $S_{pq}^r B([0,1)^d)$, Triebel-Lizorkin spaces $S_{pq}^r F([0,1)^d)$ and Sobolev spaces $S_p^r H([0,1)^d)$. The best possible lower bounds for the spaces $S_{pq}^r B([0,1)^d)$ have been established by Hans Triebel in \cite{T10a} while his upper bounds were not optimal. There were gaps between the exponents of the lower and the upper bounds which have been closed in the plane by Aicke Hinrichs (\cite{Hi10}). We completely solve this problem for a certain interval of the smoothness parameter $r$, closing the gap in the exponents for arbitrary dimension. To do so, we calculate the norm of the discrepancy function for the explicit constructions of Chen and Skriganov, which are known to achieve the best possible asymptotic behavior for the $L_p$-discrepancy (Theorem \ref{thm_main_conclusion}). Additionally, we prove the upper bounds in the plane for a much larger class of point sets than has been given in \cite{Hi10}, thereby generalizing Hinrichs's result (Theorem \ref{thm_hammersley_disc}). Using embeddings of the spaces with dominating mixed smoothness we get results for Triebel-Lizorkin spaces and Sobolev spaces as well (Corollaries \ref{cor_f} and \ref{cor_h}).

From \cite{T10a} we have Koksma-Hlawka type inequalities for Besov spaces with dominating mixed smoothness. Therefore, another important result is Theorem \ref{thm_int_err} on the integration error.

Many prerequisites have to be established and used. Most significant are the characterizations for the Besov spaces with dominating mixed smoothness which generalize Triebel's results for higher dimension and greater bases (Theorem \ref{thm_besov_char}). The proof is equivalent to Triebel's proof, therefore, we just follow the original proofs without giving complete calculations in detail. Elsewise it would go beyond the scope of this work.

The point sets used for our purpose (generalized Hammersley and Chen-Skriganov point sets) are $b$-adic, i.e. in higher base, hence $b$-adic Haar bases must be used for the calculation.

Additionally, we give a slightly modified proof of Roth's theorem, calculating the best constant known so far (Theorem \ref{thm_roth_const}) improving the former one significantly. 

This work is arranged in the following way. The first chapter gives the necessary definitions, explanations (including well known facts, proofs and examples), alternatives and historical remarks. Also, literature recommendations are given. In that chapter we define the discrepancy function, spaces with dominating mixed smoothness, $b$-adic Haar and Walsh bases, digital nets and their dual counterparts.

The second chapter deals with the characterization of the Besov spaces with dominating mixed smoothness using $b$-adic Haar bases. We give a proof for the fact that they are a basis for $L_2$-spaces. Then we find equivalent $b$-adic definitions for the $S_{p q}^r B$-norms and finally prove the characterizations.

The third chapter summarizes the known results on $L_p$-discrepancy, including star discrepancy. Also the calculation of the constant for the lower bound of $L_2$-discrepancy can be found there. Additional historical remarks are given.

The fourth chapter deals with the calculation of upper bounds for the discrepancy in Besov spaces with dominating mixed smoothness of generalized Hammersley point sets and Chen-Skriganov point sets. Results for other spaces with dominating mixed smoothness are derived.

The last chapter concludes the results for the integration errors for spaces with dominating mixed smoothness.

%% file: BasicNotation.tex
\section{Basic notation} \label{BasicNotation}
Let $\N$ denote the set of the natural numbers and $\N_0 = \N \cup \{0\}$ and $\N_{-1} = \N \cup \{-1,0\}$. Let $\Z$ denote the set of all integers, $\R$ the set of all real numbers and $\Cx$ the complex plane. For a positive integer $b$ we mean by $\Z_b$ the ring of the residue classes modulo $b$, identified with $\{ 0, 1, \ldots, b - 1 \}$ with addition and multiplication modulo $b$. If $b$ is a prime power, then $\Fb_b$ is the finite field of order $b$. We will only use it for $b$ prime so we can identify it with $\Z_b$. $\Fb_b[x]$ will stand for the set of polynomials over $\Fb_b$.

By $d \in \N$ we will denote the dimension. We will either use the Euclidian space $\R^d$ or the $d$-dimensional unit cube $\Q^d$. The scalar product of $x,y \in \R^d$ is given by $x \, y = x_1 \, y_1 + \ldots + x_d \, y_d$ for $x = (x_1, \ldots, x_d), y = (y_1, \ldots, y_d)$.

Let $a,b \in \R^d$, then by $[a,b)$ we will mean the rectangular box $[a_1,b_1) \times \ldots \times [a_d,b_d)$ whenever $a_1 < b_1, \ldots, a_d < b_d$ where $a = (a_1,\ldots,a_d)$ and $b = (b_1,\ldots,b_d)$ and call it an interval. We will denote $0 = (0, \ldots, 0) \in \R^d$, but it will always be clear from the context if the real number $0$ or the vector $(0, \ldots, 0)$ is meant. For a measurable set $\Omega \subset \R^d$ we denote by $|\Omega|$ the volume of $\Omega$. For instance we have that $|[a,b)| = (b_1 - a_1) \cdot \ldots \cdot (b_d - a_d)$ is the volume of the interval $[a,b)$. Any measurability or integration in this work will be considered with respect to the Lebesgue measure. For any finite set $A$ we denote by $\# A$ the cardinality of $A$.
Let $\Omega \subset \R^d$, then by $\chi_\Omega$ we will denote the characteristic function of the set $\Omega$ defined as
\[ \chi_\Omega(x) = \begin{cases} 1 & \text{ if } x \in \Omega, \\ 0 & \text{ if } x \notin \Omega \end{cases} \]
for $x \in \R^d$. By $\log$ we denote the natural logarithm, by $\log_b$ the logarithm in base $b$.

We will use many constants which we will denote either by $c$ if we need only one or by $c_1$ and $c_2$ or $c$ and $C$ if we need two. If the constant changes in a proof we will use indices as well, increasing the index every time the constant changes. If we want to stress the fact that the constant depends on the dimension $d$, we use the notation $c_d$.

Since we are going to deal with irregularities of point distribution it is clear that we will use point sets in $\Q^d$. By $N \in \N$ we will denote the cardinality of a point set. An arbitrary point set with $N$ points will be denoted by $\P$.

We call a function on $\R^d$ rapidly decreasing if for all multi-indices $\alpha, \beta \in \N^d$ we have
\[ \sup_{x \in \R^d} |x^{\alpha} \, D^{\beta} f(x)| < \infty \]
where $D^{\beta}$ is the derivative of order $\beta$. Let $\S(\R^d)$ denote the Schwartz space of all complex-valued, rapidly decreasing, infinitely differentiable functions on $\R^d$ and $\S'(\R^d)$ its topological dual, the space of all tempered distributions on $\R^d$. Let $\D(\Q^d)$ consist of all complex-valued infinitely differentiable functions on $\R^d$ with compact support in the interior of $\Q^d$ and let $\D'(\Q^d)$ be its dual space of all distributions in $\Q^d$. For $0 < p \leq \infty$ we denote the Lebesgue spaces on $\R^d$ by $L_p(\R^d)$, (quasi-)normed by
\begin{align*}
& \left\| f | L_p(\R^d) \right\| = \left( \int_{\R^d} |f(x)|^p \dint x \right)^{\frac{1}{p}}, \quad 0 < p < \infty, \\
& \left\| f | L_{\infty}(\R^d) \right\| = \esssup_{x \in \R^d} |f(x)| = \inf \left\{ \sup_{x \in \R^d \backslash \Omega} |f(x)| \, : \, \Omega \subset \R^d, \, |\Omega| = 0 \right\}.
\end{align*}
For $1 \leq p \leq \infty$ these spaces are Banach spaces, for $0 < p < 1$ quasi-Banach spaces. Here $\esssup$ stands for essential supremum. Analogously we define $L_p(\Q^d)$. It is well known that for $1 \leq p \leq q \leq \infty$ we have the embedding
\[ L_q(\Q^d) \hookrightarrow L_p(\Q^d). \]
By the symbol "$\hookrightarrow$" we mean the following. Let $M_1$ and $M_2$ be two (quasi-) normed spaces. Then by $M_1 \hookrightarrow M_2$ we mean that $M_1 \subset M_2$ and there exists a constant $c > 0$ such that, for any $x \in M_1$ we have $\left\| x | M_2 \right\| \leq c \, \left\| x | M_1 \right\|$.

For any quasi-Banach space $V$ we denote by $V'$ its dual space, i.e. the set of all linear functionals $V \longrightarrow \Cx$. For any $1 \leq p < \infty$ and
\[ \frac{1}{p} + \frac{1}{p'} = 1 \]
we have that $(L_p(\R^d))' = L_{p'}(\R^d)$ and $(L_p([0,1)^d))' = L_{p'}([0,1)^d)$.

For $1 \leq p \leq \infty$ and $k \in \N_0$ we denote the Sobolev spaces by $W_p^k(\Q^d)$, normed by
\[ \left\| f | W_p^k(\Q^d) \right\| = \left( \sum_{|\alpha| \leq k} \left\| D^{\alpha} f | L_p(\Q^d) \right\|^p \right)^{\frac{1}{p}} \]
for $f \in L_p(\Q^d)$ satisfying $D^{\alpha} f \in L_p(\Q^d)$ for all $\alpha \in \N_0^d$ with $|\alpha| \leq k$. For $\alpha = (\alpha_1, \ldots, \alpha_d) \in \N_0^d$ we put $|\alpha| = \alpha_1 + \ldots + \alpha_d$. By $D^{\alpha}$ we denote the weak derivative of order $\alpha$, which is defined in the following sense. A measurable function $g$ on $\Q^d$ is the weak derivative of order $\alpha$ of $f$ if
\[ \int_{\Q^d} g(x) \varphi(x) \dint x = (-1)^{\alpha} \int_{\Q^d} f(x) D^{\alpha} \varphi(x) \dint x \]
for all infinitely differentiable functions $\varphi$ with compact support in $\Q^d$. The spaces $L_2(\Q^d)$ and $W_2^k(\Q^d)$ are Hilbert spaces. The inner product of $L_2(\Q^d)$ is given by
\[ \left\langle f, g\right\rangle_{L_2} = \int_{[0,1)^d} f(x) \overline{g(x)} \dint x \]
for $f, g \in L_2(\Q^d)$. The inner product of $W_2^k(\Q^d)$ is given by
\[ \left\langle f, g\right\rangle_{W_2^k} = \sum_{|\alpha| \leq k} \int_{[0,1)^d} D^{\alpha} f(x) \overline{D^{\alpha} g(x)} \dint x \]
for $f, g \in W_2^k(\Q^d)$. For any $p$ we have $L_p(\Q^d) = W_p^0(\Q^d)$.

For $\varphi \in \S(\R^d)$ we denote by
\[ \F \varphi(\xi) = (2 \pi)^{-\frac{d}{2}} \int_{\R^d} \varphi(x) \e^{-\im x \xi} \dint x, \quad \xi \in \R^d \]
the Fourier transform of $\varphi$. The inverse Fourier transform is given by
\[ \F^{-1} \varphi(x) = (2 \pi)^{-\frac{d}{2}} \int_{\R^d} \varphi(\xi) \e^{\im x \xi} \dint \xi, \quad x \in \R^d. \]
We extend $\F$ and $\F^{-1}$ in the usual way from $\S$ to $\S'$. For $f \in \S'(\R^d)$,
\[ \F f(\varphi) = f(\F \varphi), \quad \varphi \in \S(\R^d) \]
and
\[ \F^{-1} f(\varphi) = f(\F^{-1} \varphi), \quad \varphi \in \S(\R^d). \]

%% file: DiscrepancyFunction.tex
\section{Irregularities of point distribution}
In different contexts one often asks what is the most uniform way of distributing a finite point set in $[0,1)^d$ and how big is the irregularity of such a distribution. Questions of this kind were motivated by problems in number theory. But to answer such questions one has to clarify the notion of uniformity and irregularity first. In this section we give an introduction into some of the concepts concentrating on the discrepancy of left lower corners since it is the central part of the results given in this work. We advise the interested reader to study such monographs as \cite{DP10}, \cite{M99}, \cite{NW10}, \cite{KN74} and the references given there. 
\subsection{Discrepancy function of left lower corners}
\begin{df} \label{def_discrepancy_function}
Let $N$ be some positive integer and let $\P$ be a point set in the unit cube $[0,1)^d$ with $N$ points. Then the discrepancy function $D_{\P}$ is defined as
\begin{align}
D_{\P}(x) = \frac{1}{N} \sum_{z \in \P} \chi_{[0,x)}(z) - x_1 \cdot \ldots \cdot x_d.
\end{align}
for any $x = (x_1, \ldots, x_d) \in [0,1)^d$.
\end{df}
We also will call it the discrepancy function of left lower corners if we will have to distinguish it from other kinds of discrepancy functions though it will be the one used throughout this work. The term $\sum_z \chi_{[0,x)}(z)$ is equal to $\# (\P \cap [0,x))$. The discrepancy function measures the deviation of the number of points of $\P$ in $[0,x)$ from the fair number of points $N |[0,x)| = N \, x_1 \cdot \ldots \cdot x_d$ which would be achieved by a (practically impossible) perfectly uniform distribution of the points of $\P$, normalized by the total number of points. The following image shows the $2$-dimensional case.

There, we have a point set $\P$ with $21$ points and $5$ points of $\P$ are in the interval $[0,x)$ of volume $0.26$. So we have $D_{\P}(x) = \frac{5}{21} - 0.26 \approx -0.02$.
\begin{center} \includegraphics[height = 8cm]{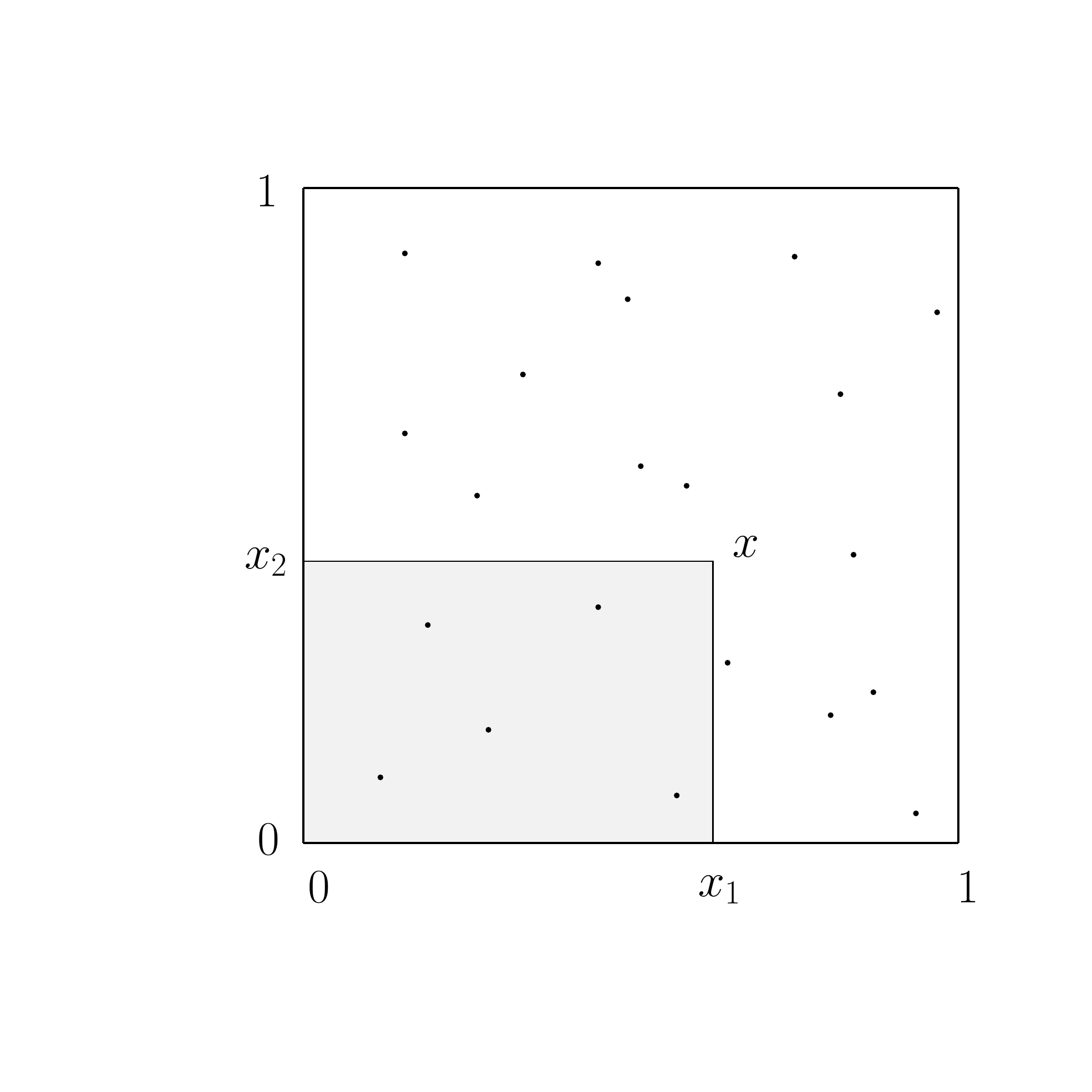} \end{center}

Sometimes instead of $D_{\P}$ the discrepancy function is introduced as $N \cdot D_{\P}$. Usually one is interested in calculating the norm of the discrepancy function in some normed space of functions on $[0,1)^d$ to which the discrepancy function belongs. Then there are two major tasks to work on with the discrepancy function. Before we describe them we need the following notation.
\begin{df} \label{M_discr}
Let $M([0,1)^d)$ be some Banach space of functions on $[0,1)^d$ such that, for every positive integer $N$ and every point set $\P$ in $[0,1)^d$ with $N$ points, the discrepancy function $D_{\P}$ belongs to $M([0,1)^d)$. Then we denote
\begin{align}
D^M(N) = \inf_{\# \P = N} \left\| D_{\P} | M([0,1)^d) \right\|
\end{align}
as $M$-discrepancy.
\end{df}
The aforementioned tasks are to find functions $f_1$ and $f_2$ such that, there exist constants $c_1, c_2 > 0$ and for every positive integer $N$, we have
\[ c_1 \, f_1(N) \leq D^M(N) \leq c_2 \, f_2(N). \]
The first task is to prove that for every $N$ and every point set $\P$ in $[0,1)^d$ with $N$ points we have
\[  \left\| D_{\P} | M([0,1)^d) \right\| \geq c_1 \, f_1(N). \]
The second major task is to find point sets with the best possible discrepancy, i.e. to prove that for every $N$ there exists a point set $\P$ in $[0,1)^d$ with $N$ points such that
\[  \left\| D_{\P} | M([0,1)^d) \right\| \leq c_2 \, f_2(N). \]
Ideally, we want $f_1 = f_2$.

In the case $d = 1$ nothing beats the set of $N$ equidistant points
\[ \left\{ \frac{1}{2N} + \frac{k}{N} \, : \, k = 0, \ldots, N - 1 \right\}. \]
One easily calculates the value of the discrepancy function.

\subsection{Generalized discrepancy functions}
The definition for the discrepancy function given above is not a very general one. In fact, it is very specific. There are several other ways to study irregularity of point distributions. For instance we introduced it only for left lower corners $[0,x)$ while we could have defined it for any other class of geometrical figures. Let $\A$ be such a possible class, e.g. the class of all axis-parallel rectangular boxes, the class of all rectangular boxes, the class of all balls and so on, then the discrepancy function of a point set $\P$ in $[0,1)^d$ with $N$ points could be defined for $A \in \A$ as
\[ D_{\P}^{\A}(A) = \frac{1}{N} \sum_{z \in \P} \chi_{A}(z) - |A|. \]
For more information on other approaches of this kind the reader is referred to \cite{M99} and the references given there. In this context we only consider the discrepancy function of the class of all axis-parallel rectangular boxes of the form $[a,b)$ for $a,b \in [0,1)^d$. The discrepancy function
\[ D_{\P}^{ap}([a,b)) = \frac{1}{N} \sum_{z \in \P} \chi_{[a,b)}(z) - |[a,b)| \]
seems to be more general than the discrepancy function of the left lower corners. But after considering the following well known fact, it becomes clear that the discrepancy function of left lower corners and the discrepancy function of all axis-parallel rectangular boxes are connected and will give us at least in $\sup$-norm the same results if we are not interested in the exact constant of proportionality.

\begin{prp} \label{prp_apboxes_corners}

For any finite set $\P$ in $[0,1)^d$ and any rectangular box $[a,b)$ there exists a point $x \in [0,1)^d$ such that, we have
\[ D_{\P}(x) \leq D_{\P}^{ap}([a,b)) \leq 2^d \, D_{\P}(x). \]

\end{prp}
Before we prove this fact we consider another well known fact that shows that the discrepancy function can be interpreted as an additive signed measure.

\begin{lem}

Let $\A$ be some class of geometric figures as above and $A,B \in \A$. Then, if $A$ and $B$ are disjoint, then
\[ |D_{\P}^{\A}(A \cup B)| \leq |D_{\P}^{\A}(A)| + |D_{\P}^{\A}(B)| \]
while, if $A \subset B$, then
\[ |D_{\P}^{\A}(A \backslash B)| \leq |D_{\P}^{\A}(A)| + |D_{\P}^{\A}(B)|. \]

\end{lem}

\begin{proof}
Clearly, $\# \left( \P \cap (A \cup B) \right) = \# \left( \P \cap A \right) + \# \left( \P \cap B \right)$ and $|A \cup B| = |A| + |B|$ so the first part follows. The second part follows analogously.

\end{proof}

\begin{proof}[Proof of Proposition \ref{prp_apboxes_corners}]
We give the idea for $d = 2$. The general case will follow analogously. We represent the rectangular box as
\begin{multline*}
[a_1,b_1) \times [a_2,b_2) \\ = \left( [0,b_1) \times [0,b_2) \backslash [0,a_1) \times [0,b_2) \right) \backslash \left( [0,b_1) \times [0,a_2) \backslash [0,a_1) \times [0,a_2) \right)
\end{multline*}
and the proposition follows from the previous lemma.

\end{proof}

Another generalization of the discrepancy function is the so called weighted discrepancy function, which can be defined as follows. Let $a=(a_z)_{z\in\cal P}$ be a system of real numbers associating a weight $a_z$ with a point $z\in\cal P$. Then the weighted discrepancy function is defined as
\begin{align} \label{weighted_disc}
D_{\P,a}(x) =  \sum_{z \in \P} a_z \chi_{[0,x)}(z) - x_1 \cdot \ldots \cdot x_d.
\end{align}
The discrepancy function defined by Definition \ref{def_discrepancy_function} is obtained in the case that all points of $\P$ have weight $\frac{1}{N}$. Of course there are even more different approaches for the discrepancy function, some can be found for example in \cite{M99}.

\subsection{Uniform distribution of infinite sequences}
As a conclusion of this section we want to consider the uniform distribution of infinite sequences in the one-dimensional case. For more information on this topic we refer to \cite{KN74}. Let $u = (u_1,u_2, \ldots)$ be an infinite sequence of points in $[0,1)$.

\begin{df}
The sequence $u = (u_1,u_2, \ldots)$ is called uniformly distributed in $[0,1)$ if we have for each $x \in [0,1)$ that
\[ \lim_{N \rightarrow \infty} \left( \frac{1}{N} |\{ u_1, \ldots, u_N \} \cap [0,x)| \right) = x. \]
\end{df}
One easily proves the following equivalent formulation (\cite{W16}) using standard methods. A sequence $u = (u_1,u_2, \ldots)$ is uniformly distributed in $[0,1)$ if and only if we have for any Riemann-integrable function $f: \, [0,1) \longrightarrow \R$ that
\[ \lim_{N \rightarrow \infty} \left( \frac{1}{N} \sum_{i = 1}^N f(u_i) \right) = \int_0^1 f(x) \dint x. \]
Another equivalent formulation is the so called Weyl criterion. A sequence $u = (u_1,u_2, \ldots)$ is uniformly distributed in $[0,1)$ if and only if we have for all integers $k \neq 0$,
\[ \lim_{N \rightarrow \infty} \left( \frac{1}{N} \sum_{j = 1}^N \e^{2 \pi \im k u_j} \right) = 0. \]
We can use this criterion for the following fact.

\begin{prp}
Let $\theta$ be an irrational number. The sequence $u = (u_1,u_2,\ldots)$ given by $u_n = \{ \theta n \}$ is uniformly distributed in $[0,1)$.
\end{prp}
By $\{ x \}$ we mean the fractional part of $x$. We quote the proof from \cite{M99}.

\begin{proof}
We have $\e^{2 \pi \im k u_n} = \e^{2 \pi \im k \theta n}$ and
\[ \sum_{j = 1}^n \e^{2 \pi \im k u_j} = \frac{\e^{2 \pi \im k \theta (n + 1)} - \e^{2 \pi \im k \theta}}{\e^{2 \pi \im k \theta} - 1}. \]
Since $k \theta$ is not an integer, we have
\[ \left| \frac{\e^{2 \pi \im k \theta (n + 1)} - \e^{2 \pi \im k \theta}}{\e^{2 \pi \im k \theta} - 1} \right| \leq \frac{2}{\left| \e^{2 \pi \im k \theta} - 1 \right|} \]
giving us the condition of the Weyl criterion, and therefore, the uniform distribution of the sequence.

\end{proof}

After having seen a uniformly distributed sequence it is necessary to compare the uniformity or the nonuniformity of the sequences.
\begin{df}
The discrepancy function of an infinite sequence $u$ in $[0,1)$ is the function
\[ \Delta_{u,n}(x) = \frac{1}{n} \sum_{j = 1}^n \chi_{[0,x)}(u_j) - x. \]
\end{df}

There is a strong connection between uniform distribution of infinite sequences and uniformly distributed finite point sets. The following well known (e.g. \cite{M99}) result summarizes this connection. We mention that a similar connection can be established between $d$-dimensional finite point sets and $(d-1)$-dimensional infinite sequences.

\begin{prp} Let $N$ be a positive integer.

\begin{enumerate}[(i)]
	\item Let $u$ be an infinite sequence in $[0,1)$. Then there exists a point set $\P$ in $[0,1)^2$ with $N$ points such that, \label{prp_connect_set_seq_i}
	\[ N \sup_{x \in [0,1)^2} |D_{\P}(x)| \leq \sup_{x \in [0,1)} \max_{1 \leq k < N} k \, \Delta_{u,k}(x) + 1. \]
	\item Let $\P$ be a point set in $[0,1)^2$ with $N$ points. Then there exists an infinite sequence $u$ in $[0,1)$ such that, \label{prp_connect_set_seq_ii}
	\[ \sup_{x \in [0,1)} \max_{1 \leq k < N} k \, \Delta_{u,k}(x) \leq 2N \sup_{x \in [0,1)^2} |D_{\P}(x)|. \]
\end{enumerate}

\end{prp}

\begin{proof}
We put
\[ \P = \left\{ \left( \frac{k}{N},u_k \right) \, : \, k = 1, \ldots, N \right\}. \]
Then for $x = (x_1,x_2) \in [0,1)^2 $ we find an integer $m$ such that, $m < N x_1 \leq m + 1$ and we have
\begin{align*}
N D_{\P}(x) & = \# \{ k = 1, \ldots, N \, : \, \frac{k}{N} < x_1, \, u_k < x_2 \} - N x_1 x_2 \\
& = \# \{ k = 1, \ldots, N \, : \, k < N x_1, \, u_k < x_2 \} - N x_1 x_2 \\
& = \# \{ u_1, \ldots, u_m \} \cap [0,x_2) - N x_1 x_2 \\
& = \# \{ u_1, \ldots, u_m \} \cap [0,x_2) - m x_2 + (m - N x_1) x_2 \\
& = m \, \Delta_{u,m}(x_2) + (m - N \, x_1) x_2 \\
& \leq m \, \Delta_{u,m}(x_2) + 1.
\end{align*}
Clearly, $1 \leq m \leq N - 1$. Analogously, we prove $N \, D_{\P}(x) \geq - m \, \Delta_{u,m + 1}(x_2) - 1$ and taking the supremum on both sides gives us \eqref{prp_connect_set_seq_i}.

For \eqref{prp_connect_set_seq_ii} we denote $z_j = (x_j,y_j) \in \P$ for $j = 1, \ldots, N$. Then we put $u_j = y_j$. Without loss of generality suppose that we have $0 \leq x_1 \leq x_2 \leq \ldots \leq x_N < 1$. Then the statement follows analogously to the proof of \eqref{prp_connect_set_seq_i}.

\end{proof}

In \cite{W16} one finds results for uniform distribution of infinite sequences in $[0,1)^d$.

\subsection{Historical remarks}
The question concerning uniform distribution of infinite sequences was first raised by Weyl in his article \cite{W16} in the beginning of the last century. Therefore, the roots of discrepancy theory lie in number theory. Since then it has affected different mathematical branches like function theory, probability theory, numerical analysis, functional analysis, topological algebra and more. Through the $30$s and $40$s today's theory of point distribution emerged through the work of such mathematicians as van der Corput, van Aardenne-Ehrenfest and others. Van Aardenne-Ehrenfest gave a negative answer to the question if there would exist a sequence whose discrepancy function would stay bounded as $N$ approaches infinity.  Roth proved maybe the most famous result in 1954, his lower bound for the $L_2$-discrepancy. Roth also was the one who introduced the discrepancy function of left lower corners. His motivation was the improvement of van Aardenne-Ehrenfest's lower bound for sequences. In the early $70$s the $2$-dimensional problem was already solved quite satisfactorily while in arbitrary dimension the problem is far from being solved by now.

The possible applications that emerged throughout the years were in financial calculations, computer graphics, computational physics and quasi-Monte Carlo methods. The generalizations of the discrepancy function were animated by Erd\H{o}s in the $60$s. The star discrepancy in the $2$-dimensional case was solved by van der Corput and Schmidt while in arbitrary dimension improvements came throughout the years with most recent results by Bilyk, Lacey and Vagharshakyan. The best upper bound is by Halton. $L_p$-discrepancy was improved by Davenport, Roth, Schmidt, Hal�sz, Chen, Beck and others. We will present these results in much detail in a later chapter.

The starting point of quasi-Monte Carlo methods for numerical integration was the Koksma-Hlawka inequality (\cite{K43} in the one-dimensional case and \cite{H61} in arbitrary dimension). First constructions of digital nets were given by Sobol', Faure and Niederreiter in the $60$-$80$s. Niederreiter's paper \cite{N87} from the late $80$s is regarded as the initiation of the theory of nets. The explicit constructions for the best possible $L_2$-discrepancy have been given by Chen and Skriganov in 2002. Other notable constructions are due to van der Corput, Halton, Hammersley, Zaremba, Faure, Sobol' and many others. Triebel started to study discrepancy in the context of function spaces (\cite{T10a}) recently and Hinrichs added results in this direction. Spaces with dominating mixed smoothness, Hardy spaces, Orlicz spaces, weighted  $L_p$ spaces, BMO spaces and others are subjects of study (see also \cite{B11}). So the topic continues to attract interest of researchers from different points of view.

%% file: NumericalIntegration.tex
\section{Uniformly distributed point sets for numerical integration}
The problem of numerical integration occurs in many practical and theoretical contexts, ranging from computer graphics and physics over engineering to chemistry and biology. Often it is not possible to calculate an integral analytically. Then one tries to approximate it aiming to reduce the error while using as little data as possible. Suppose we have a function $f \, : \Q^d \longrightarrow \R$ and our goal is to approximate the number
\[ \int_{\Q^d} f(x) \dint x \]
with quadrature formulas of the form
\[ \frac{1}{N} \sum_{i = 1}^N f(x_i) \approx \int_{\Q^d} f(x) \dint x \]
where $x_1, \ldots, x_N$ are some points in $\Q^d$. The question that then arises is how many points are necessary and how should they be distributed to assure that the integration error is not greater than some given $\varepsilon > 0$. The Koksma-Hlawka inequality (\cite{K43} and \cite{H61}) gives the connection between the integration error and the discrepancy of a point set. Let $\P = \{ x_1, \ldots, x_N \}$. It states
\begin{align}
\left| \int_{\Q^d} f(x) \dint x - \frac{1}{N} \sum_{i = 1}^N f(x_i) \right| \leq V_{p'}(f) \, \left\| D_{\P} | L_p(\Q^d) \right\|
\end{align}
where $V_{p'}(f)$ is determined solely by $f$ and $p$. The interested reader is referred to \cite{H61} or \cite[Chapter 2]{KN74} for more information on $V_{p'}(f)$. We just mention that in the one-dimensional case we have $V_{p'}(f) = \left\| f' | L_{p'} \right\|$.

This makes it clear that in order to guarantee best possible results for numerical integration it is an important task to find good point sets, good in the sense of discrepancy. If we compare this approach with Monte Carlo methods where one uses random points then typically the star discrepancy of such point sets is $\frac{1}{\sqrt{N}}$ for $d = 2$ (see \cite{M99}) and with high probability even much worse while, as we will see later, one can find point sets of far better discrepancy.

%% file: FunctionSpacesOfDominatingMixedSmoothness.tex
\section{Function spaces with dominating mixed smoothness}
A significant part of this work deals with the discrepancy function in spaces with dominating mixed smoothness. This section will give a necessary introduction, containing definitions and embeddings which will be used later. The spaces with dominating mixed smoothness go back to the $60$s when Nikol'sky introduced the Sobolev spaces with dominating mixed smoothness as well as the Besov spaces with dominating mixed smoothness, though not yet in full generality. The theory for the Besov spaces with dominating mixed smoothness goes back to Amanov with preliminary work by Lizorkin, Dzabrailov and many others. The references for this topic are \cite{T10a}, \cite{A76} and \cite{ST87} as well as the references given there. One also finds some more historical remarks there.

Let $\varphi_0 \in \mathcal{S}(\R)$ satisfy $\varphi_0(t) = 1$ for $|t| \leq 1$ and $\varphi_0(t) = 0$ for $|t| > \frac{3}{2}$. Let
\[ \varphi_k(t) = \varphi_0(2^{-k} t) - \varphi_0(2^{-k + 1} t) \]
where $t \in \R, \, k \in \N$ and
\[ \varphi_k(t) = \varphi_{k_1}(t_1) \ldots \varphi_{k_d}(t_d) \]
where $k = (k_1,\ldots,k_d) \in \N_0^d, \, t = (t_1,\ldots,t_d) \in \R^d$.
The functions $\varphi_k$ are a dyadic resolution of unity since
\[ \sum_{k \in \N_0^d} \varphi_k(x) = 1 \]
for all $x \in \R^d$. The functions $\mathcal{F}^{-1}(\varphi_k \mathcal{F} f)$ are entire analytic functions for any $f \in \mathcal{S}'(\R^d)$.
\begin{df}
Let $0 < p,q \leq \infty$ and $r \in \R$. Let $(\varphi_k)$ be a dyadic resolution of unity.
\begin{enumerate}[(i)]
	\item The Besov space with dominating mixed smoothness $S_{pq}^r B(\R^d)$ consists of all $f \in \S'(\R^d)$ with finite quasi-norm
  \[ \left\| f | S_{pq}^r B(\R^d) \right\| = \left( \sum_{k \in \N_0^d} 2^{r |k| q} \left\| \mathcal{F}^{-1}(\varphi_k \mathcal{F} f) | L_p(\R^d) \right\|^q \right)^{\frac{1}{q}} \]
  with the usual modification if $q = \infty$.
  \item The Besov space with dominating mixed smoothness $S_{pq}^r B(\Q^d)$ consists of all $f \in \D'([0,1)^d)$ with finite quasi-norm
\[ \left\| f | S_{pq}^r B(\Q^d) \right\| = \inf \left\{ \left\| g | S_{pq}^r B(\R^d) \right\| : \: g \in S_{pq}^r B(\R^d), \: g|_{\Q^d} = f \right\}. \]
\end{enumerate}
\end{df}
\pagebreak
\begin{df}
Let $0 < p < \infty$, $0 < q \leq \infty$ and $r \in \R$. Let $(\varphi_k)$ be a dyadic resolution of unity.
\begin{enumerate}[(i)]
  \item The Triebel-Lizorkin space with dominating mixed smoothness $S_{pq}^r F(\R^d)$ consists of all $f \in \S'(\R^d)$ with finite quasi-norm
  \[ \left\| f | S_{pq}^r F(\R^d) \right\| = \left\| \left( \sum_{k \in \N_0^d} 2^{r |k| q} \left| \mathcal{F}^{-1}(\varphi_k \mathcal{F} f)(\cdot) \right|^q \right)^{\frac{1}{q}} | L_p(\R^d) \right\| \]
  with the usual modification if $q = \infty$.
  \item The Triebel-Lizorkin space with dominating mixed smoothness $S_{pq}^r F(\Q^d)$ consists of all $f \in \D'(\Q^d)$ with finite quasi-norm
\[ \left\| f | S_{pq}^r F(\Q^d) \right\| = \inf \left\{ \left\| g | S_{pq}^r F(\R^d) \right\| : \: g \in S_{pq}^r F(\R^d), \: g|_{\Q^d} = f \right\}. \]
\end{enumerate}
\end{df}
The spaces $S_{pq}^r B(\R^d), \, S_{pq}^r F(\R^d), \, S_{pq}^r B(\Q^d)$ and $S_{pq}^r F(\Q^d)$ are quasi-Banach spaces. They are independent of the choice of the dyadic resolution of unity since different resolutions give equivalent quasi-norms. We will give some characterizations for the spaces $S_{pq}^r B(\Q^d)$ in the next chapter. We will see that the dyadic definition is equivalent to a $b$-adic definition.
\begin{df} \label{sobolev_df}
Let $0 < p < \infty$ and $r \in \R$. Then
\[ S_p^r H(\Q^d) = S_{p \, 2}^r F(\Q)^d) \]
is called Sobolev space with dominating mixed smoothness. For $r \in \N_0$ it is denoted by $S_p^r W(\Q^d)$ and is called classical Sobolev space with dominating mixed smoothness.
\end{df}
An equivalent norm for $S_p^r W(\Q^d)$ is
\[ \sum_{\alpha \in \N_0^d; \, 0 \leq \alpha_i \leq r} \left\| D^{\alpha} f | L_p(\Q^d) \right\|. \]
Of special interest is the case $r = 0$ since
\begin{align} \label{norm_SpqrF_Lp}
S_p^0 W(\Q^d) = L_p(\Q^d).
\end{align}

For the following embeddings the reader is referred to \cite[Remark 6.28]{T10a} and \cite[Proposition 2.3.7]{Hn10}.

\begin{prp} \label{embeddings_SpqrBF}

Let $r \in \R$.
\begin{enumerate}[(i)]
	\item For $0 < p < \infty, \, 0 < q \leq \infty$ we have
  \[ S_{p, \min(p,q)}^r B(\Q^d) \hookrightarrow S_{pq}^r F(\Q^d) \hookrightarrow S_{p, \max(p,q)}^r B(\Q^d). \]
  \item For $0 < p_2 \leq q \leq p_1 < \infty$ we have
  \[ S_{p_1 q}^r F(\Q^d) \hookrightarrow S_{qq}^r B(\Q^d) \hookrightarrow S_{p_2 q}^r F(\Q^d). \]
\end{enumerate}

\end{prp}

For our purposes the following embeddings will be helpful.

\begin{cor} \label{cor_emb_BF}

Let $0 < p,q < \infty$ and $r \in \R$. Then we have
\[ S_{\max(p,q),q}^r B(\Q^d) \hookrightarrow S_{pq}^r F(\Q^d) \hookrightarrow S_{\min(p,q),q}^r B(\Q^d). \]

\end{cor}

\begin{proof}
First suppose that $p \geq q$ then from the first part of Proposition \ref{embeddings_SpqrBF} we get
\[ S_{pq}^r B(\Q^d) \hookrightarrow S_{pq}^r F(\Q^d) \]
and from the second part we get
\[ S_{pq}^r F(\Q^d) \hookrightarrow S_{qq}^r B(\Q^d). \]
If instead $p < q$ then analogously we have
\[ S_{pq}^r F(\Q^d) \hookrightarrow S_{pq}^r B(\Q^d) \]
and
\[S_{qq}^r B(\Q^d) \hookrightarrow S_{pq}^r F(\Q^d). \]
\end{proof}

For $1 \leq p, q \leq \infty$ let
\[ \frac{1}{p} + \frac{1}{p'} = \frac{1}{q} + \frac{1}{q'} = 1.\]
In \cite[Proposition 2.3.15]{Hn10} and \cite[(1.75), (2.272), (6.36)]{T10a} we find results on duality for function spaces with dominating mixed smoothness.

\begin{prp} \label{dom_mix_smoo_duality}

\mbox{}
\begin{enumerate}[(i)]
	\item Let $1 \leq p, q < \infty$ and let $r \in \R$. Then we have
	\[ (S_{p q}^r B(\R^d))' = S_{p' q'}^{-r} B(\R^d), \]
        \item Let $1 \leq p, q < \infty$ and let $\frac{1}{p} - 1 < r < \frac{1}{p}$. Then we have
	\[ (S_{p q}^r B([0,1)^d))' = S_{p' q'}^{-r} B([0,1)^d), \]
	\item Let $1 < p, q < \infty$ and let $r \in \R$. Then we have
	\[ (S_{p q}^r F(\R^d))' = S_{p' q'}^{-r} F(\R^d), \]
	\item Let $1 < p < \infty$ and let $r \in \R$. Then we have
	\[ (S_p^r H(\R^d))' = S_{p'}^{-r} H(\R^d). \]
\end{enumerate}

\end{prp}

%% file: bAdicBases.tex
\section{$b$-adic bases}
We will deal with the discrepancy function in function spaces with dominating mixed smoothness. Our approach will be to consider constructions given by Chen and Skriganov which are digital nets. As can be seen later there is a necessity to use large bases $b$ for such constructions. Therefore, we can not use dyadic Haar bases but need to use generalizations.
\subsection{Haar bases} \label{notation_eta}
Let $b \geq 2$ be an integer. We start by fixing some notation. We put $\Dd_j = \{ 0,1, \ldots, b^j - 1 \}$ and $\B_j = \{ 1, \ldots, b - 1 \}$ for $j \in \N_0$ and $\Dd_{-1} = \{ 0 \}$ and $\B_{-1} = \{ 1 \}$. For $j = (j_1, \dots, j_d) \in \N_{-1}^d$ let $\Dd_j = \Dd_{j_1} \times \ldots \times \Dd_{j_d}$ and $\B_j = \B_{j_1} \times \ldots \times \B_{j_d}$. We put $s = \# \{ i = 1, \ldots, d: \, j_i \neq -1 \}$ and choose the unique subsequence $(\eta_\nu)_{\nu = 1}^s$ of $(1, \ldots, d)$ such that, for all $\nu = 1, \ldots, s$, we have $j_{\eta_{\nu}} \neq -1$ while all other $j_i$ are equal to $-1$. We generalize a notation from Section \ref{BasicNotation}, writing $|j| = j_{\eta_1} + \ldots + j_{\eta_s}$. We continue with the definition of $b$-adic intervals.
\begin{df}
\mbox{}
\begin{enumerate}[(i)]
	\item For $j \in \N_{-1}$ and $m \in \Dd_j$ we call the interval
  \[ I_{jm} = \big[ b^{-j} m, b^{-j} (m+1) \big) \]
  the $m$-th $b$-adic interval in $\Q$ on level $j$.
  \item For $j \in \N_{0}$, $m \in \Dd_j$ and any $k = 0, \ldots, b - 1$ we call the interval $I_{jm}^k = I_{j + 1,bm + k}$ the $k$-th child of $I_{jm}$. The interval $I_{jm}$ is then called the parent of $I_{jm}^k$.
  \item We put $I_{-1,0}^{-1} = I_{-1,0} = \Q$ and call $I_{-1,0}$ the $0$-th $b$-adic interval in $\Q$ on level $-1$ and the parent of its only child $I_{-1,0}^{-1}$.
  \item For $j \in \N_{-1}^d$ and $m = (m_1, \ldots, m_d) \in \Dd_j$ we call $I_{jm} = I_{j_1 m_1} \times \ldots \times I_{j_d m_d}$ the $m$-th $b$-adic interval in $\Q^d$ on level $j$.
  \item Let $j \in \N_{-1}^d$ and $k = (k_1, \ldots, k_d)$ where $k_i \in \{ 0, \ldots, b - 1 \}$ if $j_i \in \N_0$ and $k_i = -1$ if $j_i = -1$ for any $i = 1, \ldots, d$. Then $I_{jm}^k = I_{j_1 m_1}^{k_1} \times \ldots \times I_{j_d m_d}^{k_d}$ will be called $k$-th child of $I_{jm}$ and $I_{jm}$ the parent of $I_{jm}^k$.
  \item Let $j \in \N_{-1}^d$. For any $m \in \Dd_j$ we call the number $|j|$ the order of the $b$-adic interval $I_{jm}$.
\end{enumerate}
\end{df}

\begin{rem}

Let $j \neq -1$. The $b$-adic interval $I_{jm}$ has length $b^{-j}$ while the length of its children is $b^{-j - 1}$. The children are disjoint, the union of all children of one interval gives the parent itself so the parents are partitioned in their $b$ children. For $j \in \N_{-1}^d$ the volume of a $b$-adic interval is $b^{-|j|}$. Again the children are disjoint and their union gives the parent.

\end{rem}

\begin{df}
\mbox{}
\begin{enumerate}[(i)]
  \item Let $j \in \N_{0}$, $m \in \Dd_j$ and $l \in \B_j$. Let $h_{jml}$ be the function on $\Q$ with support in $I_{jm}$ and the constant value $\e^{\frac{2\pi \im}{b} l k}$ on $I_{jm}^k$ for any $k=0, \ldots, b - 1$. We call $h_{jml}$ a $b$-adic Haar function on $\Q$.
  \item We put $h_{-1,0,1} = \chi_{I_{-1,0}}$ on $\Q$ and call it a $b$-adic Haar function on $\Q$ as well.
  \item The functions $h_{jml}, \, j \in \N_{-1}, \, m \in \Dd_j, \, l \in \B_j$ are called $b$-adic Haar system on $\Q$.
  \item Let $j \in \N_{-1}^d$, $m \in \Dd_j$ and $l = (l_1, \ldots, l_d) \in \B_j$. The function $h_{jml}$ given as the tensor product $h_{jml}(x) = h_{j_1 m_1 l_1}(x_1) \ldots h_{j_d m_d l_d}(x_d)$ for $x = (x_1, \ldots, x_d) \in \Q^d$ is called a $b$-adic Haar function on $\Q^d$.
  \item The functions $h_{jml}, \, j \in \N_{-1}^d, \, m \in \Dd_j, \, l\in\B_j$ are called $b$-adic Haar system on $\Q^d$.
\end{enumerate}
\end{df}
In the dyadic case, i.e. $b = 2$ for $j \in \N_0$ there is only one value taken by $l$, which is $1$. Therefore, we omit $l$ in the notation in that case and write $h_{jm}$ instead of $h_{jml}$.

\begin{thm} \label{thm_dyadic_haar_basis}

The system
\begin{align} \label{dyadic_haarbasis}
\left\{ 2^{\frac{|j|}{2}} h_{jm} \, : \, j \in \N_{-1}^d, \, m \in \Dd_j \right\}
\end{align}
is an orthonormal basis of $L_2(\Q^d)$, an unconditional basis of $L_p(\Q^d)$ for $1 < p < \infty$ and a conditional basis of $L_1(\Q^d)$. For any function $f \in L_2(\Q^d)$ we have
\begin{align} \label{dyadic_parsevals_equation}
\left\| f | L_2 (\Q^d) \right\|^2 = \sum_{j \in \N_{-1}^d} 2^{|j|} \sum_{m \in \Dd_j} |\mu_{jm}|^2
\end{align}
where
\begin{align}
\mu_{jm} = \mu_{jm}(f) = \int_{\Q^d} f(x) h_{jm}(x) \, \dint x
\end{align}

\end{thm}

\begin{rem}

The expression \eqref{dyadic_parsevals_equation} is Parseval's equation. The dyadic Haar system was given first by Haar in \cite{Ha10}. Schauder proved in \cite{S28} that it is a basis of $L_p(\Q^d)$. We refer for a complete proof to \cite{W97} or \cite{LT79} (for a very nice one-dimensional proof) though we will get this result as a special case in the next chapter.

\end{rem}

\begin{df} \label{def_bas_coeff}
The system \eqref{dyadic_haarbasis} is called a dyadic Haar basis. The sequence $(\mu_{jm}(f))$ is called the sequence of dyadic Haar coefficients of $f$.
\end{df}

Normalized in $L_2(\Q^d)$ the functions $h_{jml}$ for arbitrary $b \geq 2$ are an orthonormal basis as well as we will see in the next chapter.

For technical reasons we will give an additional definition of $b$-adic Haar bases on $\R^d$. They will not be needed very much throughout this work but they find an application in the lemmas before Theorem \ref{thm_besov_char} stating the characterization of the Besov spaces with dominating mixed smoothness. Even there they are not necessary but make the understanding of the proofs easier.
\begin{df}
\mbox{}
\begin{enumerate}[(i)]
	\item For $j \in N_0, m \in \Z$ we call
	\[ I_{jm} = \big[ b^{-j} m, b^{-j} (m+1) \big) \]
	$b$-adic interval in $\R$. We define additionally $I_{-1,m}$ for $m \in \Z$ and $d$-dimensional $b$-adic intervals in $\R$ according to the definition above. Also the children $I_{jm}^k$ of $I_{jm}$ are defined according to the definition above.
	\item For $j \in N_{-1}, m \in \Z, \, l \in \B_j$ we define the function $h_{jml}$ as a function with support in $I_{jm}$ and constant values (according to above) in $I_{jm}^k$. For $j \in N_{-1}^d, m \in \Z^d, \, l \in \B_j$ the function $h_{jml}$ is defined as tensor product according to above.
\end{enumerate}
\end{df}

\begin{thm}

The system of dyadic Haar functions $h_{jm}, \, j \in \N_{-1}^d, \, m \in \Z^d$ is an orthogonal basis of $L_2(\R^d)$, an unconditional basis of $L_p(\R^d)$ for $1 < p < \infty$ and a conditional basis of $L_1(\R^d)$.

\end{thm}

We refer again to \cite{W97} and the references given there.

\subsection{Walsh bases}
Let $b \geq 2$ be an integer.
\begin{df}
\mbox{}
\begin{enumerate}[(i)]
	\item For $\alpha \in \N$ with $b$-adic expansion $\alpha = \alpha_0 + \alpha_1 b + \ldots + \alpha_{h - 1} b^{h - 1}$ with digits $\alpha_0, \alpha_1, \ldots, \alpha_{h - 1} \in \{ 0, 1, \ldots, b - 1 \}$ such that, $\alpha_{h - 1} \neq 0$, the Niederreiter-Rosen-bloom-Tsfasman (NRT) weight is given by $\varrho(\alpha) = h$. Furthermore, $\varrho(0) = 0$.
	\item The number of non-zero digits $\alpha_{\nu}, \, 0 \leq \nu < \varrho(\alpha)$ is the Hamming weight $\varkappa(\alpha)$.
	\item For $\alpha = (\alpha^1, \ldots, \alpha^d) \in \N_0^d$, the NRT weight is given by
\[ \varrho^d(\alpha) = \sum_{i = 1}^d \varrho(\alpha^i) \]
  and the Hamming weight by
  \[ \varkappa^d(\alpha) = \sum_{i = 1}^d \varkappa(\alpha^i). \]
\end{enumerate}
\end{df}

\begin{rem}

Clearly, $\varrho(\alpha) = 0$ if and only if $\alpha = 0$. Also the triangle inequality is easy to verify. Hence, $\varrho$ defines a norm on $\N_0$.

\end{rem}

\begin{df}
\mbox{}
\begin{enumerate}[(i)]
	\item For $\alpha \in \N_0$ with $b$-adic expansion $\alpha = \alpha_0 + \alpha_1 b + \ldots + \alpha_{\varrho(\alpha) - 1} b^{\varrho(\alpha) - 1}$ the $\alpha$-th $b$-adic Walsh function $\wal_{\alpha}: \, \Q \rightarrow \Cx$ is given by
	\[ \wal_{\alpha}(x) = \e^{\frac{2 \pi \im}{b} \left( \alpha_0 x_1 + \alpha_1 x_2 + \ldots + \alpha_{\varrho(\alpha) - 1} x_{\varrho(\alpha)} \right)}, \]
	for $x \in \Q$ with $b$-adic expansion  $x = x_1 b^{-1} + x_2 b^{-2} + \ldots$.
	\item The functions $\wal_{\alpha}, \, \alpha \in \N_0$ are called $b$-adic Walsh system on $\Q$.
	\item For $\alpha = (\alpha^1, \ldots, \alpha^d) \in \N_0^d$ the $b$-adic Walsh function $\wal_{\alpha}$ on $\Q^d$ is given as the tensor product $\wal_{\alpha}(x) = \wal_{\alpha^1}(x^1) \ldots \wal_{\alpha^d}(x^d)$ for $x = (x^1, \ldots, x^d) \in \Q^d$.
	\item The functions $\wal_{\alpha}, \, \alpha \in \N_0^d$ are called $b$-adic Walsh system on $\Q^d$.
\end{enumerate}
\end{df}

The following well known results can be found for instance in \cite[Appendix A]{DP10}.

\begin{prp} \label{prp_constant_walsh}

Let $\alpha \in \N_0$. Then $\wal_{\alpha}$ is constant on $b$-adic intervals $I_{\varrho(\alpha),m}$ for any $m \in \Dd_{\varrho(\alpha)}$. Further, $\wal_0$ is the characteristic function of $\Q$.

\end{prp}

\begin{proof}
Let $x \in I_{\varrho(\alpha),m}$. Hence its $b$-adic expansion can be written as
\[ x = m b^{-\varrho(\alpha)} + x_{\varrho(\alpha) + 1} b^{-\varrho(\alpha) - 1} + \ldots \]
where
\[ m = m_1 + m_2 b + \ldots + m_{\varrho(\alpha)} b^{\varrho(\alpha) - 1}. \]
Then
\[ \wal_{\alpha}(x) = \e^{\frac{2 \pi \im}{b} \left( \alpha_0 m_{\varrho(\alpha)} + \ldots + \alpha_{\varrho(\alpha) - 1} m_1 \right)} = \wal_{\alpha}(m b^{-\varrho(\alpha)}). \]
\end{proof}

\begin{prp}

We have for $\alpha \in \N_0$
\[ \int_{\Q} \wal_{\alpha}(x) \dint x = \begin{cases} 1 & \text{ if } \alpha = 0, \\ 0 & \text{ if } \alpha \neq 0. \end{cases} \]

\end{prp}

\begin{prp}

Let $\alpha, \beta \in \N_0^d$. Then we have
\[ \int_{\Q^d} \wal_{\alpha}(x) \overline{\wal_{\beta}(x)} \dint x = \begin{cases} 1 & \text{ if } \alpha = \beta, \\ 0 & \text{ if } \alpha \neq \beta. \end{cases} \]

\end{prp}

\begin{thm}

The system
\begin{align} \label{walshbasis}
\left\{ \wal_{\alpha} \, : \, \alpha \in \N_0^d \right\}
\end{align}
is an orthonormal basis of $L_2(\Q^d)$.

\end{thm}

\begin{rem}

The system \eqref{walshbasis} will be called a $b$-adic Walsh basis. Without going into details we mention that Walsh functions are characters on the Cantor group. We refer to the monograph \cite{SWS90} for much more information on Walsh functions.

\end{rem}

%% file: DigitalNets.tex
\section{Digital nets}
The idea of $(v,n,d)$-nets is the central property of uniform distribution that all intervals of the same order have to contain an approximately proportional number of points of a set. To achieve that goal we choose a large class of intervals and make sure that a constructed point set is distributed in a way that all the intervals from the chosen class contain the right number of points. We are eminently interested in so called digital $(v,n,d)$-nets since we are going to work with constructions by Chen and Skriganov.

For a finite point set in $\Q^d$ we can always find subsets of $\Q^d$ that do not contain a proportional number of points. For example we can even always find an interval that contains no points at all.
\begin{df}
For an integer $N$ and a class $J$ of subsets of $\Q^d$ we call a point set $\P$ in $\Q^d$ with $N$ points fair (with respect to $J$) if
\[ \frac{\# (I \cap \P)}{N} = |I| \]
for all $I \in J$.
\end{df}
It is desirable to consider a class as large as possible. We are going to work with the class of $b$-adic intervals. Then we can define the nets.
\begin{df}
For a given dimension $d \geq 1$, an integer $b \geq 2$, a positive integer $n$ and an integer $v$ with $0 \leq v \leq n$, a point set $\P$ in $\Q^d$ with $b^n$ points is called a $(v,n,d)$-net in base $b$ if the point set $\P$ is fair with respect to the class of all $b$-adic intervals in $\Q^d$ of order $n - v$. The number $v$ is called quality parameter of the $(v,n,d)$-net. A $(v,n,d)$-net in base $b$ is called strict for $v = 0$ or for $v \geq 1$ if it is not a $(v - 1,n,d)$-net in base $b$. 
\end{df}

\begin{rem} \label{rem_numberininterval}

The property for a $(v,n,d)$-net $\P$ in base $b$ means that every $b$-adic interval in $\Q^d$ of volume $b^{-n + v}$ contains exactly $b^v$ points of $\P$.

\end{rem}

Every $b$-adic interval of order $k$ for $k \geq 0$ is the union of $b$ disjoint $b$-adic intervals of order $k + 1$. Every $(v,n,d)$-net in base $b$ with $v \leq n - 1$ is also a $(v + 1,n,d)$-net in base $b$. Every point set of $b^n$ points in $\Q^d$ is an $(n,n,d)$-net in base $b$. The condition is then trivial. The following results can be found in \cite[Chapter 4]{DP10}.

\begin{lem}

For $1 \leq i \leq r$ let $\P_i$ be $(v_i,n_i,d)$-nets in base $b$ with $n_1, \ldots, n_r$ such that $b^{n_1} + \ldots + b^{n_r} = b^n$ for some integer $n$. Then the point set $\P = \P_1 \cup \ldots \cup \P_r$ is a $(v,n,d)$-net in base $b$ with
\[ v = n - \min_{1 \leq i \leq r} (n_i - v_i). \]

\end{lem}

\begin{proof}
Let $I$ be some $b$-adic interval in $\Q^d$ of order $n - v$. For every $1 \leq i \leq r$, $I$ contains exactly $b^{-n + n_i + v}$ points of $\P_i$. Therefore, $I$ contains exactly $b^v$ points of $\P$ and $\P$ is a $(v,n,d)$-net in base $b$.

\end{proof}

\begin{lem}

Let $\P$ be a $(v,n,d)$-net in base $b$. Let $1 \leq \tilde{d} \leq d$. We put
\[ \tilde{\P} = \left\{ (x_1, \ldots x_{\tilde{d}}) \, : \, (x_1, \ldots, x_d) \in \P \right\}. \]
Then $\tilde{\P}$ is a $(v,n,\tilde{d})$-net in base $b$.

\end{lem}

\begin{proof}
Let $\tilde{I}$ be some $b$-adic interval in $\Q^{\tilde{d}}$ of order $n - v$. Then $I = \tilde{I} \times \Q^{d - \tilde{d}}$ is a $b$-adic interval in $\Q^{d}$ of order $n - v$. Therefore, $I$ contains exactly $b^v$ points of $\P$. If we now fix the first $\tilde{d}$ coordinates of the points of $\P$ then exactly $b^v$ points of $\tilde{\P}$ are contained in $\tilde{I}$. Hence, $\P$ is a $(v,n,\tilde{d})$-net in base $b$.

\end{proof}

\cite[Corollary 4.19]{DP10} also gives us an existence rule for nets which will be important for our purposes later.

\begin{lem} \label{non_existence_for_nets}

A $(0,n,d)$-net in base $b$ cannot exist if $n \geq 2$ and $d \geq b + 2$.

\end{lem}

We mention just briefly that there is also a concept of so called $(v,d)$-sequences and $(V,d)$-sequences which is closely connected to $(v,n,d)$-nets. A sequence $(x_1, x_2, \ldots)$ in $\Q^d$ is called a $(v,d)$-sequence in base $b$ if for all integers $n \geq v$ and $k \geq 0$, the point set consisting of the points $x_{kb^n}, \ldots, x_{kb^n + b^n - 1}$ is a $(v,n,d)$-net in base $b$. The $(V,d)$-sequences are a more general concept.

Such sequences have a very ordered structure. For an integer $N \geq 1$ with $b$-adic expansion $N = N_0 + N_1 b + \ldots + a_n b^n$ the point set $\{ x_1, \ldots, x_N \}$ consisting of the first $N$ points of a $(v,d)$-sequence in base $b$ is the union of $N_n$ of $(v,n,d)$-nets in base $b$, $N_{n-1}$ of $(v,n-1,d)$-nets in base $b$, $\ldots$, $N_{v+1}$ of $(v,v+1,d)$-nets in base $b$ and $N_0 + N_1 b + \ldots + N_v b^v$ points without a special structure.

Additionally, every $(v,d)$-sequence is uniformly distributed.

Now we come to the subject which is the main goal of this section. Though $(v,n,d)$-nets have nice properties we did not give a method so far to construct them. And here digital nets come into play. For the rest of the section the base $b$ will be a prime. The construction of digital nets is clearer that way because there exists a finite field of order $b$ and it can be identified with $\Z_b$. But there are also digital nets in non-prime bases and for prime power bases the construction works in the same way. We describe the digital method to construct digital nets.

Let $n \in \N_0$. Let $C_1, \ldots, C_d$ be $n \times n$ matrices with entries from $\Fb_b$. We generate the net point $x_r = (x_r^1, \ldots, x_r^d)$ with $0 \leq r < b^n$. We expand $r$ in base $b$ as
\[ r = r_0 + r_1 b + \ldots + r_{n-1}b^{n-1} \]
with digits $r_k \in \{ 0, 1, \ldots, b-1 \}$, $1 \leq k \leq n-1$. We put $\bar{r} = (r_0, \ldots, r_{n-1})^{\top} \in \Fb_b^n$ and $\bar{h}_{r,i} = C_i \, \bar{r} = (h_{r,i,1}, \ldots, h_{r,i,n})^{\top} \in \Fb_b$, $1 \leq i \leq d$. Then we get $x_r^i$ as
\[ x_r^i = \frac{h_{r,i,1}}{b} + \ldots + \frac{h_{r,i,n}}{b^n}. \]
\begin{df}
A point set $\left\{ x_0, \ldots x_{b^n-1} \right\}$ constructed with the digital method is called a digital $(v,n,d)$-net in base $b$ with generating matrices $C_1, \ldots C_d$ if it is a $(v,n,d)$-net in base $b$.
\end{df}
The definition makes sense because, as we found out before, every point set of $b^n$ points is at least an $(n,n,d)$-net in base $b$. So the question is only what is the connection between the quality parameter $v$ of the digital $(v,n,d)$-net and the generating matrices.
\begin{df}
Let $b$ be a prime power and $C_1, \ldots, C_d$ be $n \times n$ matrices with entries from $\Fb_b$. Let $\varrho(C_1, \ldots, C_d)$ be the largest integer such that for any choice of $\gamma_1, \ldots, \gamma_d \in \N_0$ with $\gamma_1 + \ldots + \gamma_d = \varrho(C_1, \ldots, C_d)$, we have that the first $\gamma_1$ row vectors of $C_1$ together with the first $\gamma_2$ row vectors of $C_2$ together with $\ldots$ together with the first $\gamma_d$ row vectors of $C_d$ (i.e. $\varrho(C_1, \ldots, C_d)$ vectors), are linearly independent. We call $\varrho(C_1, \ldots, C_d)$ the linear independence parameter.
\end{df}
Now we can quote the result connecting the quality parameter with the generating matrices from \cite{DP10}.

\begin{prp}

Let $b$ be a prime power and $C_1, \ldots, C_d$ be $n \times n$ matrices with entries from $\Fb_b$. The point set constructed with the digital method using the matrices $C_1, \ldots, C_d$ is a strict $(n - \varrho(C_1, \ldots, C_d),n,d)$-net in base $b$.

\end{prp}

Now we quote again from \cite{DP10} - a result establishing a connection between digital nets and Walsh functions. It will have a significant importance later.
\begin{df} \label{df_d_prime}
Let $b$ be a prime. For a digital net with generating matrices $C_1, \ldots, C_d$ over $\Fb_b$, we call the matrix $C = (C_1^{\top} | \ldots | C_d^{\top}) \in \Fb_b^{n \times dn}$ the overall generating matrix of the digital net. The corresponding dual net is
\[ \Dn(C_1, \ldots, C_d) = \left\{ t \in \{ 0, \ldots, b^n - 1 \}^d \, : \, C_1^{\top} \, \bar{t}_1 + \ldots + C_d^{\top} \, \bar{t}_d = 0 \right\} \]
where $t = (t_1, \ldots, t_d)$ and for $1 \leq i \leq d$ we denote by $\bar{t}_i$ the $n$-dimensional column vectors of $b$-adic digits of $t_i$. We also put
\[ \Dn'(C_1, \ldots, C_d) = \Dn(C_1, \ldots, C_d) \backslash \{ 0 \}. \]
\end{df}

\begin{lem} \label{lem_duality_into_disc}

Let $b$ be a prime and let $\{ x_0, \ldots, x_{b^n - 1} \}$ be a digital $(v,n,d)$-net in base $b$ generated by the matrices $C_1, \ldots, C_d$. Then for $t \in \{ 0, \ldots, b^n - 1 \}^d$, we have
\[ \sum_{h = 0}^{b^n - 1} \wal_t(x_h) = \begin{cases} b^n & \text{ if } t \in \Dn(C_1, \ldots, C_d), \\ 0 & \text{ otherwise}. \end{cases} \]

\end{lem}

\begin{proof}
Since $\wal_t$ is a character, we have
\[ \sum_{h = 0}^{b^n - 1} \wal_t(x_h) = \begin{cases} b^n & \text{ if } \wal_t(x_h) = 1 \text{ for all } 0 \leq h < b^n, \\ 0 & \text{ otherwise}. \end{cases}. \]
We have $\wal_t(x_h) = 1$ for all $0 \leq h < b^n$ if and only if
\[ \sum_{i = 1}^d \bar{t}_i \, \bar{x}_h^i = 0 \]
for all $0 \leq h < b^n$. By definition of the digital nets we have $\bar{x}_h^i = C_i \, \bar{h}$. Hence, we have $\wal_t(x_h) = 1$ for all $0 \leq h < b^n$ if and only if
\[ \sum_{i = 1}^d \bar{t}_i \cdot C_i \, \bar{h} = 0 \]
for all $0 \leq h < b^n$ which is equivalent to
\[ C_1^{\top} \, \bar{t}_1 + \ldots + C_d^{\top} \, \bar{t}_d = 0. \]
\end{proof}

We mention the concept of digital $(v,d)$- and $(V,d)$-sequences just briefly again. Instead of $n \times n$ matrices one uses $\N \times \N$ matrices. Instead of $n$-dimensional vectors one uses sequences. Every digital sequence is a $(v,d)$-sequence. More information on this topic can be found in \cite[Chapter 4]{DP10} and the references given there.

%% file: DualityTheory.tex
\section{Duality theory}
In this section we deal with the simplification of the constructions of digital $(v,n,d)$-nets. Instead of constructing such directly one constructs certain $\Fb_b$-linear subspaces of $\Fb_b^{dn}$. We mainly quote from \cite[Chapter 7]{DP10}. We start with some definitions. Let $b$ be a prime. By the standard inner product in $\Fb_b^{dn}$ we mean
\[ A \cdot B = \sum_{i, j} a_{i j} b_{i j} \]
for $A = (a_{i j})_{i j}, B = (b_{i j})_{i j} \in \Fb_b^{dn}$.
\begin{df}
Let $\C$ be some $\Fb_b$-linear subspace of $\Fb_b^{dn}$. Then the dual space $\C^{\perp}$ relative to the standard inner product in $\Fb_b^{dn}$ is
\[ \C^{\perp} = \left\{ A \in \Fb_b^{dn}: \, B \cdot A = 0 \text{ for all } B \in \C \right\}. \]
\end{df}

\begin{rem}

We have $\dim(\C^{\perp}) = dn - \dim(\C)$ and $(\C^{\perp})^{\perp} = \C$.

\end{rem}

Recall that we have defined NRT and Hamming weights. We now give dual versions.
\begin{df}
\mbox{}
\begin{enumerate}[(i)]
  \item For $a = (a_1, \ldots, a_n) \in \Fb_b^n$ let
\[ v_n(a) = \begin{cases} 0 & \text{ if } a = 0, \\ \max \left\{ \nu: \, a_{\nu} \neq 0 \right\} & \text{ if } a \neq 0. \end{cases} \]
  \item Let $\varkappa(a)$ be the number of indices $1 \leq \nu \leq n$ such that, $a_{\nu} \neq 0$.
  \item For $A = (a_1, \ldots, a_d) \in \Fb_b^{dn}$ with $a_i \in \Fb_b^n$ for $1 \leq i \leq d$ let
\[ v_n^d(A) = \sum_{i = 1}^d v_n(a_i) \text{ and } \varkappa_n^d(A) = \sum_{i = 1}^d \varkappa_n(a_i). \]
\end{enumerate}
We call $v_n$ and $v_n^d$ NRT weight, $\varkappa_n$ and $\varkappa_n^d$ Hamming weight.
\end{df}
\begin{df}
Let $\C \neq \left\{ (0, \ldots, 0) \right\}$ be an $\Fb_b$-linear subspace of $\Fb_b^{dn}$. 
\begin{enumerate}[(i)]
	\item The minimum distance of $\C$ is given by
\[ \delta_n(\C) = \min \left\{ v_n^d(A): \, A \in \C \textbackslash \left\{ (0, \ldots, 0) \right\} \right\}. \]
  Furthermore, let $\delta_n(\left\{ (0, \ldots, 0) \right\}) = dn + 1$.
  \item The Hamming weight of $\C$ is
\[ \varkappa_n(\C) = \min\left\{ \varkappa_n(A): \, A \in \C \textbackslash \left\{ (0, \ldots, 0) \right\} \right\}. \]
\end{enumerate}
\end{df}

\begin{prp}

For any $\Fb_b$-linear subspace $\C$ of $\Fb_b^{dn}$ we have
\[ 1 \leq \delta_n(\C) \leq dn - \dim(\C) + 1. \]

\end{prp}

This fact is part of \cite[Proposition 7.3]{DP10}. Our goal is to transfer the subspaces $\C$ into point sets in $\Q^d$. To do so we need the following tool.
\begin{df} \label{df_mapping_Phi}
Let the mapping $\Phi_n^d: \, \Fb_b^{dn} \rightarrow \Q^d$ be given in the following way. For $a = (\alpha_1, \ldots, \alpha_n) \in \Fb_b^n$, let
\[ \Phi_n(a) = \frac{\alpha_1}{b} + \ldots + \frac{\alpha_n}{b^n} \]
and for $A = (a_1, \ldots, a_d) \in F_b^{dn}$, let
\[ \Phi_n^d(A) = \left( \Phi_n(a_1), \ldots, \Phi_n(a_d) \right). \]
\end{df}

The following result is \cite[Theorem 7.14]{DP10} and is our important duality tool.

\begin{prp} \label{prp_dig_net_duality}

Let $n, d \in \N, \, n \geq 2$. Let $\C$ and $\C^{\perp}$ be mutually dual $\Fb_b$-linear subspaces of $\Fb_b^{dn}$ of dimensions $n$ and $nd - n$, respectively. Then $\Phi_n^d(\C)$ is a digital $(v,n,d)$-net in base $b$ if and only if $\delta_n(\C^{\perp}) \geq n - v + 1$.

\end{prp}

\begin{rem}

The point set $\Phi_n^d(\C)$ is always at least an $(n,n,d)$-net in base $b$. So we can call $\Phi_n^d(\C)$ the corresponding digital $(v,n,d)$-net in base $b$.

\end{rem}

%% file: HaarBases.tex
\section{The $b$-adic Haar basis}
We give a $b$-adic generalization of Theorem \ref{thm_dyadic_haar_basis}.

\begin{thm}

The system
\begin{align} \label{haarbasis}
\left\{ b^{\frac{|j|}{2}} h_{jml} \, : \, j \in \N_{-1}^d, \, m \in \Dd_j, \, l \in \B_j \right\}
\end{align}
is an orthonormal basis of $L_2(\Q^d)$, an unconditional basis of $L_p(\Q^d)$ for $1 < p < \infty$ and a conditional basis of $L_1(\Q^d)$. For any function $f \in L_2(\Q^d)$ we have
\begin{align} \label{parsevals_equation}
\left\| f | L_2 (\Q^d) \right\|^2 = \sum_{j \in \N_{-1}^d} b^{|j|} \sum_{m \in \Dd_j, \, l \in \B_j} |\mu_{jml}|^2
\end{align}
where
\begin{align}
\mu_{jml} = \mu_{jml}(f) = \int_{\Q^d} f(x) h_{jml}(x) \, \dint x
\end{align}

\end{thm}

\begin{proof}
We start by proving that the system \eqref{haarbasis} is a Schauder basis of $L_p(\Q^d)$ for $1 \leq p < \infty$. The orthonormality is trivial, therefore, we will have proved that the system \eqref{haarbasis} is a conditional basis of $L_1(\Q^d)$ (since every basis in $L_1(\Q^d)$ is conditional, see \cite[Theorem II.D.10]{W91}) and an orthonormal basis of $L_2(\Q^d)$. In a second step we prove the unconditionality of the basis for $p > 1$. The formula \eqref{parsevals_equation} is Parseval's equation.

Let $1 \leq p < \infty$ and $f \in L_p(\Q^d)$. We denote by $s_n f$ the partial sum of the Haar series of $f$
\[ s_n f = \sum_{j_1, \ldots, j_d = -1}^n b^{\frac{j_{\eta_1} + \ldots + j_{\eta_s}}{2}} \sum_{m \in \Dd_j, \, l \in \B_j} \mu_{jml} \, h_{jml}. \]
We denote $\bar{n} = (n, \ldots, n)$. The function $s_n f$ is constant on all intervals $I_{\bar{n} m}$ for $m \in \Dd_n$ and one proves inductively that for every $n \in \N_0$ and any $m \in \Dd_{\bar{n}}$ we have
\[ s_n f(x) = b^{dn} \int_{I_{\bar{n} m}} f(y) \dint y \]
for all $x \in I_{\bar{n} m}$. Now we assume that $f$ is a continuous function. For every $\varepsilon > 0$ we can find an $n_0(\varepsilon)$ such that, for all $x, y \in I_{\bar{n} m}$ for any $m \in \Dd_{\bar{n}}$ we have
\[ |f(x) - f(y)| < \varepsilon, \]
and therefore,
\[ \left| f(x) - s_n f(x) \right| \leq b^{dn} \int_{I_{\bar{n} m}} |f(x) - f(y)| \dint y < \varepsilon \]
for all $n > n_0(\varepsilon)$. Hence,
\[ \left\| f - s_n f | L_{\infty}(\Q^d) \right\| < \varepsilon. \]
This means that the linear span of the Haar functions is dense in the space of continuous functions on $\Q^d$ with respect to the $\sup$-norm which in turn is dense in $L_p(\Q^d)$ which gives us $L_p(\Q^d)$-convergence of $s_n f$ to $f$. Therefore, we have completeness. Hölder's inequality gives us additionally
\[ \left\| s_n f | L_p (\Q^d) \right\| \leq \left\| f | L_p (\Q^d) \right\| \]
since
\begin{align*}
\int_0^1 |s_n f(x)|^p \dint x & = \sum_{m \in \Dd_{\bar{n}}} \int_{I_{\bar{n} m}} |s_n f(x)|^p \dint x \\
                              & = \sum_{m \in \Dd_{\bar{n}}} \int_{I_{\bar{n} m}} \dint x \, b^{dnp} \left| \int_{I_{\bar{n} m}} f(y) \dint y \right|^p \\
                              & \leq \sum_{m \in \Dd_{\bar{n}}} b^{dn (p - 1)} b^{-dn (p - 1)} \int_{I_{\bar{n} m}} |f(y)|^p \dint y \\
                              & = \int_0^1 |f(y)|^p \dint y.
\end{align*}
Now, let $p > 1$. The unconditionality of the basis follows for the case $p > 2$ from the results in the next section and Corollary \ref{cor_emb_BF} since in this case we have
\[ S_{p 2}^0 B(\Q^d) \hookrightarrow S_{p 2}^0 F(\Q^d) \]
and by \eqref{norm_SpqrF_Lp} we have
\[ S_{p 2}^0 F(\Q^d) = L_p(\Q^d). \]
Therefore, we get unconditionality from the unconditionality in $S_{p 2}^0 B(\Q^d)$ which we prove in the next section. In the case $1 < p < 2$ the unconditionality follows from duality.

\end{proof}

\begin{df} \label{def_bas_coeff}
The system \eqref{haarbasis} is called a $b$-adic Haar basis. The sequence $(\mu_{jml}(f))$ is called the sequence of $b$-adic Haar coefficients of $f$.
\end{df}
Analogously one proves the following result.

\begin{thm}

The system of $b$-adic Haar functions $h_{jml}, \, j \in \N_{-1}^d, \, m \in \Z^d, \, l \in \B_j$ is an orthogonal basis of $L_2(\R^d)$, an unconditional basis of $L_p(\R^d)$ for $1 < p < \infty$ and a conditional basis of $L_1(\R^d)$.

\end{thm}

The interested reader is referred to \cite{RW98} for much more information on $b$-adic wavelets, especially $b$-adic Haar functions.

%% file: Characterizations.tex
\section{Equivalent norms for $S_{pq}^r B (\Q^d)$}
The definition of the spaces $S_{pq}^r B (\Q^d)$ and $S_{pq}^r F (\Q^d)$ is not applicable for practical problems. Since we are going to calculate the norms of the discrepancy function, we need some equivalent norms. In \cite[Theorem 2.41]{T10a} Triebel gave such norms for the Besov spaces with dominating mixed smoothness for $d = 2$ using dyadic (i.e. $b = 2$) Haar bases. We generalize this theorem for arbitrary dimension and arbitrary base $b$.  We will get results for the spaces $S_{pq}^r F (\Q^d)$ using the embedding results given by Corollary \ref{cor_emb_BF}.

The definition of the spaces $S_{pq}^r B (\Q^d)$ was dyadic making it difficult to gain any $b$-adic results. Hence, we have to change the base first.

Let $\varphi_0 \in \mathcal{S}(\R)$ satisfy $\varphi_0(t) = 1$ for $|t| \leq 1$ and $\varphi_0(t) = 0$ for $|t| > \frac{b + 1}{b}$. Let
\[ \varphi_k(t) = \varphi_0(b^{-k} t) - \varphi_0(b^{-k + 1} t) \]
where $t \in \R, \, k \in \N$ and
\[ \varphi_k(t) = \varphi_{k_1}(t_1) \ldots \varphi_{k_d}(t_d) \]
where $k = (k_1,\ldots,k_d) \in \N_0^d, \, t = (t_1,\ldots,t_d) \in \R^d$.
The functions $\varphi_k$ are a $b$-adic resolution of unity since
\[ \sum_{k \in \N_0^d} \varphi_k(x) = 1 \]
for all $x \in \R^d$. The functions $\mathcal{F}^{-1}(\varphi_k \mathcal{F} f)$ are entire analytic functions for any $f \in \mathcal{S}'(\R^d)$.
\begin{df} \label{b_adic_SB}
Let $0 < p,q \leq \infty$ and $r \in \R$. Let $(\varphi_k)$ be a $b$-adic resolution of unity. The $b$-adic Besov space with dominating mixed smoothness $S_{pq}^r B^b(\R^d)$ consists of all $f \in \S'(\R^d)$ with finite quasi-norm
\[ \left\| f | S_{pq}^r B^b(\R^d) \right\| = \left( \sum_{k \in \N_0^d} b^{r |k| q} \left\| \mathcal{F}^{-1}(\varphi_k \mathcal{F} f) | L_p(\R^d) \right\|^q \right)^{\frac{1}{q}} \]
with the usual modification if $q = \infty$.
\end{df}

\begin{lem} \label{schm_triebel_1987}

Let $0 < p \leq \infty$ and $l > \frac{1}{\min(1,p)} - \frac{1}{2}$. Then there exists a constant $c > 0$ such that, for every $M \in S_2^l W(\R^d)$, all positive $\beta_1, \ldots, \beta_d$ and every $f \in L_p(\R^d)$ for which $\F f$ has compact support in $[-\beta_1, \beta_1] \times \ldots \times [-\beta_d, \beta_d]$, we have
\[ \left\| \F^{-1} (M \F f) | L_p(\R^d) \right\| \leq c \left\| M(\beta_1 \cdot, \ldots, \beta_d \cdot) | S_2^l W(\R^d) \right\| \left\| f | L_p(\R^d) \right\|. \]

\end{lem}

This fact is \cite[Proposition 2.3.3]{Hn10}.

\begin{prp}

Let $0 < p,q \leq \infty$ and $r \in \R$. Then for all $f \in \S'(\R^d)$ we have that $f \in S_{pq}^r B(\R^d)$ if and only if $f \in S_{pq}^r B^b(\R^d)$ and the norms $\left\| \cdot | S_{pq}^r B(\R^d) \right\|$ and $\left\| \cdot | S_{pq}^r B^b(\R^d) \right\|$ are equivalent on $S_{pq}^r B(\R^d)$.

\end{prp}

\begin{proof}
We will prove the following fact from which the proposition can be concluded. The spaces $S_{pq}^r B^b(\R^d)$ and $S_{pq}^r B^{b+1}(\R^d)$ are equal and their norms are equivalent.

Let the functions $\varphi_k$ be a $b$-adic one-dimensional resolution of unity and the functions $\psi_k$ a $(b+1)$-adic one-dimensional resolution of unity. We observe that 
\[ \supp \varphi_k \subset [-b^{k+1}, -b^{k-1}] \cup [b^{k-1}, b^{k+1}] \]
and
\[ \supp \psi_k \subset [-(b+1)^{k+1}, -(b+1)^{k-1}] \cup [(b+1)^{k-1}, (b+1)^{k+1}]. \]
Now we check that for every $j \in \N_0$ there are at most $2$ such $k \in \N_0$ that
\[ [-b^{k+1}, -b^{k-1}] \cup [b^{k-1}, b^{k+1}] \subset [-(b+1)^{k+1}, -(b+1)^{k-1}] \cup [(b+1)^{k-1}, (b+1)^{k+1}]. \]
It is sufficient to check $[b^{k-1}, b^{k+1}] \subset [(b+1)^{k-1}, (b+1)^{k+1}]$ (because of the symmetry). But this is easy since $(b+1)^{j-1} \leq b^{k-1}$ and $b^{k+1} \leq (b+1)^{j+1}$ is equivalent to
\begin{align} \label{k_j_limited}
(j - 1) \frac{\log(b + 1)}{\log(b)} + 1 \leq k \leq (j + 1) \frac{\log(b + 1)}{\log(b)} - 1.
\end{align}
The fact that the cardinality of the set of such $k$ is at most $2$ follows from
\[ 2 \frac{\log(b + 1)}{\log(b)} - 2 < 2 \]
which is equivalent to
\[ \frac{\log(b + 1)}{\log(b)} < 2 \]
which is equivalent to $0 < b^2 - b - 1$ which is clearly satisfied since $b \geq 2$. Therefore, we know that for every $j$ there are no more than two $k$ such that,
\[ \supp \varphi_k \subset \supp \psi_j. \]
For every $j \in \N_0$ we denote by $\Lambda(j)$ the set of such $k$ that
\[ \supp \varphi_k \cap \supp \psi_j \neq \emptyset. \]
Suppose that $\supp \varphi_k \subset \supp \psi_j$ and $\supp \varphi_{k+1} \subset \supp \psi_j$ then (for $k \geq 2$)
\[ \Lambda(j) = \{ k - 2, k - 1, k, k + 1, k + 2, k + 3 \}, \]
and therefore, the cardinality of such sets is at most $6$ but for sure they are not empty.
Conversely, for every $k \in \N_0$, there are at most $3$ such $j \in \N_{-1}$ that
\[ \supp \varphi_k \cap \supp \psi_j \neq \emptyset \]
and they are of the form $j - 1, j, j + 1$ (if $j \geq 1$). The cases $j = 1$ or $k < 2$ are to be modified just slightly. We denote by $\Omega(k)$ the set of such $j$. Additionally, we put for $j \in \N_0^d$
\[ \Lambda(j) = \Lambda(j_1) \times \ldots \times \Lambda(j_d) \]
and for $k \in \N_0^d$
\[ \Omega(k) = \Omega(k_1) \times \ldots \times \Omega(k_d). \]
Hence, for all $x \in \R^d$ we have
\[ \varphi_k(x) = \varphi_k(x) \sum_{j \in \Omega(k)} \psi_j(x) \]
and
\[ \psi_j(x) = \psi_j(x) \sum_{k \in \Lambda(j)} \varphi_k(x). \]
Now let $j,k \in \N_0^d$ then we have
\begin{align*}
\mathcal{F}^{-1}(\varphi_k \mathcal{F} f) = \sum_{j \in \Omega(k)} \mathcal{F}^{-1} \left( \varphi_k \mathcal{F} \left( \mathcal{F}^{-1} (\psi_{j} \mathcal{F} f) \right) \right)
\end{align*}
and
\begin{align*}
\mathcal{F}^{-1}(\psi_j \mathcal{F} f) = \sum_{k \in \Lambda(j)} \mathcal{F}^{-1} \left( \psi_j \mathcal{F} \left( \mathcal{F}^{-1} (\varphi_{k} \mathcal{F} f) \right) \right).
\end{align*}
Let $l > \frac{1}{\min(1,p)} - \frac{1}{2}$. From Lemma \ref{schm_triebel_1987} for $M = \varphi_k$ and $\beta_1 = b^{k_1 + 2}, \ldots, \beta_d = b^{k_d + 2}$ we get (with a constant $c > 0$) that
\begin{align*}
& \left\| \mathcal{F}^{-1} \left( \varphi_k \mathcal{F} \left( \mathcal{F}^{-1} (\psi_{j} \mathcal{F} f) \right) \right) | L_p (\R^d) \right\| \\
& \qquad \qquad \leq c \left\| \varphi_k(b^{k_1 + 2} \cdot, \ldots, b^{k_d + 2} \cdot) | S_2^l W(\R^d) \right\| \left\| \mathcal{F}^{-1} (\psi_{j} \mathcal{F} f) | L_p(\R^d) \right\| \\
& \qquad \qquad \leq c_1 \prod_{i=1}^d \left\| \varphi_{k_i}(b^{k_i + 2} \cdot) | W_2^l(\R) \right\| \left\| \mathcal{F}^{-1} (\psi_{j} \mathcal{F} f) | L_p(\R^d) \right\|.
\end{align*}
Since $\varphi_{k_i} \in \mathcal{S}(\R)$ there exists a constant $c_2 > 0$ such that for all $i$ we have
\[ \left\| \varphi_{k_i}(b^{k_i + 2} \cdot) | W_2^l(\R) \right\| \leq c_2. \]
Consequently, we get
\[ \left\| \mathcal{F}^{-1} \left( \varphi_k \mathcal{F} \left( \mathcal{F}^{-1} (\psi_{j} \mathcal{F} f) \right) \right) | L_p (\R^d) \right\| \leq c_3 \left\| \mathcal{F}^{-1} (\psi_{j} \mathcal{F} f) | L_p(\R^d) \right\| \]
for $j \in \Omega(k)$ and analogously (using Lemma \ref{schm_triebel_1987} for $M = \psi_j$ and $\beta_1 = (b + 1)^{j_1 + 2}, \ldots, \beta_d = (b + 1)^{j_d + 2}$) we get
\[ \left\| \mathcal{F}^{-1} \left( \psi_j \mathcal{F} \left( \mathcal{F}^{-1} (\varphi_{k} \mathcal{F} f) \right) \right) | L_p (\R^d) \right\| \leq c_4 \left\| \mathcal{F}^{-1} (\varphi_{k} \mathcal{F} f) | L_p(\R^d) \right\| \]
for $k \in \Lambda(j)$. So we have proved for every $k \in \N_0^d$ that
\[ \left\| \mathcal{F}^{-1} \left( \varphi_k \mathcal{F} f \right) | L_p (\R^d) \right\| \leq c \sum_{j \in \Omega(k)} \left\| \mathcal{F}^{-1} (\psi_{j} \mathcal{F} f) | L_p(\R^d) \right\|. \]
Multiplying with $b^{r |k| q}$ and summing over $k$ will obviously give us on the left side $\left\| f | S_{pq}^r B^b(\R^d) \right\|$. On the right side we get at most $3$ identical summands which we can incorporate into the constant. The norming factor can be easily estimated with a constant since the difference of $j$ and $k$ is limited by \eqref{k_j_limited}. Conversely, we have for every $j \in \N_0^d$
\[ \left\| \mathcal{F}^{-1} \left( \psi_j \mathcal{F} f \right) | L_p (\R^d) \right\| \leq c \sum_{k \in \Lambda(j)} \left\| \mathcal{F}^{-1} (\varphi_{k} \mathcal{F} f) | L_p(\R^d) \right\|. \]
Multiplying with $(b+1)^{r |j| q}$ and summing over $j$ will obviously give us on the left side $\left\| f | S_{pq}^r B^{b+1}(\R^d) \right\|$. On the right side we get at most $6$ identical summands which we can incorporate into the constant. The same applies again to the norming factor.

\end{proof}

\begin{rem}

Analogously one could define $b$-adic Triebel-Lizorkin spaces with dominating mixed smoothness and prove the equivalence of dyadic and $b$-adic norms.

\end{rem}

From now on we will omit $b$ in $S_{pq}^r B^b(\R^d)$ and write $S_{pq}^r B(\R^d)$ instead. Having proved the equivalence of the norms for all bases $b$ puts us in the position to generalize all the results from \cite{T10a}. The proofs can be rewritten, replacing $2$ by $b$. We will give the results with a minimum of comments.

We denote by $\chi_{jm}$ the characteristic function of the $b$-adic interval $I_{jm}$ for $j \in \N_0, \, m \in \Z$. For $j = (j_1, \ldots, j_d) \in \N_0^d, \, m = (m_1, \ldots, m_d) \in \Z^d$ we put $\chi_{jm}(x) = \chi_{j_1 m_1}(x_1) \cdot \ldots \cdot \chi_{j_d m_d}(x_d)$ for $x = (x_1, \ldots, x_d) \in \R^d$.

\begin{lem} \label{counterp_tr_234}

Let $0 < p, q \leq \infty$ and $\max(\frac{1}{p},1) - 1 < r < \min(\frac{1}{p},1)$. Let the sequence $(\mu_{jm})$ satisfy
\[ \left( \sum_{j \in \N_0^d} b^{|j|(r - \frac{1}{p} + 1) q} \left( \sum_{m \in \Z^d} | \mu_{jm}|^p \right)^{\frac{q}{p}} \right)^{\frac{1}{q}} < \infty. \]
Then
\begin{align} \label{f_given_with_chi}
f = \sum_{j \in \N_0^d} b^{|j|} \sum_{m \in \Z^d} \mu_{jm} \, \chi_{jm}
\end{align}
belongs to $S_{pq}^r B(\R^d)$ and there is a constant $c > 0$ independent of the sequence $(\mu_{jm})$ such that,
\begin{align} \label{eq_norm_half}
\left\| f | S_{pq}^r B(\R^d) \right\| \leq c \left( \sum_{j \in \N_0^d} b^{|j|(r - \frac{1}{p} + 1) q} \left( \sum_{m \in \Z^d} | \mu_{jm}|^p \right)^{\frac{q}{p}} \right)^{\frac{1}{q}}.
\end{align}

\end{lem}

\begin{proof}
This result is a counterpart of \cite[Proposition 2.34]{T10a}. In order to prove it we will follow closely Triebel's proof. Let $(\mu_{jm})$ be a sequence satisfying 
\[ \left( \sum_{j \in \N_0^d} \left( \sum_{m \in \Z^d} | \mu_{jm}|^p \right)^{\frac{q}{p}} \right)^{\frac{1}{q}} < \infty \]
and let $f$ be given by
\begin{align} \label{space_char_chi}
f = \sum_{j \in \N_0^d} \sum_{m \in \Z^d} \mu_{jm} \, b^{-|j|(r - \frac{1}{p})} \, \chi_{jm}.
\end{align}
We prove that $f \in S_{pq}^r B(\R^d)$ and
\[ \left\| f | S_{pq}^r B(\R^d) \right\| \leq c \left( \sum_{j \in \N_0^d} \left( \sum_{m \in \Z^d} | \mu_{jm}|^p \right)^{\frac{q}{p}} \right)^{\frac{1}{q}} \]
which is an equivalent formulation of the lemma making it easier to follow the proof of \cite[Proposition 2.34]{T10a}. One should keep in mind that we are dealing with a $b$-adic case though it works in the same way. Let $\psi_M, \psi_F$ be real compactly supported $L_2$-normed $b$-adic Daubechies wavelets on $\R$ analogous to \cite[(1.55--1.56)]{T10a}. Their existence is guaranteed by \cite[Theorem 5.1]{RW98}. We will not define wavelets here. For basic and advanced facts on dyadic wavelets we refer to \cite{W97}, on $b$-adic wavelets to \cite{RW98}. We just state here that they give an orthonormal basis. We now expand $\chi_{j_1 m_1}(x_1), \ldots, \chi_{j_d m_d}(x_d)$ into the wavelet representation according to \cite[(2.51--2.53)]{T10a} and obtain for $1 \leq i \leq d$
\begin{multline*}
\chi_{j_i m_i}(x_i) = \sum_{l_i \in \Z} \lambda_{l_i}^{0,F}(\chi_{j_i m_i}(x_i)) \psi_F(x_i - l_i) + \\
+ \sum_{k_i = 0} \sum_{l_i \in \Z} \lambda_{l_i}^{k_i,M}(\chi_{j_i m_i}(x_i)) \psi_M(2^{k_i} x_i - l_i)
\end{multline*}
with
\[ \lambda_{l_i}^{0,F}(\chi_{j_i m_i}(x_i)) = \int_{\R} \chi_{j_i m_i}(x_i) \psi_F (y_i - l_i) \dint y_i, \]
and
\[ \lambda_{l_i}^{k_i,M}(\chi_{j_i m_i}(x_i)) = b^{k_i} \int_{\R} \chi_{j_i m_i}(x_i) \psi_M (b^{k_i} y_i - l_i) \dint y_i. \]
Then we insert $\chi_{jm}(x) = \chi_{j_1 m_1}(x_1) \cdot \ldots \cdot \chi_{j_d m_d}(x_d)$ into \eqref{space_char_chi}. We split the resulting expansions as in \cite[(2.56--2.60)]{T10a}. Now we have $2^d$ terms sorted into the cases $(j_1 \geq k_1, \ldots, j_d \geq k_d), \ldots, (j_1 < k_1, \ldots, j_d < k_d)$. We get a $b$-adic version of \cite[(2.54)]{T10a} and \cite[(2.55)]{T10a}. This guarantees counterparts of \cite[(2.62--2.66)]{T10a} and \cite[(2.73--2.74)]{T10a}. This observation leads to the norm estimate of the lemma, and therefore, proves it.

\end{proof}
\pagebreak

\begin{lem} \label{counterp_tr_237}

Let $0 < p, q \leq \infty$ and $\max(\frac{1}{p},1) - 1 < r < \min(\frac{1}{p},1)$. Then there exists a constant $c > 0$ such that,
\[ \left\| f | S_{pq}^r B(\R^d) \right\| \geq c \left( \sum_{j \in \N_{-1}^d} b^{|j|(r - \frac{1}{p} + 1) q} \left( \sum_{m \in \Z^d, \, l \in \B_j} | \mu_{jml} (f)|^p \right)^{\frac{q}{p}} \right)^{\frac{1}{q}} \]
for all $f \in S_{pq}^r B(\R^d)$. The sequence $(\mu_{jml}(f))$ of the $b$-adic Haar coefficients is given by
\[ \mu_{jml}(f) = \int_{\R^d} f(x) h_{jml}(x) \dint x. \]

\end{lem}

\begin{proof}
This result is a counterpart of \cite[Proposition 2.37]{T10a} and the proof is straightforward applicable because the generalization of \cite[Theorem 1.52]{T10a} to our case is straightforward and we use it with $A = 0$ and $B = 1$.

\end{proof}

\begin{prp} \label{counterp_tr_238}

Let $0 < p, q \leq \infty$ and $q > 1$ if $p = \infty$. Let $\frac{1}{p} - 1 < r < \min(\frac{1}{p},1)$. Let $f \in \S'(\R^d)$. Then $f\in S_{pq}^r B(\R^d)$ if and only if it can be represented as
\begin{align} \label{repres_f_lem}
f = \sum_{j \in \N_{-1}^d} b^{|j|} \sum_{m \in \Z^d, \, l \in \B_j} \mu_{jml} \, h_{jml}
\end{align}
for some sequence $(\mu_{jml})$ satisfying
\begin{align} \label{eq_quasinorm_lem}
\left( \sum_{j \in \N_{-1}^d} b^{|j|(r - \frac{1}{p} + 1) q} \left( \sum_{m \in \Z^d, \, l \in \B_j} | \mu_{jml}|^p \right)^{\frac{q}{p}} \right)^{\frac{1}{q}} < \infty.
\end{align}
The convergence of \eqref{repres_f_lem} is unconditional in $\S'(\R^d)$ and in any $S_{pq}^\rho B(\R^d)$ with $\rho < r$. The representation \eqref{repres_f_lem} of $f$ is unique with the $b$-adic Haar coefficients
\[ \mu_{jml} = \int_{\R^d} f(x) h_{jml}(x) \dint x. \]
The expression \eqref{eq_quasinorm_lem} is an equivalent quasi-norm on $S_{pq}^r B(\R^d)$.

\end{prp}

\begin{proof}
This result is a counterpart of \cite[Theorem 2.38]{T10a} and again we follow closely Triebel's proof. Let $0 < p, q \leq \infty$ and $\max(\frac{1}{p},1) - 1 < r < \min(\frac{1}{p},1)$. Let $f$ be given by \eqref{repres_f_lem}. We represent the $b$-adic Haar functions with characteristic functions. Let $j \in \N_0, m \in \Z, l \in \B_j$. Then
\begin{align*}
   h_{jml} & = \sum_{k = 0}^{b - 1} \e^{\frac{2\pi \im}{b} k l} \chi_{j+1,bm+k}, \\
h_{-1,0,1} & = \chi_{0,0}.
\end{align*}
So, $f$ can be given in the form \eqref{f_given_with_chi}. Therefore, by Lemma \ref{counterp_tr_234} we have $f \in S_{pq}^r B(\R^d)$ and \eqref{eq_norm_half} holds.

Conversely, if $f \in S_{pq}^r B(\R^d)$, then we have Lemma \ref{counterp_tr_237}. The representability of $f$ as in \eqref{repres_f_lem} follows from the fact that the $b$-adic Haar system is an orthogonal basis in $L_2(\R^d)$. Therefore, one obtains the equivalence of the norms. All further technicalities can be found in the proof of \cite[Theorem 2.9]{T10a} and the references given there. The unconditionality is clear in view of \eqref{eq_quasinorm_lem}

The assertion can be obtained for $1 < p, q \leq \infty$ with $\frac{1}{p} - 1 < r < 0$ as explained in Step 2 of the proof of \cite[Proposition 2.38]{T10a} using duality of the spaces (Proposition \ref{dom_mix_smoo_duality}). It is also explained there how to prove the generalization of the duality. \cite[Theorem 1.20]{T10a} is here helpful as well.

The remaining cases with $q < \infty$ can be obtained by real interpolation as explained in Step 3 of the proof of \cite[Proposition 2.38]{T10a}. One finds the necessary references there. We will not define it here. Instead we just state that the point is that by interpolation it suffices to prove the assertion for the spaces $S_{p q_0}^{r_0} B(\R^d)$ with $1 < p < \infty$, $0 < q_0 < \infty$, $0 < r_0 < \frac{1}{p}$ and $S_{p q_1}^{r_1} B(\R^d)$ with $1 < q_1 < \infty$, $\frac{1}{p} - 1 < r_0 < 0$ to obtain the assertion for any space $S_{p q}^{r} B(\R^d)$ with $r = (1 - \theta) r_0 + \theta r_1$, $\frac{1}{q} = \frac{1 - \theta}{q_0} + \frac{\theta}{q_1}$, where $0 < \theta < 1$. But the spaces $S_{p q_0}^{r_0} B(\R^d)$ and $S_{p q_1}^{r_1} B(\R^d)$ are already covered.

All other cases $1 < p < \infty, \frac{1}{p} - 1 < r \leq 0, q = \infty$ can be solved by duality again.

\end{proof}

We are now ready to state the main result which we will use later for the discrepancy function. It is the counterpart of \cite[Theorem 2.41]{T10a}.

\begin{thm} \label{thm_besov_char}

Let $0 < p, q \leq \infty$ and $q > 1$ if $p = \infty$. Let $\frac{1}{p} - 1 < r < \min(\frac{1}{p},1)$. Let $f \in \D'(\Q^d)$. Then $f\in S_{pq}^r B(\Q^d)$ if and only if it can be represented as
\begin{align} \label{repres_f}
f = \sum_{j \in \N_{-1}^d} b^{|j|} \sum_{m \in \Dd_j, \, l \in \B_j} \mu_{jml} \, h_{jml}
\end{align}
for some sequence $(\mu_{jml})$ satisfying
\begin{align} \label{eq_quasinorm}
\left( \sum_{j \in \N_{-1}^d} b^{|j|(r - \frac{1}{p} + 1) q} \left( \sum_{m \in \Dd_j, \, l \in \B_j} | \mu_{jml}|^p \right)^{\frac{q}{p}} \right)^{\frac{1}{q}} < \infty.
\end{align}
The convergence of \eqref{repres_f} is unconditional in $\D'(\Q^d)$ and in any $S_{pq}^\rho B(\Q^d)$ with $\rho < r$. The representation \eqref{repres_f} of $f$ is unique with the $b$-adic Haar coefficients
\[ \mu_{jml} = \int_{\Q^d} f(x) h_{jml}(x) \dint x. \]
The expression \eqref{eq_quasinorm} is an equivalent quasi-norm on $S_{pq}^r B(\Q^d)$.

\end{thm}

\begin{proof}
Once again we follow Triebel's proof. First we restrict ourselves to $\max(\frac{1}{p},1) - 1 < r < \min(\frac{1}{p},1)$ and put
\[ \tilde{S}_{pq}^r B(\Q^d) = \left\{ f \in S_{pq}^r B(\R^d): \, \supp f \subset [0,1]^d \right\}. \]
Let $f \in \tilde{S}_{pq}^r B(\Q^d)$ then by Proposition \ref{counterp_tr_238} we get the representation \eqref{repres_f} (with $\Dd_j$ instead of $\Z^d$). The spaces $\tilde{S}_{pq}^r B(\Q^d)$ can be identified with the spaces $S_{pq}^r B(\Q^d)$. Conversely, let $f$ be given by \eqref{repres_f} then again Proposition \ref{counterp_tr_238} can be used and we get $f \in S_{pq}^r B(\R^d)$ while of course $\supp f \subset [0,1]^d$. Unconditionality follows from the Proposition \ref{counterp_tr_238}. Further technicalities are explained in the proof of \cite[Proposition 2.41]{T10a} and the references given there. All other cases are solved by duality.

\end{proof}

\begin{rem}

There is no necessity to go through $\R^d$ as we did it here since we do not need this case for later calculations. Instead one could have considered only $\Q^d$ right away. We did it for completeness.

\end{rem}

%% file: L2Discrepancy.tex
We now come to concrete results on irregularities of point distribution, giving new results as well as historical results illustrating the development of the theory. We deal with general lower bounds, point sets with best possible discrepancy or just pure existence assertions of those and with concrete constants of the bounds. In this chapter our topic is $L_p$-discrepancy which also includes star discrepancy and can be considered as the starting point and most practically applicable area of research in context of discrepancy.

\section{$L_2$-discrepancy}
Practically all results on $L_p$-discrepancy are based on one sole idea by Klaus Roth. In $L_2$ the idea is based on orthogonality arguments, in $L_p$ Littlewood-Paley can be applied to replace orthogonality. But not only $L_p$-discrepancy is based on Roth's work. In the next chapter we will introduce discrepancy in spaces with dominating mixed smoothness and also there the similarity in the methods will be obvious. Even upper bounds are connected to Roth's method.

\subsection{The lower bounds}
The first and the last result on asymptotical lower bounds for the $L_2$-discrepancy was given by Roth in 1954 (\cite{R54}). It was the last one because it was the best possible and in the same paper Roth also was the first one to state this problem in the plane or a higher dimension. This paper can be regarded as the starting point of the modern theory of discrepancy. He referenced van Aardenne-Ehrenfest's result from 1949 concerning the distribution of sequences, improving it significantly. The actual proof was for the plane but in a remark he explained a possible generalization to arbitrary dimension. A recent paper (\cite{B11}) by Bilyk deals mostly with Roth's result and surveys it and its implications in much detail. In this subsection we will give a slightly modified version of Roth's proof using $b$-adic Haar bases. As a result we will obtain the best constant in Roth's lower bound known so far. This result can also be found in \cite{HM11}. For a positive real number $x$ we denote by $\left\lceil x \right\rceil$ the smallest integer that is greater than $x$. We need some easy calculations for the result.

\begin{lem} \label{lem_haar_coeff_besov_x_1}

Let $f(x) = x_1 \cdot \ldots \cdot x_d$ for $x=(x_1,\ldots,x_d) \in \Q^d$. Let $j \in \N_0^d, \, m \in \Dd_j, \, l \in \B_j$ and let $\mu_{jml}$ be the $b$-adic Haar coefficient of $f$. Then
\[ \mu_{jml} = \frac{b^{-2|j| - d}}{(\e^{\frac{2 \pi \im}{b} l_1} - 1) \cdot \ldots \cdot (\e^{\frac{2 \pi \im}{b} l_d} - 1)}. \]

\end{lem}
One easily checks the one-dimensional case and concludes with tensor products. The next result is again easily derived from the one-dimensional case.

\begin{lem} \label{lem_haar_coeff_besov_indicator_1}

Let $z = (z_1,\ldots,z_d) \in \Q^d$ and $g(x) = \chi_{[0,x)}(z)$ for $x = (x_1, \ldots, x_d) \in \Q^d$. Let $j \in \N_0^d, \, m \in \Dd_j, l \in \B_j$ and let $\mu_{jml}$ be the $b$-adic Haar coefficient of $g$. Then $\mu_{jml} = 0$ whenever $z$ is not contained in the interior of the $b$-adic interval $I_{jm}$ supporting the function $h_{jml}$. 

\end{lem}
The following result is an analytical masterpiece. We give a proof because of its beauty and because we cannot give a reference.

\begin{lem} \label{lem_eulerstyle1}

For any natural $b \geq 2$ we have
\[ \sum_{l = 1}^{b - 1} \cot^2 \frac{l \pi}{2b} = \frac{(2b - 1)(b - 1)}{3}. \]

\end{lem}

\begin{proof}
For $l = 1, \ldots, b - 1$ we have
\[ (-1)^l = (\e^{\frac{l \pi \im}{2b}})^{2b} = \left( \cos \frac{l \pi}{2b} + \im \sin \frac{l \pi}{2b} \right)^{2b} = \sum_{k = 0}^{2b} \binom{2b}{k} \left( \cos \frac{l \pi}{2b} \right)^k \left( \im \sin \frac{l \pi}{2b} \right)^{2b - k}. \]
We consider only the imaginary part
\[ 0 = \sum_{r = 0}^{b - 1} (-1)^{b - r + 1} \binom{2b}{2r + 1} \left( \cos \frac{l \pi}{2b} \right)^{2r + 1} \left( \sin \frac{l \pi}{2b} \right)^{2b - 2r - 1} \]
and, after dividing by $(\cos \frac{l \pi}{2b}) (\sin \frac{l \pi}{2b})^{2b - 1}$, we get
\[ 0 = \sum_{r = 0}^{b - 1} (-1)^{b - r + 1}  \binom{2b}{2r + 1} \left( \cot \frac{l \pi}{2b} \right)^{2r}. \]
So, for $l = 1, \ldots, b - 1$ the pairwise distinct terms $\cot^2 \frac{l \pi}{2b}$ are roots of the polynomial
\[ p(x) = \sum_{r = 0}^{b - 1} (-1)^{b - r + 1}  \binom{2b}{2r + 1} x^r. \]
Since $p$ has degree $b - 1$, these are all the roots and they are simple. From Vieta's formulas we get
\[ \sum_{l = 1}^{b - 1} \cot^2 \frac{l \pi}{2b} = \frac{\binom{2b}{2b - 3}}{\binom{2b}{2b - 1}} = \frac{(2b - 1)(b - 1)}{3}. \]
\end{proof}

\begin{lem} \label{lem_eulerstyle2}

For any natural $b \geq 2$ we have
\[ \sum_{l = 1}^{b - 1} \cot^2 \frac{l \pi}{2b - 1} = \frac{(2b - 3)(b - 1)}{3}. \]

\end{lem}

The proof is analogous to Lemma \ref{lem_eulerstyle1}.

\begin{prp} \label{prp_diagonals}

For any natural $b \geq 2$ we have
\[ \sum_{l = 1}^{b - 1} \frac{1}{\left| \e^{\frac{2 \pi \im}{b} l} - 1 \right|^2} = \frac{b^2 - 1}{12}. \]

\end{prp}

\begin{proof}
From Lemmas \ref{lem_eulerstyle1} and \ref{lem_eulerstyle2} and from the fact that
\[ \cot^2 x = \frac{1}{\sin^2 x} - 1 \]
we get for any natural $b \geq 2$ that
\[ \sum_{l = 1}^{b - 1} \frac{1}{\sin^2 \frac{l \pi}{2b}} = \frac{2(b^2 - 1)}{3} \]
and
\[ \sum_{l = 1}^{b - 1} \frac{1}{\sin^2 \frac{l \pi}{2b - 1}} = \frac{2b(b - 1)}{3}. \]
For $l = 1, \ldots, b - 1$ we have $0 < \frac{l \pi}{2b} < \frac{\pi}{2}$ and $0 < \frac{l \pi}{2b-1} < \frac{\pi}{2}$. For $l = b + 1, \ldots, 2b - 1$ we have $\frac{\pi}{2} < \frac{l \pi}{2b} < \pi$ and for $l = b, \ldots, 2b - 2$ we have $\frac{\pi}{2} < \frac{l \pi}{2b - 1} < \pi$. Using this and the symmetry of the sine function we get
\[ \sum_{l = 1}^{2b - 1} \frac{1}{\sin^2 \frac{l \pi}{2b}} = 2 \sum_{l = 1}^{b - 1} \frac{1}{\sin^2 \frac{l \pi}{2b}} + 1 = \frac{(2b)^2 - 1}{3} \]
and
\[ \sum_{l = 1}^{2b - 2} \frac{1}{\sin^2 \frac{l \pi}{2b - 1}} = 2 \sum_{l = 1}^{b - 1} \frac{1}{\sin^2 \frac{l \pi}{2b - 1}} = \frac{4b(b - 1)}{3} = \frac{(2b - 1)^2 - 1}{3}. \]
Hence, for any natural $b \geq 2$ we have
\[ \sum_{l = 1}^{b - 1} \frac{1}{\sin^2 \frac{l \pi}{b}} = \frac{b^2 - 1}{3}. \]
Then one gets
\begin{align*}
\sum_{l = 1}^{b - 1} \frac{1}{\left| \e^{\frac{2 \pi \im}{b} l} - 1 \right|^2} & = \frac{1}{2} \sum_{l = 1}^{b - 1} \frac{1}{1 - \cos \frac{2 \pi l}{b}} \\
& = \frac{1}{4} \sum_{l = 1}^{b - 1} \frac{1}{\sin^2 \frac{2 \pi l}{b}} \\
& = \frac{b^2 - 1}{12}.
\end{align*}
\end{proof}

\begin{rem}

The last proposition can be regarded as a property of diagonals of a regular polygon if we define the two sides in an edge to be the first and the $(b - 1)$-th diagonal. Then the term $\left| e^{\frac{2\pi i}{b} l} - 1 \right|$ is the length of the $l$-th diagonal.

\end{rem}

We are now ready to state and prove the celebrated theorem of Roth and calculate the best known constant.

\begin{thm} \label{thm_roth_const}

For any positive integer $N$ and all point sets $\P$ in $\Q^d$ with $N$ points the inequality
\[ \left\| D_{\P} | L_2(\Q^d) \right\| \geq c_d \, \frac{\left( \log N \right)^\frac{d-1}{2}}{N} \]
holds with
\[ c_d = {\frac{1}{\sqrt{21} \cdot 2^{2d-1} \, \sqrt{(d-1)!} \, (\log 2)^{\frac{d-1}{2}}}}. \]
  
\end{thm}

\begin{proof}
Let $N \in \N$ and let $\P$ be a point set in $\Q^d$ with $N$ points. Let $j \in \N_0^d, \, m \in \Dd_j$ be such that no point of $\P$ is contained in the interior of $I_{jm}$. Let $l \in \B_j$. Using Lemmas \ref{lem_haar_coeff_besov_x_1} and \ref{lem_haar_coeff_besov_indicator_1} one concludes that the $b$-adic Haar coefficient $\mu_{jml}$ of the discrepancy function can be given in such case as
\[ \mu_{jml} = -\frac{b^{-2|j| - d}}{(e^{\frac{2\pi i}{b} l_1} - 1) \cdot \ldots \cdot (e^{\frac{2\pi i}{b} l_d} - 1)}. \]
For fixed $j \in \N_0^d$ the cardinality of $\Dd_j$ is $b^{|j|}$. This implies that there are at least $b^{|j|} - N$ such $m \in \Dd_j$ for which no point of $\P$ lies in the interior of $I_{jm}$. We abbreviate $M = \left\lceil \log_b N \right\rceil$. We use Parseval's equation \eqref{parsevals_equation}, including only $j \in \N_0^d$, therefore, reducing the norm and use Proposition \ref{prp_diagonals}
\begin{align*}
\left\| D_{\P} | L_2(\Q^d) \right\|^2 & \geq \sum_{|j| \geq M} b^{|j|} \, (b^{|j|} - N) \, b^{-4|j| - 2d} \cdot \\
& \qquad \qquad \qquad \cdot \sum_{l \in \B_j} \left|\ \e^{\frac{2 \pi \im}{b} l_1} - 1 \right|^{-2} \ldots \left| \e^{\frac{2 \pi \im}{b} l_d} - 1 \right|^{-2} \\
& = \left( \frac{b^2 - 1}{12 b^2} \right)^d \sum_{|j| \geq M} b^{-2|j|} \left( 1 - N b^{-|j|} \right)
\end{align*}
Before we continue to estimate we have to insert some calculations. It is well known that for any positive integer $K$ the cardinality of the set
\[ \left\{ j \in \N_0^d \, : \, |j| = K \right\} \]
is
\[ \binom{K + d - 1}{d - 1} = \frac{(K + d - 1)!}{K! (d - 1)!}. \]
We also need that for any $q > 1$
\[ \sum_{K = M}^{\infty} q^{-K} = \frac{q^{-M + 1}}{q - 1}. \]
We continue to estimate keeping in mind that for any integer $K \geq M$ we have $0 < N b^{-M} \leq 1$, hence
\pagebreak
\begin{align*}
N^2 & \left\| D_{\P} | L_2(\Q^d) \right\|^2 \\
& \qquad \geq \left( \frac{b^2 - 1}{12 b^2} \right)^d \frac{1}{(d - 1)!} N^2 \sum_{K = M}^{\infty} b^{-2K} (1 - N b^{-K}) \frac{(K + d - 1)!}{K!} \\
& \qquad \geq \left( \frac{b^2 - 1}{12 b^2} \right)^d \frac{1}{(d - 1)!} N^2 \sum_{K = M}^{\infty} b^{-2K} (1 - N b^{-K}) K^{d-1} \\
& \qquad \geq M^{d - 1} \left( \frac{b^2 - 1}{12 b^2} \right)^d \frac{1}{(d - 1)!} N^2 \sum_{K = M}^{\infty} \left( b^{-2K} - N b^{-3K} \right) \\
& \qquad = M^{d - 1} \left( \frac{b^2 - 1}{12 b^2} \right)^d \frac{1}{(d - 1)!} N^2 \left( \frac{b^{-2M + 2}}{b^2 - 1} - N \frac{b^{-3M + 3}}{b^3 - 1} \right) \\
& \qquad = M^{d - 1} \left( \frac{b^2 - 1}{12 b^2} \right)^d \frac{b^3}{(d - 1)!} \left[ \frac{(N b^{-M})^2}{b(b^2 - 1)} - \frac{(N b^{-M})^3}{b^3 - 1} \right]
\end{align*}
Now let $t = M - \log_b N$ so that $0 \leq t < 1$ and $N b^{-M} = b^{-t}$. We put
\[ B = \left( \frac{b^2 - 1}{12 b^2} \right)^d \frac{b^3}{(d - 1)!}. \]
Then we have proved that
\[ N^2 \left\| D_{\P} | L_2(\Q^d) \right\|^2  \geq \gamma (\log_b N)^{d - 1} \]
for all $N \in \N$ if we can verify that
\[ M^{d - 1} \, B \left( \frac{b^{-2t}}{b(b^2 - 1)} - \frac{b^{-3t}}{b^3 - 1} \right) \geq \gamma (M - t)^{d - 1} \]
for all $M \in \N_0$ and $0 \leq t < 1$. The last inequality is equivalent to
\[ \gamma \left( M^{d - 1} - (M - t)^{d - 1} \right) \geq M^{d - 1} \left[ \gamma - B \left( \frac{b^{-2t}}{b(b^2 - 1)} - \frac{b^{-3t}}{b^3 - 1} \right) \right] \]
which is certainly satisfied  whenever $\gamma \geq 0$ and
\[ \gamma \leq B \left( \frac{b^{-2t}}{b(b^2 - 1)} - \frac{b^{-3t}}{b^3 - 1} \right) \]
for all $0 \leq t < 1$, since clearly $M^{d - 1} - (M - t)^{d - 1} \geq 0$ or, alternatively
\begin{align} \label{gamma_const_minmax}
\gamma \leq B \left( \frac{y^2}{b(b^2 - 1)} - \frac{y^3}{b^3 - 1} \right)
\end{align}
for all $b^{-1} < y \leq 1$. The minimal value of the right-hand side is easily seen to be
\begin{align*}
\gamma_b & = B \, \frac{1}{b \, (b+1)(b^3-1)} \\
         & = \frac{(b^2 - 1)^d}{2^{2d} \, 3^d \, b^{2d-2} \, (b+1) (b^3-1) (d - 1)!}
\end{align*}
for
\[ y = b^{-1} \text{ or } y = 1. \]
To get the constant, we have to find the optimal base $b$. We easily verify that
\[ c_d = \sqrt{\frac{\gamma_b}{(\log b)^{d - 1}}} = \frac{(b^2 - 1)^{\frac{d}{2}}}{2^d \, 3^{\frac{d}{2}} \, b^{d-1} \, \sqrt{(b+1)(b^3-1)(d - 1)!} (\log b)^{\frac{d - 1}{2}}} \]
is nonincreasing in $b$, therefore, the optimal constant is obtained for $b = 2$.

\end{proof}

\begin{rem}

The so far best constant for arbitrary dimension from \cite{DP10} is here improved by a factor of $\frac{32}{\sqrt{21}}$.

\end{rem}

\begin{rem} \label{rem_supinf}

We consider again \eqref{gamma_const_minmax} from the proof of the last theorem. The maximal value of the right-hand side is easily seen to be
\begin{align*}
\overline{\gamma_b} & = B \frac{4}{27} \frac{(b^2 + b + 1)^2}{(b-1)(b+1)^3 \, b^3} \\
                    & = \frac{(b^2 - 1)^{d - 1} (b^2 + b + 1)^2}{2^{2d-2}  \, 3^{d+3} \, b^{2d} \, (b+1)^2 (d - 1)!}
\end{align*}
for
\[ y = \frac{2}{3} \frac{b^2 + b + 1}{b(b + 1)}. \]
We put
\[ \overline{c_d} = \limsup_{N \rightarrow \infty} \frac{D^{L_2}(N)}{(\log N)^{\frac{d-1}{2}}}. \]
Analogously to above we get
\[ \overline{c_d} \geq \sqrt{\frac{\overline{\gamma_2}}{(\log 2)^{d - 1}}} = {\frac{7}{27 \cdot 2^{2d-1} \, \sqrt{(d-1)!} \, (\log 2)^{\frac{d-1}{2}}}}. \]

\end{rem}

\subsection{The upper bounds}
Theorem \ref{thm_roth_const} gave lower bounds for the $L_2$-discrepancy and these bounds are asymptotically best possible. This means that there exist point sets with asymptotical $L_2$-discrepancy of the same rate.

\begin{thm} \label{thm_upper_bound_CS}

There exists a constant $C_d > 0$ such that, for any positive integer $N$, there exists a point set $\P$ in $\Q^d$ with $N$ points such that
\[ \left\| D_{\P} | L_2(\Q^d) \right\| \leq C_d \, \frac{\left( \log N \right)^\frac{d-1}{2}}{N}. \]

\end{thm}

This result was first proved for $d = 2$ by Davenport in \cite{D56}. Davenport used the following constructions. Let $\theta$ be any irrational number having a continued fraction with bounded partial quotients and let $\{ \alpha \}$ denote the fractional part of any real number $\alpha$. For an even number $N$, the coordinates of the points can be given by $x_{\nu}^{\pm} = \{ \pm \nu \theta \}, \, y_{\nu} = \frac{2\nu}{N}$. Then the point set used by Davenport is
\[ \P = \left\{ (x_{\nu}^+,y_{\nu}), (x_{\nu}^-,y_{\nu}): \, \nu = 1, \ldots, \frac{N}{2} \right\}. \]
Davenport proved that these point sets satisfy the upper bounds of the theorem. He also speculated about a possible generalization to $d = 3$, though the conditions for such a generalization are equivalent to the falsity of Littlewood's conjecture, which is a famous open problem.

In \cite{R76} Roth gave an alternative proof for the case $d = 2$. He did not give explicitly a point set satisfying the upper bound but used probabilistic methods. Instead he proved in \cite{R76} that there must exist a permutation $n_0, n_1, \ldots, n_{N-1}$ of the numbers $0, 1, \ldots, N-1$ such that for
\[ x_j = \frac{n_j}{N}, \, y_j = \frac{j}{N} \]
where $j = 0, 1, \ldots, N-1$ and the point set $\P_N^* = \{ (x_j,y_j) : \, j = 0, 1 \ldots, N-1 \}$ can be shifted such that the shifted point set satisfies the upper bound. By a shifted point set $\P_N^*(t)$ where $t$ is some real number we mean that every point $(x,y)$ from $\P_N^*$ is shifted horizontally in $t \, (\modulo \, 1)$. If $(x,y) \in \P_N^*$, then $(\{ x + t \}, y) \in \P_N^*(t)$. Roth proved that there is a constant $c > 0$ such that
\[ N^2 \, \int_0^1 \left\| D_{\P_N^*(t)} | L_2(\Q^2) \right\|^2 \dint t \leq c  \, \log N. \]
Therefore, there must exist a real number $t$ such that
\[ N^2 \, \left\| D_{\P_N^*(t)} | L_2(\Q^2) \right\|^2 \leq c  \, \log N. \]
In \cite{R79} he realized that his proof could be simplified significantly, starting the translations with Hammersley type point sets and improved it to the $3$-rd dimension. In \cite{R80} Roth finally generalized the approach to arbitrary dimension.

Another alternative proof for the $2$-dimensional case was given by Halton and Zaremba in \cite{HZ69} by an alternative explicit construction.

The search for an explicitly given point set in arbitrary dimension satisfying the upper bound remained an open problem for a long time and was solved only in 2002 by Chen and Skriganov. They constructed the point set as a digital net and proved in \cite{CS02} the upper bounds. In this work we will analyze the discrepancy in function spaces with dominating mixed smoothness of point sets of Chen and Skriganov. Therefore, we will explain them in detail in the next chapter.

The best value known so far for the constant $C_2$ of the $2$-dimensional case of Theorem \ref{thm_upper_bound_CS} can be found in \cite{FPPS10} where generalized scrambled Hammersley type point sets were used. Hammersley type point sets will be explained in a later chapter of this work. The constant from \cite{FPPS10} is
\[ C_2 = \sqrt{\frac{278629}{2811072 \log 22}}. \]
The best constant in arbitrary dimension can be obtained via digital shifts and can be found in \cite[Section 16.6]{DP10}. Its value is given by
\[C_d = \frac{22^d}{\sqrt{(d-1)!} \, (\log 2)^{\frac{d-1}{2}}}. \]
So, for example in the case $d = 2$ the value is
\[ \frac{22^2}{2 \sqrt{\log 2}} \]
which is much worse than the constant from \cite{FPPS10}.

\subsection{Conclusion}
If we compare the constants from the lower and the upper bounds we realize that the $2$-dimensional case is not that bad anymore. The constant from the lower bound is
\[ c_2 = {\frac{1}{\sqrt{21} \cdot 8 \, \sqrt{\log 2}}} = 0.032763 \ldots, \]
the constant from the upper bound is
\[ C_2 = \sqrt{\frac{278629}{2811072 \log 22}} = 0.179070 \ldots, \]
so they differ only by a factor of around $5$.

We recall Remark \ref{rem_supinf} for the case $d = 2$. We have
\[ \overline{c_2} \geq {\frac{7}{216 \, \sqrt{\log 2}}} = 0.038925 \ldots, \]
which indicates a better constant. We would like to call attention to \cite{BTY12} where the authors made numerical experiments with $L_2$-discrepancy of Fibonacci sets and obtained a slightly better value for $C_2$ of ca. $0.176006$. Though they do not prove it, it is a hint that Fibonacci sets might have the best possible $L_2$-discrepancy.

In arbitrary dimension the constant of the upper bound is bad and the difference to the constant in the lower bound is huge. 

We recall the weighted discrepancy function as defined by \eqref{weighted_disc}. Thanks to Lemma \ref{lem_haar_coeff_besov_indicator_1}, the Haar coefficient with respect to a Haar function whose support does not intersect $\P$ does not depend on the weights. So one gets the same lower bound with the same constant for the weighted $L_2$-discrepancy as in the case without weights. Hence we have the following generalization of Theorem \ref{thm_roth_const} to the weighted discrepancy.

\begin{thm}

For any positive integer $N$, all point sets $\P$ in $\Q^d$ with $N$ points, and all weights $a = (a_z)_{z \in \P}$, the inequality
\[ \left\| D_{\P,a} | L_2(\Q^d) \right\| \geq c_d \, \frac{\left( \log N \right)^\frac{d-1}{2}}{N} \]
holds with
\[ c_d = {\frac{1}{\sqrt{21} \cdot 2^{2d-1} \, \sqrt{(d-1)!} \, (\log 2)^{\frac{d-1}{2}}}}. \]

\end{thm}

%% file: LpDiscrepancy.tex
\section{$L_p$-discrepancy for $1 < p < \infty$}
Some results for $L_p$-discrepancy can be transferred directly from $L_2$-discrepancy, thanks to the embeddings of the Lebesgues spaces. Other cases have to be adopted to the more difficult situation where we do not have orthogonality.

\subsection{The lower bounds}
Schmidt proved in \cite{S77} the following result.

\begin{thm} \label{thm_schmidt}

Let $1 < p < \infty$. Then there exists a constant $c_d > 0$ such that, for any positive integer $N$ and all point sets $\P$ in $\Q^d$ with $N$ points, we have
\[ \left\| D_{\P} | L_p(\Q^d) \right\| \geq c_d \, \frac{\left( \log N \right)^\frac{d-1}{2}}{N}. \]
  
\end{thm}

This result is nontrivial for $1 < p < 2$ (for $2 < p < \infty$ this follows from Theorem \ref{thm_roth_const} via embeddings). Schmidt's idea to substitute orthogonality can be improved and shortened using the well known Littlewood-Paley theory. We will quote the corresponding results from \cite{B11}. For $j \in \N_{-1}, \, m \in \Dd_j$ let $h_{jm} = h_{jml}$ be the dyadic Haar functions, i.e. Haar functions with $b = 2$. Then the function
\[ Sf(x) = \left( \sum_{j = -1}^{\infty} b^j \left( \sum_{m = 0}^{b^j - 1} \mu_{jm} h_{jm}(x) \right)^2 \right)^{\frac{1}{2}} \]
is called dyadic square function of $f$. The Littlewood-Paley inequalities then state that for $1 < p < \infty$ there exist constants $0 < c_{p,d} < C_{p,d}$ such that, for every function $f \in L_p(\Q)$, we have
\[ c_{p,d} \left\| Sf | L_p (\Q) \right\| \leq \left\| f | L_p (\Q) \right\| \leq C_{p,d} \left\| Sf | L_p (\Q) \right\|. \]
Without going into details, we just state that this approach applied coordinatewise similar to Roth's method delivers Schmidt's result (see \cite{B11} and the references given there).

\subsection{The upper bounds}
The lower bounds from Theorem \ref{thm_schmidt} are the best possible. This is clear for $1 < p < 2$, thanks to the embeddings of the Lebesgues spaces. Chen proved it for $2 < p < \infty$ in \cite{C80}. He remarked that the $2$-dimensional case could be easily deduced from \cite{R76}, which indeed is possible though "easily" might not be the right word. But it is not difficult. The proof changes where the function is being squared. One gets additional terms. Davenport's proof of the case $p = 2$ from \cite{D56} cannot deliver the general case, since Parseval's equation was used.

\begin{thm} \label{thm_upper_chen}

Let $1 < p < \infty$. Then there exists a constant $C_d > 0$ such that, for any positive integer $N$, there exists a point set $\P$ in $\Q^d$ with $N$ points such that
\[ \left\| D_{\P} | L_p(\Q^d) \right\| \leq C_d \, \frac{\left( \log N \right)^\frac{d-1}{2}}{N}. \]

\end{thm}

Chen uses similar methods as Roth in \cite{R76}, \cite{R79} and \cite{R80}, translating point sets $\modulo \, 1$ and calculating the expectation of the norm of the discrepancy function of such translations. Proving that the expectation satisfies the upper bounds shows that there is such a translation that satisfies the bounds.

In \cite{S06} Skriganov proved that the constructions from \cite{CS02} satisfy the upper bounds of Theorem \ref{thm_upper_chen}, therefore, they are explicitly given point sets with best possible $L_p$-discrepancy.

%% file: StarDiscrepancy.tex
\section{Star discrepancy}
In this section we are going to deal with the $L_{\infty}$-discrepancy which is usually called star discrepancy and denoted by
\[ D_{\P}^* = \left\| D_{\P} | L_{\infty}(\Q^d) \right\|. \]
It is often considered the most important case in the theory.

\subsection{The lower bounds}
Of course Roth's lower bound from Theorem \ref{thm_schmidt} is also true for the star discrepancy. The star discrepancy was what Roth actually had in mind when he worked on \cite{R54}. But as it turns out this bound is not the best possible. The following result dor $d = 2$ is known from \cite{S72} though Schmidt proved it for the equivalent problem of one-dimensional sequences.

\begin{thm} \label{thm_star_disc_schmidt}

There exists a constant $c > 0$ such that, for any positive integer $N$ and all point sets $\P$ in $\Q^2$ with $N$ points, we have
\[ D_{\P}^* \geq c \, \frac{\log N}{N}. \]

\end{thm}

We will see that this result is the best possible. But higher dimensional analogues do not exist so far. For a long time Roth's bound was the best known lower bound. Beck improved it for $d = 3$ in \cite{B89} proving that for any $\varepsilon > 0$ there exists a positive integer $N_0$ such that, for any point set $\P$ in $\Q^3$ with $N \geq N_0$ points, we have
\[ D_{\P}^* \geq \frac{\log N \, (\log \log N)^{\frac{1}{8} - \varepsilon}}{N}. \]
Bilyk and Lacey improved this result in \cite{BL08}. They proved that there exist constants $c > 0$ and $0 < \eta < \frac{1}{2}$ such that, for any positive integer $N$ and all point sets $\P$ in $\Q^3$ with $N$ points, we have
\[ D_{\P}^* \geq c \, \frac{(\log N)^{1 + \eta}}{N}. \]
Later they generalized it together with Vagharshakyan in \cite{BLV08} for arbitrary $d \geq 3$ which is the best known lower bound by now.

\begin{thm} \label{thm_bilyketal_star}

For any dimension $d \geq 3$ there exist constants $c_d > 0$ and $0 < \eta_d < \frac{1}{2}$ such that, for any positive integer $N$ and all point sets $\P$ in $\Q^d$ with $N$ points, we have
\[ D_{\P}^* \geq c_d \, \frac{(\log N)^{\frac{d - 1}{2} + \eta_d}}{N}. \]

\end{thm}

\subsection{The upper bounds}
Point sets with best possible star discrepancy in the plane are known for a long time though the early examples were given in the form of one-dimensional infinite sequences. Although van der Corput proved the upper bound in \cite{C35}, the general ideas go back to the beginning of the $20$-th century, to i.a. Ostrowski, Hardy, Littlewood and even Lerch. The generalization of van der Corput's point set to arbitrary dimension was proposed by Hammersley (\cite{Hm60}) and the bound was calculated by Halton (\cite{Hl60}).

\begin{thm} \label{thm_corput_hammersley_halton}

There exists a constant $C_d > 0$ such that, for any positive integer $N$, there exists a point set $\P$ in $\Q^d$ with $N$ points such that
\[ D_{\P}^* \leq C_d \, \frac{\left( \log N \right)^{d - 1}}{N}. \]

\end{thm}

As mentioned above the point sets satisfying this theorem are the Hammersley-Halton point sets (called van der Corput point sets in the $2$-dimensional case). We will use their slightly generalized $2$-dimensional version in the next chapter. For the definition of the point sets we define the bit reversal function for any prime $b$ as
\[r_b(i) = \frac{i_0}{b} + \frac{i_1}{b^2} + \ldots \]
where $i = 0, 1, \ldots N - 1$ is given in its $b$-adic expansion $i = i_0 + i_1 \, b + i_2 \, b^2 +\ldots$ (meaning that $i_0, i_1, i_2, \ldots \in \{ 0, 1, \ldots, b - 1$ \}). Then we choose $d - 1$ distinct primes $b_1, \ldots, b_{d-1}$. Then the point set consists of the points
\[ \left( \frac{i}{N}, r_{b_1}(i), \ldots, r_{b_{d-1}}(i) \right) \]
for $i = 0, \ldots, N - 1$. Van der Corput's version was for $b_1 = 2$.

\subsection{Conclusion}
In view of Theorem \ref{thm_star_disc_schmidt} and Theorem \ref{thm_corput_hammersley_halton} the case $d = 2$ is perfectly solved while for arbitrary dimension the gap in the exponent is still huge (compare Theorem \ref{thm_bilyketal_star} and Theorem \ref{thm_corput_hammersley_halton}). There are several conjectures about the best possible lower bound, the following three possibly being the most popular ones
\begin{align*}
& D_{\P}^* \geq c \, \frac{(\log N)^{\frac{d}{2}}}{N}, \\
& D_{\P}^* \geq c \, \frac{(\log N)^{d - 1}}{N}, \\
& D_{\P}^* \geq c \, \frac{(\log N)^{\frac{d - 1}{2} + \frac{d - 1}{d}}}{N}.
\end{align*}

%% file: L1Discrepancy.tex
\section{$L_1$-discrepancy}
This section deals with yet another unsatisfactorily solved case for the discrepancy. The lower bound is due to Halász.

\begin{thm} \label{thm_halasz}

There exists a constant $c_d > 0$ such that, for any positive integer $N$ and all point sets $\P$ in $\Q^d$ with $N$ points, we have
\[ \left\| D_{\P} | L_1(\Q^d) \right\| \geq c_d \, \frac{\sqrt{\log N}}{N}. \]

\end{thm}

This result is from \cite{H81}. Since the results in cases before depended on the dimension, one is not too surprised that this result is not believed to be the best possible for $d > 2$. It is conjectured by many in the field that the best lower bound is
\[ \left\| D_{\P} | L_1(\Q^d) \right\| \geq c \, \frac{\left( \log N \right)^\frac{d-1}{2}}{N} \]
which fits with the upper bound that can be deduced from Theorem \ref{thm_upper_bound_CS} using simple embeddings.

\begin{thm} \label{thm_L1_upper}

There exists a constant $C_d$ such that, for any positive integer $N$, there exists a point set $\P$ in $\Q^d$ with $N$ points such that
\[ \left\| D_{\P} | L_1(\Q^d) \right\| \leq C_d \, \frac{\left( \log N \right)^{\frac{d-1}{2}}}{N}. \]

\end{thm}

%% file: Conclusion.tex
\section{Conclusion}
We want to summarize the results of this chapter and present them in an easily understandable form as a table. We will give the bounds and the references. 

The content of this chapter can be abstracted in the following way. Let $1 \leq p \leq \infty$. There exist constants $c_{p,d}, C_{p,d}$ that depend only on $p$ and on the dimension $d$ and $\alpha, \beta$ such that, for any positive integer $N$, we have
\[ c_{p,d} \, \frac{\left( \log N \right)^{\alpha}}{N} \leq D^{L_p}(N) \leq C_{p,d} \, \frac{\left( \log N \right)^{\beta}}{N} \]
where $D^{L_p}(N)$ is the $L_p$-discrepancy as defined by Definition \ref{M_discr}
\[ D^{L_p}(N) = \inf_{\# \P = N} \left\| D_{\P} | L_p([0,1)^d) \right\|. \]

The exponents $\alpha$ and $\beta$ are shown in the following table sorted by $p$. There is an additional row for the lower and the upper bounds respectively, giving the references in historical order. In the case of the upper bounds we differentiate between proofs using probabilistic methods and explicit constructions. Cases that follow from a smaller or greater $p$ by simple embedding arguments are labeled by an arrow in the corresponding direction. The constant $0 < \eta_d < \frac{1}{2}$ depends only on the dimension.
\begin{table}[h]
\begin{center}
\begin{tabular}{|c|c|c|c|l|}
\toprule
& \large{$\alpha$} & & \large{$\beta$} & \bigstrut \\
\midrule
\multirow{2}{*}{$p = \infty$} & \multicolumn{1}{|l|}{$d = 2$: \quad 1}                          & \cite{S72}   & \multirow{2}{*}{\quad $d - 1$ \quad} & $d = 2$: \; \cite{C35} (expl.) \bigstrut \\
                              & \multicolumn{1}{|l|}{$d \geq 3$: \; $\frac{d - 1}{2} + \eta_d$} & \cite{BLV08} &                                & $d \geq 3$: \; \cite{Hl60} (expl.) \bigstrut \\
\midrule
\multirow{3}{*}{$2 < p < \infty$} & \multirow{3}{*}{\Large $\frac{d - 1}{2}$} & \multirow{3}{*}{\large $\uparrow$} & \multirow{3}{*}{\Large $\frac{d - 1}{2}$} & $d = 2$: \; \cite{R76} (prob.) \bigstrut \\
                                                                                                                                                         & & & & $d \geq 3$: \; \cite{C80} (prob.) \\
                                                                                                                                                         & & & & \cite{S06} (expl.) \bigstrut \\
\midrule
\multirow{4}{*}{$p = 2$} & \multirow{4}{*}{\Large $\frac{d - 1}{2}$} & \multirow{4}{*}{\cite{R54}} & \multirow{4}{*}{\Large $\frac{d - 1}{2}$} & $d = 2$: \; \cite{D56} (expl.) \bigstrut \\
                                                                                                                           & & & & $d = 3$: \; \cite{R79} (prob.) \\
                                                                                                                           & & & & $d \geq 4$: \; \cite{R80} (prob.) \\
                                                                                                                           & & & & \cite{CS02} (expl.) \bigstrut \\
\midrule
\multirow{3}{*}{$1 < p < 2$} & \multirow{3}{*}{\Large $\frac{d - 1}{2}$} & \multirow{3}{*}{\cite{S77}} & \multirow{3}{*}{\Large $\frac{d - 1}{2}$} & \\
                                                                                                                                             & & & & \multicolumn{1}{|c|}{$\downarrow$} \\
                                                                                                                                             & & & & \\
\midrule
\multirow{3}{*}{$p = 1$} & \multirow{3}{*}{\Large $\frac{1}{2}$} & \multirow{3}{*}{\cite{H81}} & \multirow{3}{*}{\Large $\frac{d - 1}{2}$} & \\
                                                                                                                                     & & & & \multicolumn{1}{|c|}{$\downarrow$} \\
                                                                                                                                     & & & & \\
\bottomrule
\end{tabular}
\end{center}
\end{table}

%% file: LowerUpperBoundsSpqrB.tex
Discrepancy in spaces with dominating mixed smoothness was first considered by Triebel (\cite{T10b}, \cite{T10a}). The main results of this work are upper bounds of the $S_{pq}^r B$-discrepancy.

\section{Lower bounds}
In \cite[Theorem 6.13]{T10a} one finds the following result.

\begin{thm} \label{lower_to_main_B}

Let $1 \leq p, q \leq \infty$ and $q < \infty$ if $p = 1$ and $q > 1$ if $p = \infty$. Let $\frac{1}{p} - 1 < r < \frac{1}{p}$. Then there exists a constant $c > 0$ such that, for any integer $N \geq 2$ and all point sets $\P$ in $\Q^d$ with $N$ points, we have
\[ \left\| D_{\P} | S_{pq}^r B(\Q^d) \right\| \geq c \, N^{r-1} \, \left( \log N \right)^{\frac{d-1}{q}}. \]
	
\end{thm}

This bound is best possible for $r \geq 0$ as can be seen in the next section. We want to take advantage of the embeddings given by Corollary \ref{cor_emb_BF}.

\begin{cor} \label{lower_to_main_F}

Let $1 \leq p < \infty$. Let $1 \leq q \leq \infty$ and $q < \infty$ if $p = 1$. Let $\frac{1}{\min(p,q)} - 1 < r < \frac{1}{p}$. Then there exists a constant $c > 0$ such that, for any integer $N \geq 2$ and all point sets $\P$ in $\Q^d$ with $N$ points, we have
\[ \left\| D_{\P} | S_{pq}^r F(\Q^d) \right\| \geq c \, N^{r-1} \, \left( \log N \right)^{\frac{d-1}{q}}. \]

\end{cor}

\begin{proof}
Let $q < \infty$. From Corollary \ref{cor_emb_BF} we have $S_{p q}^r F([0,1)^d) \hookrightarrow S_{\min(p,q),q}^r B([0,1)^d)$. Therefore, we get the assertion for $\frac{1}{\min(p,q)} - 1 < r < \frac{1}{\min(p,q)}$ from the last theorem if we can guarantee that $D_{\P}$ makes sense in $S_{p q}^r F([0,1)^d)$. By \cite[Proposition 6.3]{T10a} this is only the case for $r < \frac{1}{p}$.

From the first part of Proposition \ref{embeddings_SpqrBF} we have $S_{p, \infty}^r F(\Q^d) \hookrightarrow S_{p, \infty}^r B(\Q^d)$, therefore, we get the assertion for $\frac{1}{p} - 1 < r < \frac{1}{p}$.

\end{proof}

In Definition \ref{sobolev_df} we mentioned Sobolev spaces with dominating mixed smoothness $S_p^r H(\Q^d) = S_{p \, 2}^r F(\Q)^d)$. We state discrepancy results for these spaces as well. They follow from the last corollary.

\begin{cor} \label{lower_to_main_H}

Let $1 \leq p < \infty$. Let $\frac{1}{\min(p,2)} - 1 < r < \frac{1}{p}$. Then there exists a constant $c > 0$ such that, for any integer $N \geq 2$ and all point sets $\P$ in $\Q^d$ with $N$ points, we have
\[ \left\| D_{\P} | S_p^r H(\Q^d) \right\| \geq c \, N^{r-1} \, \left( \log N \right)^{\frac{d-1}{2}}. \]

\end{cor}

\begin{rem}

We recall that $S_p^0 H(\Q^d) = L_p(\Q^d)$, therefore, we get Theorem \ref{thm_schmidt} as a consequence of the last corollary. The case $L_1(\Q^d)$ is not included.

\end{rem}

We would like to point out that there is a counterpart of Theorem \ref{lower_to_main_B} for the Triebel-Lizorkin spaces in \cite[Remark 6.28]{T10a}, which is not supported by arguments due to the fact that the embeddings of the spaces give changed conditions on $r$, as could be seen in this section. By the same argument we see that also the conditions on $r$ for the integration errors of the Triebel-Lizorkin spaces change. We will give a corrected version of this statement in the next chapter.

\section{Upper bounds}
In \cite[Theorem 6.13]{T10a} Triebel proved that for $1 \leq p, q \leq \infty$ and $q < \infty$ if $p = 1$ and $q > 1$ if $p = \infty$ and $\frac{1}{p} - 1 < r < \frac{1}{p}$, there exists a constant $C > 0$ such that, for any positive integer $N$ there exists a point set $\P$ in $\Q^d$ with $N$ points and we have
\[ \left\| D_{\P} | S_{pq}^r B(\Q^d) \right\| \leq C \, N^{r-1} \, \left( \log N \right)^{(d-1)(\frac{1}{q} + 1 - r)}. \]

Hinrichs conjectured that the correct upper bound might be the same as the lower bound and proved it in \cite{Hi10} in the $2$-dimensional case.

\begin{thm} \label{hinrichs_2_dim_Spq}

Let $1 \leq p, q \leq \infty$. Let $0 \leq r < \frac{1}{p}$. Then there exists a constant $C > 0$ such that, for any integer $N \geq 2$ there exists a point set $\P$ in $\Q^2$ with $N$ points such that
\[ \left\| D_{\P} | S_{pq}^r B(\Q^2) \right\| \leq C \, N^{r-1} \, \left( \log N \right)^{\frac{1}{q}}. \]

\end{thm}

The point sets used to prove the last theorem are the Hammersley type point sets. We will consider a generalization of these sets in the next subsection. The last theorem will follow as a consequence of our result.

%% file: DiscrepancyGeneralizedHammersleySpqrB.tex
\subsection{Discrepancy of generalized Hammersley type point sets}
We will generalize Theorem \ref{hinrichs_2_dim_Spq}, and though the bound will be the same, we will have a much larger class of point sets satisfying the optimal bound of $S_{pq}^r B$-discrepancy. This result can also be found in \cite{M13a}. The generalization will not work straightforward, it will require many additional calculations. Our goal is to close the gap in the exponents of the lower and upper bounds. We will prove results for arbitrary dimension in the next subsection using $b$-adic constructions. As a preparation we use much simpler $2$-dimensional $b$-adic constructions to demonstrate the possibility of such an approach.
\begin{df}
For any positive integer $n$ the point sets
\begin{multline*}
\Rn = \Big\{ \left( \frac{t_n}{b} + \frac{t_{n-1}}{b^2} + \ldots + \frac{t_1}{b^n},\frac{s_1}{b} + \frac{s_{2}}{b^2} + \ldots + \frac{s_n}{b^n} \right)| \\
t_1,\ldots,t_n\in\{0,1,\ldots,b-1\} \Big\}
\end{multline*}
where, for any $i = 1, \ldots, n$ either $s_i = t_i$ or $s_i = b - 1 - t_i$, are called generalized Hammersley type point sets.
\end{df}
The point sets $\Rn$ contain exactly $N = b^n$ points. For $b = 2$ one obtains original Hammersley type point sets proposed by Hammersley in \cite{Hm60}. The generalized Hammersley type point sets were defined by Faure in \cite{F81} and used in \cite{FP09} and \cite{FPPS10} to calculate their $L_2$-discrepancy. We denote additionally for any $\Rn$
\[ a_n = \# \{ i = 1, \ldots, n: \, s_i = t_i \}. \]
In \cite{Hi10} Hinrichs used only such point sets with $a_n = \left\lfloor \frac{n}{2} \right\rfloor$. The following results are nothing further but easy exercises.

\begin{lem} \label{lem_factor5}

For any integer $b\geq 2$ and for any $l \in \left\{ 1, \ldots, b - 1 \right\}$ we have
\[ \sum_{k = 1}^{b-1} k \e^{\frac{2 \pi \im}{b} l k} = \frac{b}{\e^{\frac{2 \pi \im}{b} l} - 1} = \sum_{k = 0}^{b - 2} \sum_{r = k + 1}^{b - 1} \e^{\frac{2\pi \im}{b} r l}. \]

\end{lem}

\begin{lem} \label{lem_haar_coeff_besov_x_ham}

Let $f(x) = x_1 x_2$ for $x = (x_1,x_2) \in \Q^2$. Let $j \in \N_{-1}^2,\, m \in \Dd_j, l \in \B_j$ and let $\mu_{jml}$ be the $b$-adic Haar coefficient of $f$. Then
\begin{enumerate}[(i)]
	\item If $j = (j_1,j_2) \in \N_0^2$ then
	      \[ \mu_{jml} = \frac{b^{-2j_1 - 2j_2 - 2}}{(\e^{\frac{2\pi \im}{b} l_1} - 1)(\e^{\frac{2\pi \im}{b} l_2} - 1)}. \]
	\item If $j = (j_1,-1)$ with $j_1 \in \N_0$ then
	      \[ \mu_{jml} = \frac{1}{2}\frac{b^{-2j_1 - 1}}{\e^{\frac{2\pi \im}{b} l_1} - 1}. \]	
	\item If $j=(-1,j_2)$ with $j_2\in\N_0$ then
	      \[ \mu_{jml} = \frac{1}{2} \frac{b^{-2j_2 - 1}}{\e^{\frac{2\pi \im}{b} l_2} - 1}. \]	
	\item If $j = (-1,-1)$ then $\mu_{jml} = \frac{1}{4}$.	
\end{enumerate}

\end{lem}

\begin{lem} \label{lem_haar_coeff_besov_indicator_ham}

Let $z = (z_1, z_2) \in \Q^2$ and $g(x) = \chi_{[0,x)}(z)$ for $x = (x_1, x_2) \in \Q^2$. Let $j \in \N_{-1}^2, \, m \in \Dd_j, l \in \B_j$ and let $\mu_{jml}$ be the $b$-adic Haar coefficient of $g$. Then $\mu_{jml} = 0$ whenever $z$ is not contained in the interior of the $b$-adic interval $I_{jm}$ supporting the functions $h_{jml}$. If $z$ is contained in the interior of $I_{jm}$ then
\begin{enumerate}[(i)]
	\item If $j = (j_1,j_2) \in \N_0^2$ then there is a $k = (k_1,k_2)$ with $k_1,k_2 \in \{0,1,\ldots,b-1\}$ such that, $z$ is contained in $I_{jm}^k$. Then 
	      \begin{multline*}
        \mu_{jml} = b^{-j_1-j_2-2} \left[ (bm_1+k_1+1-b^{j_1+1}z_1) \e^{\frac{2\pi \im}{b}k_1l_1} + \sum_{r_1=k_1+1}^{b-1} \e^{\frac{2\pi \im}{b}r_1l_1} \right] \times\\
        \times \left[ (bm_2+k_2+1 - b^{j_2 + 1} z_2) \e^{\frac{2\pi \im}{b}k_2l_2} + \sum_{r_2 = k_2 + 1}^{b-1} \e^{\frac{2\pi \im}{b}r_2l_2} \right].
        \end{multline*}       
  \item If $j = (j_1,-1)$ with $j_1 \in \N_0$ then there is a $k_1 \in \{0,1,\ldots,b-1\}$ such that, $z$ is contained in $I_{jm}^{(k_1,-1)}$. Then
        \[ \mu_{jml} = b^{-j_1-1} \left[ (bm_1 + k_1 + 1 - b^{j_1 + 1} z_1) \e^{\frac{2\pi \im}{b} k_1 l_1} + \sum_{r_1 = k_1 + 1}^{b-1} \e^{\frac{2\pi \im}{b}r_1l_1} \right] (1 - z_2). \]       
  \item If $j = (-1,j_2)$ with $j_2 \in \N_0$ then there is a $k_2 \in \{0,1,\ldots,b-1\}$ such that, $z$ is contained in $I_{jm}^{(-1,k_2)}$. Then
        \[ \mu_{jml} = b^{-j_2-1} (1-z_1) \left[ (bm_2 + k_2 + 1 - b^{j_2 + 1} z_2) \e^{\frac{2\pi \im}{b}k_2l_2} + \sum_{r_2 = k_2 + 1}^{b-1} \e^{\frac{2\pi \im}{b}r_2 l_2} \right]. \]        
  \item If $j = (-1,-1)$ then $\mu_{jml} = (1 - z_1)(1 - z_2)$.  
\end{enumerate}
\end{lem}

The following results are the biggest hurdle in this subsection.

\begin{lem} \label{lem_coeff_calc}

Let $\Rn$ be a generalized Hammersley type point set and let $j \in \N_0^2$ such that, $j_1 + j_2 < n - 1$, $m \in \Dd_j, l \in \B_j$ . Then
\begin{multline*}
\sum_{z \in \Rn \cap I_{jm}} \left[ (bm_1 + k_1 + 1 - b^{j_1 + 1} z_1) \e^{\frac{2\pi \im}{b}k_1l_1} + \sum_{r_1=k_1+1}^{b-1} \e^{\frac{2\pi \im}{b} r_1 l_1} \right] \times\\
\shoveright{\times\left[(bm_2+k_2+1-b^{j_2+1}z_2) \e^{\frac{2\pi \im}{b}k_2l_2}+\sum_{r_2=k_2+1}^{b-1}\e^{\frac{2\pi \im}{b}r_2l_2}\right]}\\
\shoveleft{= \frac{b^{n - j_1 - j_2} \pm b^{j_1 + j_2 - n + 2}}{(\e^{\frac{2\pi \im}{b}l_1} - 1)(\e^{\frac{2\pi \im}{b} l_2} - 1)}.}\\
\text{By the sign $\pm$ in the numerator we mean either $+$ or $-$ depending on $j$.} \qquad \qquad
\end{multline*}

\end{lem}

\begin{proof}

Let $z \in I_{jm}$. Then there is a $k \in \{ 0, 1, \ldots, b - 1 \}^2$ such that, $z \in I_{jm}^k$. We have $0 \leq m_i < b^{j_i},\,i=1,2$. Hence we can expand $m_i$ in base $b$ as
\[ m_i = b^{j_i - 1} m_1^{(i)} + b^{j_i - 2} m_2^{(i)} + \ldots + m_{j_i}^{(i)}. \]
Since $z \in \Rn \cap I_{jm}^k$ we have
\[ b^{-j_1 - 1}(bm_1 + k_1) \leq \frac{t_n}{b} + \frac{t_{n-1}}{b^2} + \ldots + \frac{t_1}{b^n} < b^{-j_1 - 1}(bm_1 + k_1 + 1). \]
Inserting the expansion of $m_1$ in the last inequality gives us
\begin{align*}
\frac{m_1^{(1)}}{b} + \frac{m_2^{(1)}}{b^2} + \ldots + \frac{m_{j_1}^{(1)}}{b^{j_1}} + \frac{k_1}{b^{j_1+1}} & \leq \frac{t_n}{b} + \frac{t_{n-1}}{b^2} + \ldots+\frac{t_1}{b^n}\\
& < \frac{m_1^{(1)}}{b} + \frac{m_2^{(1)}}{b^2} + \ldots + \frac{m_{j_1}^{(1)}}{b^{j_1}} + \frac{k_1+1}{b^{j_1+1}}.
\end{align*}
Analogously we have
\[ b^{-j_2 - 1} (bm_2+k_2) \leq \frac{s_1}{b} + \frac{s_2}{b^2} + \ldots + \frac{s_n}{b^n} < b^{-j_2 - 1}(bm_2 + k_2 + 1). \]
Hence,
\begin{align*}
\frac{m_1^{(2)}}{b} + \frac{m_2^{(2)}}{b^2} + \ldots + \frac{m_{j_2}^{(2)}}{b^{j_2}} + \frac{k_2}{b^{j_2+1}} & \leq \frac{s_1}{b} + \frac{s_2}{b^2} + \ldots + \frac{s_n}{b^n}\\
& < \frac{m_1^{(2)}}{b} + \frac{m_2^{(2)}}{b^2} + \ldots + \frac{m_{j_2}^{(2)}}{b^{j_2}} + \frac{k_2+1}{b^{j_2+1}}.
\end{align*}
So one gets a characterization of the fact that $z \in \Rn\cap I_{jm}^k$ in the form
\[ t_n = m_1^{(1)},\,t_{n - 1} = m_2^{(1)},\,\ldots,\,t_{n - j_1 + 1} = m_{j_1}^{(1)},\,t_{n - j_1} = k_1 \]
and
\[ s_1 = m_1^{(2)},\,s_2 = m_2^{(2)},\,\ldots,\,s_{j_2} = m_{j_2}^{(2)},\,s_{j_2+1} = k_2. \]
Hence $t_1,t_2,\ldots,t_{j_2}$ and $t_{n-j_1+1},\ldots,t_{n-1},t_n$ are determined by the condition $z \in \Rn \cap I_{jm}$ and $t_{n-j_1}$ and $t_{j_2+1}$ are determined by $k = (k_1,k_2)$ for which $z \in I_{jm}^k$ while $t_{j_2 + 2},\ldots,t_{n - j_1 - 1}\in\{0,1,\ldots,b-1\}$ can be chosen arbitrarily. Then we calculate
\begin{align*}
& bm_1 + k_1 + 1 - b^{j_1 + 1} z_1\\
& = 1 + b^{j_1} t_n + b^{j_1 - 1} t_{n - 1} + \ldots + bt_{n - j_1 + 1} + t_{n - j_1}\\
&\qquad - b^{j_1} t_n - b^{j_1 - 1} t_{n - 1} - \ldots - b^{j_1 - n + 1} t_1\\
& = 1 - b^{-1} t_{n - j_1 - 1} - \ldots - b^{j_1 - n + 1} t_1\\
& = 1 - b^{-1} t_{n - j_1 - 1} - \ldots - b^{j_1 + j_2 - n + 2} t_{j_2 + 2} - b^{j_1+  j_2 - n + 1} t_{j_2 + 1} - \varepsilon_1
\end{align*}
where
\[ \varepsilon_1 = b^{j_1 + j_2 - n} t_{j_2} + \ldots + b^{j_1 - n + 1} t_1 \]
and
\begin{align*}
& bm_2 + k_2 + 1 - b^{j_2 + 1} z_2\\
& = 1 + b^{j_2} s_1 + b^{j_2 - 1} s_2 + \ldots + bs_{j_2} + s_{j_2 + 1}\\
&\qquad - b^{j_2} s_1 - b^{j_2-  1} s_2 - \ldots - b^{j_2 - n + 1} s_n\\
& = 1 - b^{-1} s_{j_2 + 2} - \ldots - b^{j_2 - n + 1} s_n\\
& = 1 - b^{-1} s_{j_2 - 2} - \ldots - b^{j_1 + j_2 - n + 2} s_{n - j_1 - 1} - b^{j_1 + j_2 - n + 1} s_{n - j_1} - \varepsilon_2
\end{align*}
where
\[ \varepsilon_2 = b^{j_1 + j_2 - n} s_{n - j_1 + 1} + \ldots + b^{j_1 - n + 1} s_n. \]
This means that
\begin{align} \label{brackets_1}
b m_1 + k_1 + 1 - b^{j_1 + 1} z_1 = h b^{j_1 + j_2 - n + 2} - b^{j_1 + j_2 - n + 1} t_{j_2 + 1} - \varepsilon_1
\end{align}
for $h = 1,2,\ldots,b^{n - j_1 - j_2 - 2}$. It is clear that there must be some permutation $\sigma$ of $\{ 1, 2, \ldots, b^{n - j_1 - j_2 - 2} \}$ such that,
\begin{align} \label{brackets_2}
b m_2 + k_2 + 1 - b^{j_2 + 1} z_2 = \sigma(h) b^{j_1 + j_2 - n + 2} - b^{j_1 + j_2 - n + 1} s_{n - j_1} - \varepsilon_2.
\end{align}
We abbreviate $X = n - j_1 - j_2 - 2$. Then
\begin{align*}
& \sum_{z \in \Rn \cap I_{jm}} \left[ (bm_1 + k_1 + 1 - b^{j_1 + 1} z_1) \e^{\frac{2\pi \im}{b} k_1 l_1} + \sum_{r_1 = k_1 + 1}^{b-1} \e^{\frac{2\pi \im}{b} r_1 l_1} \right] \times\\
&\qquad\qquad\qquad\qquad\qquad \times \left[ (bm_2 + k_2 + 1 - b^{j_2 + 1} z_2) \e^{\frac{2\pi \im}{b}k_2l_2} + \sum_{r_2 = k_2 + 1}^{b-1} \e^{\frac{2\pi \im}{b} r_2 l_2} \right]\\
& = \sum_{k_1 = 0}^{b-1} \sum_{k_2 = 0}^{b-1} \sum_{z \in \Rn \cap I_{jm}^k} \left[ \ldots \right] \times \left[ \ldots \right]\\
& = \sum_{k_1 = 0}^{b-1} \sum_{k_2 = 0}^{b-1} \sum_{h = 1}^{b^X} \left[ \left( hb^{-X} - b^{-X - 1} t_{j_2 + 1} - \varepsilon_1 \right) \e^{\frac{2\pi \im}{b} k_1 l_1} + \sum_{r_1 = k_1 + 1}^{b-1} \e^{\frac{2\pi \im}{b} r_1 l_1} \right] \times\\
&\qquad\qquad\qquad\qquad \times \left[ \left( \sigma(h) b^{-X} - b^{-X - 1} s_{n - j_1} - \varepsilon_2 \right) \e^{\frac{2\pi \im}{b} k_2 l_2} + \sum_{r_2 = k_2 + 1}^{b-1} \e^{\frac{2\pi \im}{b} r_2 l_2} \right].
\end{align*}

After having expanded the product and changed the order of summation we analyze the summands separately in a fitting order. We recall that $s_{n-j_1}$ depends on $k_1$ and $t_{j_2+1}$ depends on $k_2$. Except the last two, all summands are equal to zero because each has the sum of unity roots as a factor. The summands are the following
\begin{align*}
 & \sum_{h=1}^{b^X}\left(hb^{-X} - \varepsilon_1 \right) \left(\sigma(h) b^{-X} - \varepsilon_2 \right) \sum_{k_1 = 0}^{b-1} \e^{\frac{2\pi \im}{b}k_1 l_1} \sum_{k_2 = 0}^{b-1}\e^{\frac{2\pi \im}{b} k_2 l_2} = 0,\\
-& \sum_{h=1}^{b^X} \left(hb^{-X} - \varepsilon_1 \right) b^{-X - 1} \sum_{k_1 = 0}^{b-1} s_{n - j_1} \e^{\frac{2\pi \im}{b} k_1 l_1} \sum_{k_2 = 0}^{b-1} \e^{\frac{2\pi \im}{b} k_2 l_2} = 0,\\
-& \sum_{h=1}^{b^X}\left(\sigma(h) b^{-X} - \varepsilon_2 \right) b^{-X - 1} \sum_{k_2 = 0}^{b-1} t_{j_2 + 1} \e^{\frac{2\pi \im}{b} k_2 l_2} \sum_{k_1 = 0}^{b-1} \e^{\frac{2\pi \im}{b} k_1 l_1} = 0,\\
 & \sum_{h=1}^{b^X} \left(hb^{-X} - \varepsilon_1 \right) \sum_{k_2 = 0}^{b-1} \sum_{r_2 = k_2 + 1}^{b-1} \e^{\frac{2\pi \im}{b} r_2 l_2} \sum_{k_1 = 0}^{b-1} \e^{\frac{2\pi \im}{b}k _1 l_1} = 0,\\
 & \sum_{h=1}^{b^X} \left(\sigma(h) b^{-X} - \varepsilon_2 \right) \sum_{k_1 = 0}^{b-1} \sum_{r_1 = k_1 + 1}^{b-1} \e^{\frac{2\pi \im}{b} r_1 l_1} \sum_{k_2 = 0}^{b-1} \e^{\frac{2\pi \im}{b} k_2 l_2} = 0,\\
-& \sum_{h=1}^{b^X} b^{-X - 1} \sum_{k_1 = 0}^{b-1} \sum_{r_1 = k_1 + 1}^{b-1} s_{n - j_1} \e^{\frac{2\pi \im}{b} r_1 l_1} \sum_{k_2 = 0}^{b-1} \e^{\frac{2\pi \im}{b} k_2 l_2} = 0,\\
-& \sum_{h=1}^{b^X} b^{-X - 1} \sum_{k_2 = 0}^{b-1} \sum_{r_2 = k_2 + 1}^{b-1} t_{j_2 + 1} \e^{\frac{2\pi \im}{b} r_2 l_2} \sum_{k_1 = 0}^{b-1} \e^{\frac{2\pi \im}{b} k_1 l_1} = 0,\\
 & \sum_{h=1}^{b^X} \sum_{k_1 = 0}^{b-1} \sum_{r_1 = k_1 + 1}^{b-1} \e^{\frac{2\pi \im}{b} r_1 l_1} \sum_{k_2 = 0}^{b-1} \sum_{r_2 =k _2 + 1}^{b-1} \e^{\frac{2\pi \im}{b} r_2 l_2} = \frac{b^{n - j_1 - j_2}}{(\e^{\frac{2\pi \im}{b} l_1} - 1)(\e^{\frac{2\pi \im}{b} l_2} - 1)}
\end{align*}
by Lemma \ref{lem_factor5}.
Finally, the last summand is
\begin{multline*}
\sum_{h=1}^{b^X} \sum_{k_1 = 0}^{b-1} \sum_{k_2 = 0}^{b-1} b^{-X - 1} t_{j_2 + 1} b^{-X - 1} s_{n - j_1} \e^{\frac{2\pi \im}{b} k_1 l_1} \e^{\frac{2\pi \im}{b} k_2 l_2}\\
= b^{j_1 + j_2 - n} \sum_{k_1 = 0}^{b-1} s_{n - j_1} \e^{\frac{2\pi \im}{b} k_1 l_1} \sum_{k_2 = 0}^{b-1} t_{j_2 + 1} \e^{\frac{2\pi \im}{b} k_2 l_2}.
\end{multline*}
We know that $t_{n - j_1} = k_1$ and that either $s_i = t_i$ or $s_i = b - 1 - t_i$ for all $i = 1,\ldots,n$. Hence $s_{n - j_1}$ is either $k_1$ or $b - 1 - k_1$. Since
\[ \sum_{k_1 = 0}^{b-1} (b-1) \e^{\frac{2\pi \im}{b} k_1 l_1} = 0 \]
we have
\begin{align} \label{sign_term}
\sum_{k_1 = 0}^{b-1} s_{n - j_1} \e^{\frac{2\pi \im}{b} k_1 l_1} = \pm \frac{b}{\e^{\frac{2\pi \im}{b} l_1} - 1}
\end{align}
using Lemma \ref{lem_factor5} and the sign depends on $j_1$. Also we know that $s_{j_2 + 1} = k_2$ and that either $s_{j_2 + 1} = t_{j_2 + 1}$ or $s_{j_2} = b - 1 - t_{j_2 + 1}$. Hence
\[ \sum_{k_2 = 0}^{b-1} t_{j_2 + 1} \e^{\frac{2\pi \im}{b} k_2 l_2} = \pm \frac{b}{\e^{\frac{2\pi \im}{b} l_2} - 1} \]
and the sign depends on $j_2$. So altogether our last summand is
\[ b^{j_1 + j_2 - n} \frac{\pm b^2}{(\e^{\frac{2\pi \im}{b} l_1} - 1) (\e^{\frac{2\pi \im}{b} l_2} - 1)} = \frac{\pm b^{j_1 + j_2 - n + 2}}{(\e^{\frac{2\pi \im}{b} l_1} - 1)(\e^{\frac{2\pi \im}{b} l_2} - 1)} \]
and the sign depends on $j$. Adding both summands which are nonzero gives us the stated result.

\end{proof}

\begin{lem} \label{lem_xn}

Let
\[ x_n = \sum_{t_1,\ldots,t_n = 0}^{b-1}\sum_{j=1}^n b^{-j} t_j \]
and
\[ y_n = \sum_{t_1,\ldots,t_n = 0}^{b-1} \sum_{i=1}^n b^i t_i \]
for any positive integer $n$. Then
\[ x_n = \frac{1}{2}(b^n - 1) \]
and
\[ y_n = b^{n + 1} x_n = \frac{1}{2} b^{n+1} (b^n - 1). \]

\end{lem}

\begin{proof}

Clearly, $x_1 = \frac{1}{2} (b-1)$ and inductively
\begin{align*}
x_n & = \sum_{t_n} \sum_{t_1,\ldots,t_{n-1}} \sum_{j=1}^{n-1} b^{-j} t_j + b^{-n} \sum_{t_1,\ldots,t_{n-1}} \sum_{t_n} t_n\\
    & = b \, x_{n-1} + b^{-n} \, b^{n-1} \, \frac{b \, (b-1)}{2}\\
    & = b \, \frac{1}{2} \, (b^{n-1} - 1) + \frac{1}{2} \, (b-1)\\
    & = \frac{1}{2} \, (b^n-1).
\end{align*}
One sees that $y_n = b^{n+1} x_n$ simply by checking that
\[ \sum_{i=1}^n b^i t_i = b^{n+1} \sum_{i=1}^n b^{i - n - 1} t_i = b^{n+1} \sum_{i=1}^n b^{-i} t_{n + 1 - i}. \]
Summing over $t_1,\ldots,t_n$ will give us $y_n$ on the left side. On the right side it will give us $b^{n+1} x_n$ although the order of the $t_i$ is reversed with respect to the definition of the numbers $x_n$.

\end{proof}

\begin{rem} \label{use_this_fact}

We will use this fact that the order of the $t_i$ is irrelevant in further proofs. But not only the order is irrelevant but even the concrete index of the $t_j$. For example the value of
\[ \sum_{t_{n+1},\ldots,t_{2n} = 0}^{b-1} \sum_{j=1}^n b^{-j} t_{j + n} \]
is the same as the value of $x_n$.

\end{rem}

\begin{lem} \label{lem_zn}

Let
\[ \zeta_n = \sum_{t_1,\ldots,t_n = 0}^{b-1} \sum_{i,j=1}^n b^{i-j} t_i t_j \]
for any positive integer $n$. Then
\[ \zeta_n = \frac{1}{4} b^{2n+1} + \frac{n}{12} b^{n+2} - \frac{1}{2} b^{n+1} - \frac{n}{12} b^n + \frac{1}{4} b. \]

\end{lem}

\begin{proof}

Clearly, $\zeta_1 = \frac{1}{6} (b-1) b (2b-1) = \frac{1}{3} b^3 - \frac{1}{2} b^2 + \frac{1}{6} b$. Then inductively we get
\begin{align*}
\zeta_n & = \sum_{t_n} \sum_{t_1,\ldots,t_{n-1}} \sum_{i=1}^{n-1} \sum_{j=1}^{n-1} b^{i-j} t_i t_j + b^n \sum_{t_n} t_n \sum_{t_1,\ldots,t_{n-1}} \sum_{j=1}^{n-1} b^{-j} t_j +\\
    &\qquad + b^{-n} \sum_{t_n} t_n \sum_{t_1,\ldots,t_{n-1}} \sum_{i=1}^{n-1} b^i t_j + \sum_{t_1,\ldots,t_{n-1}} \sum_{t_n} t_n^2\\
    & = b \, \zeta_{n-1} + b^n \, \frac{1}{2} \, (b-1) \, b \, x_{n-1} + b^{-n} \, \frac{1}{2} \, (b-1) \, b \, y_{n-1} + b^{n-1} \, \frac{1}{6} \, (b-1) \, b \, (2b-1)\\
    & = b \left( \frac{1}{4} b^{2n-1} + \frac{n-1}{12} \, b^{n+1} - \frac{1}{2} \, b^n - \frac{n-1}{12} \, b^{n-1} + \frac{1}{4} \, b \right) +\\
    &\qquad + b^n \, \frac{1}{2} \, (b-1) \, b \left( \frac{1}{2} (b^{n-1} - 1) \right) + b^{-n} \frac{1}{2} \, (b-1) \, b \left( \frac{1}{2} \, b^n (b^{n-1} - 1) \right) +\\
    &\qquad + b^{n-1} \frac{1}{6} \, (b-1) \, b \, (2b-1)\\
    & = \frac{1}{4} b^{2n+1} + \frac{n}{12} b^{n+2} - \frac{1}{2} b^{n+1} - \frac{n}{12} b^n + \frac{1}{4} b.
\end{align*}
\end{proof}

\begin{lem} \label{lem_sum_z}

Let $\Rn$ be a generalized Hammersley type point set and $z = (z_1, z_2) \in \Rn$. Then
\[ \sum_{z \in \Rn} (1 - z_1) (1 - z_2) = 1 + b^{-n - 1} \sum_{t_1,\ldots,t_n}^{b-1} \sum_{i,j=1}^n b^{i-j} t_i s_j. \]

\end{lem}

\begin{proof}

We first calculate for some $z \in \Rn$
\begin{align*}
(1 - z_1) (1 - z_2) & = (1 - b^{-1} t_n - \ldots -b ^{-n} t_1) (1 - b^{-1} s_1 - \ldots - b^{-n} s_n)\\
                    & = 1 - b^{-1} t_n - \ldots - b^{-n} t_1 - b^{-1} s_1 - \ldots - b^{-n} s_n +\\
                    &\qquad + \sum_{i,j=1}^n b^{-n + i - j - 1} t_i s_j.
\end{align*}

Now we sum over all $z \in \Rn$ which corresponds to summing over all $t_1,\ldots,t_n \in\{0,1,\ldots,b-1\}$ and get
\begin{align*}
& \sum_{z \in \Rn} (1 - z_1) (1 - z_2)\\
& = \sum_{t_1,\ldots,t_n} \left(1 - b^{-1} t_n - \ldots - b^{-n} t_1 - b^{-1} s_1 - \ldots - b^{-n} s_n + b^{-n - 1}\sum_{i,j=1}^n b^{i-j} t_i s_j \right)\\
& = b^n - b^{-1} \, b^{n-1} \sum_{t_n = 0}^{b-1} t_n - b^{-1} \, b^{n-1} \sum_{t_1 = 0}^{b-1} s_1 - \ldots - b^{-n} \, b^{n-1} \sum_{t_1 = 0}^{b-1} t_1 -\\
&\qquad\qquad - b^{-n} \, b^{n-1} \sum_{t_n = 0}^{b-1} s_n + b^{-n-1} \sum_{t_1,\ldots,t_n} \sum_{i,j=1}^n b^{i-j} t_i s_j\\
& = b^n - 2 \left( b^{n-2} \, \frac{1}{2} \, (b-1) \, b + \ldots + b^{-1} \, \frac{1}{2} \, (b-1) \, b \right) + b^{-n - 1} \sum_{t_1,\ldots,t_n} \sum_{i,j=1}^n b^{i-j} t_i s_j\\
& = b^n - (b-1) (b^{n-1} + \ldots + 1) + b^{-n - 1} \sum_{t_1,\ldots,t_n} \sum_{i,j=1}^n b^{i-j} t_i s_j\\
& = 1 + b^{-n - 1} \sum_{t_1,\ldots,t_n} \sum_{i,j=1}^n b^{i-j} t_i s_j\\
\end{align*}
\end{proof}

\begin{lem} \label{lem_a_ident}

Let $\Rn$ be a generalized Hammersley type point set. Then
\[ \sum_{t_1,\ldots,t_n = 0}^{b-1} \sum_{i,j=1}^n b^{i-j} t_i s_j = \frac{1}{4} b^{2n+1} - \frac{1}{2} b^{n+1} + \frac{1}{4} b + (2a_n - n) \frac{b^2 - 1}{12} b^n. \]

\end{lem}

\begin{proof}

We can assume that $s_1 = t_1,\ldots,s_{a_n} = t_{a_n}, s_{a_n+1} = b - 1 - t_{a_n+1}, \ldots, s_n = b - 1 - t_n$. Otherwise we would rename the $t_j$. This assumption allows us to split the sum in a compact way. We have
\begin{align*}
\sum_{i, j=1}^n b^{i-j} t_i s_j & = \sum_{i,j=1}^{a_n} b^{i - j} t_i t_j + \sum_{i = 1}^{a_n} \sum_{j = a_n + 1}^n b^{i - j} t_i (b - 1 - t_j) + \\
                                & \qquad + \sum_{i = a_n + 1}^n \sum_{j = 1}^{a_n} b^{i - j} t_i t_j + \sum_{i, j = a_n + 1}^n b^{i - j} t_i (b - 1 - t_j) \\
                                & = \sum_{i, j = 1}^{a_n} b^{i - j} t_i t_j + (b - 1) \sum_{i = 1}^{a_n} \sum_{j = a_n + 1}^n b^{i - j} t_i - \sum_{i = 1}^{a_n} \sum_{j = a_n + 1}^n b^{i - j} t_i t_j + \\
                                & \qquad + \sum_{i = a_n + 1}^n \sum_{j = 1}^{a_n} b^{i-j} t_i t_j + (b - 1)\sum_{i = a_n + 1}^n \sum_{j = a_n + 1}^n b^{i-j} t_i - \\
                                & \qquad \qquad - \sum_{i = a_n + 1}^n \sum_{j = a_n + 1}^n b^{i - j} t_i t_j.
\end{align*}
Summing over $t_1,\ldots,t_n$ and analyzing every term separately will give us
\[ \sum_{t_1, \ldots, t_n} \sum_{i, j = 1}^{a_n} b^{i - j} t_i t_j = b^{n - a_n} \zeta_{a_n}, \]
as well as (using $y_n = b^{n+1} x_n$)
\begin{align*}
\sum_{t_1, \ldots, t_n}(b - 1) \sum_{i = 1}^{a_n} \sum_{j = a_n + 1}^n b^{i - j} t_i & = (b - 1) b^{n - a_n} y_{a_n} \sum_{j = a_n + 1}^n b^{-j} \\
                                                                                     & = b^{n+1} x_{a_n} (b^{-a_n} - b^{-n}),
\end{align*}
and
\begin{align*}
\sum_{t_1, \ldots, t_n} \sum_{i = 1}^{a_n} \sum_{j = a_n + 1}^n b^{i - j} t_i t_j & = \sum_{t_1, \ldots, t_{a_n}} \sum_{i = 1}^{a_n} b^i t_i \sum_{t_{a_n + 1}, \ldots, t_n} \sum_{j = a_n + 1}^n b^{-j} t_j \\
                                                                                  & = y_{a_n} \sum_{t_{a_n + 1}, \ldots, t_n} b^{-a_n} \sum_{j = a_n + 1}^n b^{a_n - j} t_j = x_{a_n} x_{n - a_n} b,
\end{align*}
since we have already seen that the indices of $t_j$ are irrelevant (Remark \ref{use_this_fact}). We also get with a similar argumentation
\begin{align*}
& \sum_{t_1, \ldots, t_n} \sum_{i = a_n + 1}^n \sum_{j = 1}^{a_n} b^{i - j} t_i t_j = \sum_{t_1, \ldots, t_{a_n}} \sum_{j = 1}^{a_n} b^{-j} t_j \sum_{t_{a_n + 1}, \ldots, t_n} \sum_{i = a_n + 1}^n b^i t_i \\
& = \sum_{t_1, \ldots, t_{a_n}} \sum_{j = 1}^{a_n} b^{-j} t_j \sum_{t_{a_n + 1}, \ldots, t_n} b^{a_n} \sum_{i = a_n + 1}^n b^{i - a_n} t_i \\
& = x_{a_n} b^{a_n} y_{n - a_n} = x_{a_n} x_{n - a_n} b^{n + 1}
\end{align*}
and
\begin{align*}
& \sum_{t_1, \ldots, t_n} (b - 1) \sum_{i = a_n + 1}^n \sum_{j = a_n + 1}^n b^{i - j} t_i = (b-1) b^{a_n} \sum_{t_{a_n + 1}, \ldots, t_n} \sum_{i = a_n + 1}^n b^i t_i \sum_{j = a_n + 1}^n b^{-j} \\
& = b^{a_n} y_{n - a_n} b^{a_n} (b^{-a_n} - b^{-n}) = x_{n - a_n} (b^{n + 1} - b^{a_n + 1}) \\
\end{align*}
and
\begin{multline*}
\sum_{t_1, \ldots, t_n} \sum_{i = a_n + 1}^n \sum_{j = a_n + 1}^n b^{i - j} t_i t_j \\
= b^{a_n} \sum_{t_{a_n + 1}, \ldots, t_n} \sum_{i = a_n + 1}^n \sum_{j = a_n + 1}^n b^{(i - a_n)+(a_n - j)} t_i t_j = b^{a_n} \zeta_{n - a_n}.
\end{multline*}

So what we have is
\begin{multline*}
\sum_{t_1, \ldots, t_n}^{b - 1} \sum_{i, j = 1}^n b^{i - j} t_i s_j \\
= b^{n - a_n} \zeta_{a_n} - b^{a_n} \zeta_{n - a_n} + x_{a_n} b(b^{n - a_n} - 1) + x_{a_n} x_{n - a_n} b (b^n - 1) + x_{n - a_n} b^{a_n + 1}(b^{n - a_n} - 1).
\end{multline*}

Inserting the values of $\zeta_{a_n}, \, \zeta_{n - a_n}, \, x_{a_n}$, and $x_{n - a_n}$ and simplifying will give us the stated assertion.

\end{proof}

\begin{prp} \label{prp_minus1}

Let $\Rn$ be a generalized Hammersley type point set and $\mu_{jml}$ the $b$-adic Haar coefficients of its discrepancy function. Then
\[ \mu_{(-1,-1),(0,0),(1,1)} = \frac{1}{4} b^{-2n} + \frac{1}{2} b^{-n} + (2a_n - n) \frac{b^2 - 1}{12} b^{-n - 1}. \]

\end{prp}

\begin{proof}

Using Lemma \ref{lem_a_ident} we have
\[ \sum_{t_1,\ldots,t_n = 0}^{b-1} \sum_{i,j=1}^n b^{i-j} t_i s_j = \frac{1}{4} b^{2n+1} - \frac{1}{2} b^{n+1} + \frac{1}{4} b + (2a_n - n) \frac{b^2 - 1}{12} b^n. \]
Hence using Lemmas \ref{lem_haar_coeff_besov_x_ham}, \ref{lem_haar_coeff_besov_indicator_ham} and \ref{lem_sum_z} we get
\begin{align*}
& \mu_{(-1,-1),(0,0),(1,1)} = b^{-n} \sum_{z \in \Rn} (1 - z_1) (1 - z_2) - \frac{1}{4}\\
& = b^{-n} \left( 1 + b^{-n - 1} \left( \frac{1}{4} b^{2n+1} - \frac{1}{2} b^{n+1} + \frac{1}{4} b + (2a_n - n) \, \frac{b^2 - 1}{12} \, b^n \right) \right) - \frac{1}{4}\\
& = \frac{1}{4} b^{-2n} + \frac{1}{2} b^{-n} + (2a_n - n) \, \frac{b^2 - 1}{12} \, b^{-n - 1}.
\end{align*}
\end{proof}

\begin{lem} \label{lem_middle}

Let $\Rn$ be a generalized Hammersley type point set and let $j = (j_1,-1)$ for $j_1 \in \N_0$ with $j_1 \leq n-1$, $m = (m_1,0)$ with $0 \leq m_1 < b^{j_1}$ and $l = (l_1,1)$ with $1 \leq l_1 < b$. Then
\begin{multline*}
\sum_{z \in \Rn \cap I_{jm}} \left[ (bm_1 + k_1 + 1 - b^{j_1 + 1} z_1) \e^{\frac{2\pi \im}{b} k_1 l_1} + \sum_{r_1 = k_1 + 1}^{b-1} \e^{\frac{2\pi \im}{b} r_1 l_1} \right] (1 - z_2)\\
= \frac{b^{n-j_1}(1 - 2\varepsilon) \mp b^{j_1-n+1}}{2(\e^{\frac{2\pi \im}{b} l_1} - 1)} + \frac{w_{j_1}}{(\e^{\frac{2\pi \im}{b} l_1} - 1)^2},
\end{multline*}
where $w_{j_1}$ is either $\e^{\frac{2\pi \im}{b} l_1}$ or $-1$ and the sign of $\mp$ depends on $j_1$ and we have $\varepsilon b^{n - j_1} \leq b$.

An analogous result holds for $j = (-1,j_2)$ where $j_2 \in \N_0$ with $j_2 \leq n - 1$, $m = (0,m_2)$ with $0 \leq m_2 < b^{j_2}$ and $l = (1,l_2)$ with $1 \leq l_2 < b$.
\end{lem}

\begin{proof}

Let $z \in \Rn \cap I_{jm}$. Then there is a $k = (k_1,-1)$, $k_1 \in \{ 0, 1, \ldots, b - 1 \}$ such that, $z \in \Rn \cap I_{jm}^k$. We use the methods from Lemma \ref{lem_coeff_calc} for the proof. We have
\[ bm_1 + k_1 + 1 - b^{j_1 + 1} z_1 = 1 - b^{-1} t_{n - j_1 - 1} - \ldots - b^{j_1 - n + 1} t_1 \]
which means that
\[ bm_1 + k_1 + 1 - b^{j_1 + 1} z_1 = h b^{j_1 - n + 1} \]
for $h = 1, 2, \ldots, b^{n - j_1 - 1}$. The numbers $t_{n - j_1 + 1},\ldots,t_n$ are determined by the condition $z \in \Rn \cap I_{jm}$ and $t_{n - j_1} = k_1$. The numbers $t_1, \ldots, t_{n - j_1 - 1}$ can be chosen arbitrarily. We also have
\[ 1 - z_2 = 1 - b^{-1} s_1 - \ldots - b^{j_1 - n + 1} s_{n - j_1 - 1} - b^{j_1 - n} s_{n - j_1} - \varepsilon \]
where $\varepsilon = b^{j_1 - n - 1} s_{n - j_1 + 1} + \ldots + b^{-n} s_n$. Clearly, $\varepsilon b^{n - j_1} \leq b$.

So there must be a permutation $\sigma$ such that,
\[ 1 - z_2 = \sigma(h) b^{j_1 - n + 1} - b^{j_1 - n} s_{n - j_1} - \varepsilon. \]
Hence
\begin{align*}
& \sum_{z \in \Rn \cap I_{jm}} \left[ (bm_1 + k_1 + 1 - b^{j_1 + 1} z_1) \e^{\frac{2\pi \im}{b} k_1 l_1} + \sum_{r_1 = k_1 + 1}^{b-1} \e^{\frac{2\pi \im}{b} r_1 l_1} \right] (1 - z_2)\\
& = \sum_{k_1 = 0}^{b-1} \sum_{h=1}^{b^{n - j_1 - 1}} \left[hb^{j_1 - n + 1} \e^{\frac{2\pi \im}{b} k_1 l_1} + \sum_{r_1 = k_1 + 1}^{b-1} \e^{\frac{2\pi \im}{b} r_1 l_1} \right] (\sigma(h) b^{j_1 - n + 1} - b^{j_1 - n} s_{n - j_1} - \varepsilon)\\
\end{align*}

We analyze the summands separately after having expanded the product and changed the order of summation. We have
\begin{align*}
 & \sum_{h=1}^{b^{n - j_1 - 1}} h \sigma(h) b^{j_1 - n + 1} b^{j_1 - n + 1} \sum_{k_1 = 0}^{b-1} \e^{\frac{2\pi \im}{b} k_1 l_1} = 0, \\
-& \sum_{h=1}^{b^{n - j_1 - 1}} h b^{j_1 - n + 1} b^{j_1 - n} \sum_{k_1 = 0}^{b-1} s_{n - j_1} \e^{\frac{2\pi \im}{b} k_1 l_1} \\
 & = -\frac{1}{2} b^{n - j_1 - 1} (b^{n - j_1 - 1} + 1) b^{2j_1 - 2n + 1} \frac{\pm b}{\e^{\frac{2\pi \im}{b} l_1} - 1} = \mp \frac{b^{j_1 - n + 1} + 1}{2(\e^{\frac{2\pi \im}{b} l_1} - 1)}, \\
\intertext{using \eqref{sign_term},} \\
-& \varepsilon \sum_{h=1}^{b^{n - j_1 - 1}} h b^{j_1 - n + 1} \sum_{k_1 = 0}^{b-1} \e^{\frac{2\pi \im}{b} k_1 l_1} = 0, \\
 & \sum_{h=1}^{b^{n - j_1 - 1}} \sigma(h) b^{j_1 - n + 1} \sum_{k_1 = 0}^{b-1} \sum_{r_1 = k_1 + 1}^{b-1} \e^{\frac{2\pi \im}{b} r_1 l_1} \\
 & = \frac{1}{2} b^{n - j_1 - 1} (b^{n - j_1 - 1} + 1) b^{j_1 - n + 1} \frac{b}{\e^{\frac{2\pi \im}{b} l_1} - 1} = \frac{b^{n - j_1} + b}{2(\e^{\frac{2\pi \im}{b} l_1} - 1)}, \\
-& \varepsilon \sum_{h=1}^{b^{n - j_1 - 1}} \sum_{k_1 = 0}^{b-1} \sum_{r_1 = k_1 + 1}^{b-1} \e^{\frac{2\pi \im}{b} r_1 l_1} = -\varepsilon b^{n - j_1 - 1}\frac{b}{\e^{\frac{2\pi \im}{b} l_1} - 1} = \frac{-\varepsilon b^{n - j_1}}{\e^{\frac{2\pi \im}{b} l_1} - 1}, \\
\intertext{and}
-& \sum_{h=1}^{b^{n - j_1 - 1}} b^{j_1 - n} \sum_{k_1 = 0}^{b-1} s_{n - j_1} \sum_{r_1 = k_1 + 1}^{b-1} \e^{\frac{2\pi \im}{b} r_1 l_1}. \\
\end{align*}

For the last term we use the fact that $s_{n - j_1}$ is either $k_1$ or $b - 1 - k_1$. In the first case we have
\begin{align*}
& \sum_{k_1 = 0}^{b-1} k_1 \sum_{r_1 = k_1 + 1}^{b-1} \e^{\frac{2\pi \im}{b} r_1 l_1}\\
& = \sum_{k_1 = 1}^{b-2} k_1 \frac{1 - \e^{\frac{2\pi \im}{b} (k_1 + 1) l_1}}{\e^{\frac{2\pi \im}{b} l_1} - 1}\\
& = \frac{1}{\e^{\frac{2\pi \im}{b} l_1} - 1} \left( \frac{1}{2} (b-2) (b-1) - \sum_{k_1 = 2}^{b-1} (k_1 - 1)\e^{\frac{2\pi \im}{b} k_1 l_1} \right)\\
& = \frac{1}{\e^{\frac{2\pi \im}{b} l_1} - 1} \left( \frac{1}{2} (b-2) (b-1) - \left( \frac{b}{\e^{\frac{2\pi \im}{b} l_1} - 1} - \e^{\frac{2\pi \im}{b} l_1} \right) + \left( 0 - 1 - \e^{\frac{2\pi \im}{b} l_1} \right) \right)\\
& = \frac{1}{\e^{\frac{2\pi \im}{b} l_1} - 1} \left( \frac{b^2 - 3b}{2} - \frac{b}{\e^{\frac{2\pi \im}{b} l_1} - 1} \right)\\
& = \frac{(b-3) b}{2(\e^{\frac{2\pi \im}{b} l_1} - 1)} - \frac{b}{(\e^{\frac{2\pi \im}{b} l_1} - 1)^2}.
\end{align*}
In the other case we have
\begin{align*}
& \sum_{k_1 = 0}^{b-1} (b - 1 - k_1) \sum_{r_1 = k_1 + 1}^{b-1} \e^{\frac{2\pi \im}{b} r_1 l_1}\\
& = (b-1) \sum_{k_1 = 0}^{b-1} \sum_{r_1 = k_1 + 1}^{b-1}\e^{\frac{2\pi \im}{b} r_1 l_1} - \sum_{k_1 = 0}^{b-1} k_1 \sum_{r_1 = k_1 + 1}^{b-1} \e^{\frac{2\pi \im}{b} r_1 l_1}\\
& = \frac{(b-1)b}{(\e^{\frac{2\pi \im}{b} l_1} - 1)} - \frac{(b-3)b}{2(\e^{\frac{2\pi \im}{b} l_1} - 1)} + \frac{b}{(\e^{\frac{2\pi \im}{b} l_1} - 1)^2}\\
& = \frac{b(b+1)}{2(\e^{\frac{2\pi \im}{b} l_1} - 1)} + \frac{b}{(\e^{\frac{2\pi \im}{b} l_1} - 1)^2}.
\end{align*}
So the last term is either
\[ \frac{1}{(\e^{\frac{2\pi \im}{b} l_1} - 1)^2} - \frac{b-3}{2(\e^{\frac{2\pi \im}{b} l_1} - 1)} \]
or
\[ -\frac{b+1}{2(\e^{\frac{2\pi \im}{b} l_1} - 1)} - \frac{1}{(\e^{\frac{2\pi \im}{b} l_1} - 1)^2}. \]
Now combining the results we get in the case $s_{n - j_1} = k_1$
\begin{align*}
& \sum_{z \in \Rn \cap I_{jm}} \left[ (bm_1 + k_1 + 1 - b^{j_1 + 1} z_1) \e^{\frac{2\pi \im}{b} k_1 l_1} + \sum_{r_1 = k_1 + 1}^{b-1} \e^{\frac{2\pi \im}{b} r_1 l_1} \right] (1 - z_2)\\
& = \frac{b^{n-j_1}(1 - 2\varepsilon) - b^{j_1-n+1}}{2(\e^{\frac{2\pi \im}{b} l_1} - 1)} + \frac{\e^{\frac{2\pi \im}{b} l_1}}{(\e^{\frac{2\pi \im}{b} l_1} - 1)^2}
\end{align*}
while in the case $s_{n - j_1} = b - 1 - k_1$
\begin{align*}
& \sum_{z \in \Rn \cap I_{jm}} \left[ (bm_1 + k_1 + 1 - b^{j_1 + 1} z_1) \e^{\frac{2\pi \im}{b} k_1 l_1} + \sum_{r_1 = k_1 + 1}^{b-1} \e^{\frac{2\pi \im}{b} r_1 l_1} \right] (1 - z_2)\\
& = \frac{b^{n-j_1}(1 - 2\varepsilon) + b^{j_1-n+1}}{2(\e^{\frac{2\pi \im}{b} l_1} - 1)} - \frac{1}{(\e^{\frac{2\pi \im}{b} l_1} - 1)^2}
\end{align*}
as stated by the lemma.

\end{proof}

We now summarize the results of this subsection.
\pagebreak

\begin{prp} \label{prp_haar_coeff}

Let $\Rn$ be a generalized Hammersley type point set and let $\mu_{jml}$ be the $b$-adic Haar coefficient of its discrepancy function for $j \in \N_{-1}^2, \, m \in \Dd_j$ and $l \in \B_j$. Then
\begin{enumerate}[(i)]
	\item if $j \in \N_0^2$ and $j_1 + j_2 < n-1$ then
\[ \left| \mu_{jml} \right| = \frac{b^{-2n}}{\left|\e^{\frac{2\pi \im}{b} l_1} - 1\right|\left|\e^{\frac{2\pi \im}{b} l_2} - 1\right|}, \] \label{prp_1}
	\item if $j \in \N_0^2$, $j_1 + j_2 \geq n-1$ and $j_1,j_2 \leq n$ then $\left| \mu_{jml} \right| \leq c b^{-n - j_1 - j_2}$ for some constant $c > 0$ and
\[ \left| \mu_{jml} \right| = \frac{b^{-2j_1 - 2j_2 - 2}}{\left|\e^{\frac{2\pi \im}{b} l_1} - 1\right|\left|\e^{\frac{2\pi \im}{b} l_2} - 1\right|} \]
for all but $b^n$ coefficients $\mu_{jml}$, \label{prp_2}
	\item if $j \in \N_0^2$ and $j_1 \geq n$ or $j_2 \geq n$ then
\[ \left| \mu_{jml} \right| = \frac{b^{-2j_1 - 2j_2 - 2}}{\left|\e^{\frac{2\pi \im}{b} l_1} - 1\right|\left|\e^{\frac{2\pi \im}{b} l_2} - 1\right|}, \] \label{prp_3}
	\item if $j = (j_1,-1)$ with $j_1 \in \N_0$ and $j_1 < n$ then $\left| \mu_{jml} \right| \leq c \, b^{-n-j_1}$ for some constant $c > 0$ (independent of $j_1$ and $n$), \label{prp_4}
	\item if $j = (-1,j_2)$ with $j_2 \in \N_0$ and $j_2 < n$ then $\left| \mu_{jml} \right| \leq c \, b^{-n-j_2}$ for some constant $c > 0$ (independent of $j_2$ and $n$), \label{prp_4a}
	\item if $j = (j_1,-1)$ with $j_1 \in \N_0$ and $j_1 \geq n$ then
\[ \left| \mu_{jml} \right| = \frac{1}{2}\frac{b^{-2j_1 - 1}}{\left|\e^{\frac{2\pi \im}{b} l_1} - 1 \right|}, \] \label{prp_5}
  \item if $j = (-1,j_2)$ with $j_2 \in \N_0$ and $j_2 \geq n$ then
\[ \left| \mu_{jml} \right| = \frac{1}{2}\frac{b^{-2j_2 - 1}}{\left|\e^{\frac{2\pi \im}{b} l_2} - 1 \right|}, \] \label{prp_5a}
	\item $\left| \mu_{(-1,-1),(0,0),(1,1)} \right| = | \frac{1}{4} b^{-2n} + \left( \frac{1}{2} + (2a_n - n) \frac{b - b^{-1}}{12} \right) b^{-n} |.$
\end{enumerate}

\end{prp}

\begin{proof}

Let $j \in \N_{-1}^2$ such that, $j_1 \geq n$ or $j_2 \geq n$. Then there is no point of $\Rn$ which is contained in the interior of the $b$-adic interval $I_{jm}$. Thereby \eqref{prp_3}, \eqref{prp_5} and \eqref{prp_5a} follow from Lemma \ref{lem_haar_coeff_besov_x_ham} and Lemma \ref{lem_haar_coeff_besov_indicator_ham}.

The set $\Rn$ contains $N = b^n$ points and, for fixed $j \in \N_{-1}^2$, the interiors of the $b$-adic intervals $I_{jm}$ are mutually disjoint. Therefore, there are no more than $b^n$ $b$-adic intervals which contain a point of $\Rn$. This gives us the second part of \eqref{prp_2}. The first part of \eqref{prp_2} follows from Lemma \ref{lem_haar_coeff_besov_x_ham} and Lemma \ref{lem_haar_coeff_besov_indicator_ham} because the remaining intervals contain exactly one point of $\Rn$ and the terms in the brackets in Lemma \ref{lem_haar_coeff_besov_indicator_ham} can be estimated because of \eqref{brackets_1} and \eqref{brackets_2}.

The part \eqref{prp_1} follows from Lemmas \ref{lem_haar_coeff_besov_x_ham}, \ref{lem_haar_coeff_besov_indicator_ham} and \ref{lem_coeff_calc}.

The last part is actually Proposition \ref{prp_minus1}.

Finally \eqref{prp_4} (and analogously \eqref{prp_4a}) follows from Lemma \ref{lem_middle} combined with Lemma \ref{lem_haar_coeff_besov_x_ham} and Lemma \ref{lem_haar_coeff_besov_indicator_ham}. We get
\[ \left| \mu_{jml} \right| = \left| \frac{b^{-n-j_1-1}(w_{j_1} - \varepsilon \, b^{n-j_1} \, (\e^{\frac{2\pi \im}{b} l_1} - 1))}{(\e^{\frac{2\pi \im}{b} l_1} - 1)^2} \pm \frac{b^{-2n}}{2(\e^{\frac{2\pi \im}{b} l_1} - 1)} \right| \]
where $w_{j_1}$ is either $\e^{\frac{2\pi \im}{b} l_1}$ or $-1$. Clearly,
\[ \left| w_{j_1} - \varepsilon \, b^{n-j_1} \, (\e^{\frac{2\pi \im}{b} l_1} - 1) \right| \leq c_1. \]
for some constant $c_1 > 0$ (independent of $j_1$ and $n$) since $\varepsilon b^{n - j_1} \leq b$. Hence, (using  $j_1 < n$)
\[ \left| \mu_{jml} \right| \leq c_2 \, b^{-n-j_1}. \]
\end{proof}

We are now ready to state and prove the main result of this subsection.

\begin{thm} \label{thm_hammersley_disc}

Let $1 \leq p, q \leq \infty$ and $0 \leq r < \frac{1}{p}$. Then there exists a constant $C > 0$ such that, for any $n \in \N$ and any generalized Hammersley type point set $\Rn$ with $a_n$ satisfying $|2a_n - n| \leq c_0$ for some constant $c_0 > 0$ (independent of $n$), we have
\[ \left\| D_{\Rn} | S_{pq}^r B(\Q^2) \right\| \leq C \, b^{n(r-1)} \, n^{\frac{1}{q}}. \]

\end{thm}

\begin{proof}

Let $\Rn$ be a generalized Hammersley type point set with $a_n$ satisfying $|2a_n - n| \leq c_0$ for some constant $c_0 \geq 0$. Let $\mu_{jml}$ be the $b$-adic Haar coefficients of the discrepancy function of $\Rn$. Theorem \ref{thm_besov_char} gave us an equivalent quasi-norm on $S_{pq}^r B(\Q^2)$ so that the proof of the inequality
\[ \left( \sum_{j \in \N_{-1}^2} b^{(j_1 + j_2)(r - \frac{1}{p} + 1) q} \left( \sum_{m \in \Dd_j, \, l \in \B_j} |\mu_{jml}|^p \right)^{\frac{q}{p}} \right)^{\frac{1}{q}} \leq C \, b^{n (r-1)} n^{\frac{1}{q}} \]
for some constant $C > 0$ establishes the proof of the theorem (for better readability we give a slightly different form with $j_1 + j_2$ in the exponent of $b$ which can be estimated easily).

We use different parts of Proposition \ref{prp_haar_coeff} after having split the sum by Minkowski's inequality. We have
{\allowdisplaybreaks\begin{align*}
& \left( \sum_{j\in\N_0^2; \, j_1 + j_2 < n-1} b^{(j_1 + j_2)(r - \frac{1}{p} + 1) q} \left( \sum_{m \in \Dd_j, \, l \in \B_j} | \mu_{jml}|^p \right)^{\frac{q}{p}} \right)^{\frac{1}{q}}\\
& \leq c_1 \left( \sum_{j\in\N_0^2; \, j_1 + j_2 < n-1} b^{(j_1 + j_2)(r - \frac{1}{p} + 1) q} \left( \sum_{m \in \Dd_j} b^{-2np} \right)^{\frac{q}{p}} \right)^{\frac{1}{q}}\\
& = c_1 \left( \sum_{j\in\N_0^2; \, j_1 + j_2 < n-1} b^{\left[ (j_1 + j_2)(r + 1) - 2n \right] q} \right)^{\frac{1}{q}}\\
& = c_1 \left( \sum_{\lambda = 0}^{n-2} b^{\left[ \lambda (r+1) - 2n \right] q} (\lambda + 1) \right)^{\frac{1}{q}}\\
& \leq c_1 \, n^{\frac{1}{q}} \left( \sum_{\lambda = 0}^{n-2} b^{\left[ \lambda (r+1) - 2n \right] q} \right)^{\frac{1}{q}}\\
& \leq c_2 \, n^{\frac{1}{q}} b^{n (r-1)}
\end{align*}}
from \eqref{prp_1} of Proposition \ref{prp_haar_coeff}. From \eqref{prp_2} of the same proposition we have (using the fact that $\frac{1}{p} - r > 0$)
{\allowdisplaybreaks\begin{align*}
& \left( \sum_{0 \leq j_1, j_2 \leq n; \, j_1 + j_2 \geq n-1} b^{(j_1 + j_2)(r - \frac{1}{p} + 1)q} \left( \sum_{m \in \Dd_j, \, l \in \B_j} | \mu_{jml}|^p \right)^{\frac{q}{p}} \right)^{\frac{1}{q}}\\
& \leq c_3 \left( \sum_{0 \leq j_1, j_2 \leq n; \, j_1 + j_2 \geq n-1} b^{(j_1 + j_2)(r - \frac{1}{p} + 1) q} \, b^{n \frac{q}{p}} \, b^{(-n - j_1 - j_2) q} \right)^{\frac{1}{q}}\\
& \quad + c_4 \left( \sum_{0 \leq j_1, j_2 \leq n; \, j_1 + j_2 \geq n-1} b^{(j_1 + j_2)(r - \frac{1}{p} + 1)q} \, b^{(j_1 + j_2)\frac{q}{p}} \, b^{(-2j_1 - 2j_2) q} \right)^{\frac{1}{q}}\\
& = c_3 \left( \sum_{0 \leq j_1, j_2 \leq n; \, j_1 + j_2 \geq n-1} b^{\left[ (j_1 + j_2) (r - \frac{1}{p}) + \frac{n}{p} - n \right] q} \right)^{\frac{1}{q}}\\
& \quad + c_4 \left( \sum_{0 \leq j_1, j_2 \leq n; \, j_1 + j_2 \geq n-1} b^{(j_1 + j_2)(r - 1)q} \right)^{\frac{1}{q}}\\
& = c_3 \left( \sum_{\lambda = n-1}^{2n} (2n - \lambda + 1) b^{\left[ \lambda(r - \frac{1}{p}) + \frac{n}{p} - n \right] q} \right)^{\frac{1}{q}}\\
& \quad + c_4 \left( \sum_{\lambda = n-1}^{2n} (2n - \lambda + 1) b^{\lambda (r-1) q} \right)^{\frac{1}{q}}\\
& = c_3 \, b^{\frac{n}{p} - n} \left( \sum_{\lambda = 1}^{n + 2} \lambda b^{\left[ (2n + 1 - \lambda)(r-\frac{1}{p}) \right] q} \right)^{\frac{1}{q}} + c_4 \left( \sum_{\lambda = 1}^{n+2} \lambda b^{(2n + 1 - \lambda) (r-1) q} \right)^{\frac{1}{q}}\\
& \leq c_5 \, b^{n(r-1) + n(r-\frac{1}{p})} \left( \sum_{\lambda = 1}^{n + 2} \lambda b^{\lambda (\frac{1}{p} - r) q} \right)^{\frac{1}{q}} + c_6 \, b^{2n (r-1)} \left( \sum_{\lambda = 1}^{n+2} \lambda b^{\lambda (1-r) q} \right)^{\frac{1}{q}}\\
& \leq c_5 \, b^{n(r-1) + n(r-\frac{1}{p})} (n + 2)^{\frac{1}{q}} \, b^{(n+3) (\frac{1}{p} - r)} + c_6 \, b^{2n (r-1)} (n + 2)^{\frac{1}{q}} \, b^{(n+3) (1 - r)}\\
& \leq c_7 \, b^{n(r-1)} \, n^{\frac{1}{q}}.
\end{align*}}
Part \eqref{prp_3} of Proposition \ref{prp_haar_coeff} gives us (using the fact that $r - 1 \leq 0$)
\begin{align*}
& \left( \sum_{j \in \N_0^2; \, j_1 \geq n} b^{(j_1 + j_2)(r - \frac{1}{p} + 1) q} \left( \sum_{m \in \Dd_j, \, l \in \B_j} | \mu_{jml}|^p \right)^{\frac{q}{p}} \right)^{\frac{1}{q}}\\
& \leq c_8 \left( \sum_{j \in \N_0^2; \, j_1 \geq n} b^{(j_1 + j_2)(r - \frac{1}{p} + 1) q} \, b^{(-2j_1 - 2j_2)q} \, b^{(j_1 + j_2)\frac{q}{p}} \right)^{\frac{1}{q}}\\
& = c_8 \left( \sum_{\lambda = n}^\infty (\lambda + 1) b^{\lambda(r - 1) q} \right)^{\frac{1}{q}}\\
& \leq c_9 \, n^{\frac{1}{q}} b^{n(r-1)}
\end{align*}
and an analogous result for those $j \in \N_0^2$ with $j_2 \geq n$. From \eqref{prp_4} of Proposition \ref{prp_haar_coeff} we conclude
\begin{align*}
& \left( \sum_{0 \leq j_1 < n; \, j_2 = -1} b^{(j_1 + j_2)(r - \frac{1}{p} + 1) q} \left( \sum_{m \in \Dd_j, \, l \in \B_j} | \mu_{jml}|^p \right)^{\frac{q}{p}} \right)^{\frac{1}{q}}\\
& \leq c_{10} \left( \sum_{0 \leq j_1 < n; \, j_2 = -1} b^{(j_1 + j_2)(r - \frac{1}{p} + 1) q} \, b^{(j_1 + j_2) \frac{q}{p}} \, b^{(-n - j_1) q} \right)^{\frac{1}{q}}\\
& = c_{11} \, b^{-n} \left( \sum_{j_1 = 0}^{n-1} b^{j_1 q r} \right)^{\frac{1}{q}}\\
& \leq c_{11} \, b^{-n} b^{nr} \\
& = c_{11} \, b^{n(r-1)} \\
& \leq c_{11} \, b^{n(r-1)} n^{\frac{1}{q}}.
\end{align*}
Analogously one estimates the sum for those $j \in \N_{-1}^2$ with $j_1 = -1$ and $0 \leq j_2 < n$. From \eqref{prp_5} of Proposition \ref{prp_haar_coeff} we have
\begin{align*}
& \left( \sum_{n \leq j_1; \, j_2 = -1} b^{(j_1 + j_2)(r - \frac{1}{p} + 1) q} \left( \sum_{m \in \Dd_j, \, l \in \B_j} | \mu_{jml}|^p \right)^{\frac{q}{p}} \right)^{\frac{1}{q}}\\
& \leq c_{12} \left( \sum_{n \leq j_1; \, j_2 = -1} b^{(j_1 + j_2)(r - \frac{1}{p} + 1) q} \, b^{(j_1 + j_2) \frac{q}{p}} \, b^{-2j_1 q} \right)^{\frac{1}{q}}\\
& = c_{13} \left( \sum_{j_1 = n}^\infty b^{j_1 (r-1) q} \right)^{\frac{1}{q}}\\
& \leq c_{13} \, b^{n(r-1)} \\
& \leq c_{13} \, b^{n(r-1)} n^{\frac{1}{q}}
\end{align*}
again with analogous results for the sum with those $j \in \N_{-1}^2$ where $j_1 = -1$ and $n \leq j_2$. In the cases where $p = \infty$ or $ = \infty$ the calculations have to be modified in the usual way. Finally, the last part of Proposition \ref{prp_haar_coeff} gives us (using $|2a_n - n| \leq c_0$)
\[ |\mu_{(-1,-1),(0,0),(1,1)}| \leq c_{14} b^{-n} \leq c_{14} b^{n(r-1)} n^{\frac{1}{q}}. \]
And the theorem is proved.

\end{proof}

\begin{rem}

We already have mentioned that in \cite{Hi10} Hinrichs used point sets with $a_n = \left\lfloor \frac{n}{2} \right\rfloor$. So a possible value for $c_0$ in that case would be $1$.

\end{rem}

Analogously to the last section we want to take advantage of the embeddings and get results for the Triebel-Lizorkin and the Sobolev spaces.

\begin{cor} \label{cor_f}

Let $1 \leq p, q < \infty$ and $0 \leq r < \frac{1}{\max(p,q)}$. Then there exists a constant $C > 0$ such that, for any $n \in \N$ and any generalized Hammersley type point set $\Rn$ with $a_n$ satisfying $|2a_n - n| \leq c_0$ for some constant $c_0 > 0$ (independent of $n$), we have
\[ \left\| D_{\Rn} | S_{pq}^r F(\Q^2) \right\| \leq C \, b^{n(r-1)} \, n^{\frac{1}{q}}. \]

\end{cor}

\begin{proof}
From Corollary \ref{cor_emb_BF} we have $S_{\max(p,q), q}^r B(\Q^2) \hookrightarrow S_{pq}^r F(\Q^2)$. Therefore, we get the assertion for $0 \leq r < \frac{1}{\max(p,q)}$ from Theorem \ref{thm_hammersley_disc}.

\end{proof}

\begin{cor} \label{cor_h}

Let $1 \leq p < \infty$ and $0 \leq r < \frac{1}{\max(p,2)}$. Then there exists a constant $C > 0$ such that, for any $n \in \N$ and any generalized Hammersley type point set $\Rn$ with $a_n$ satisfying $|2a_n - n| \leq c_0$ for some constant $c_0 > 0$ (independent of $n$), we have
\[ \left\| D_{\Rn} | S_p^r H(\Q^2) \right\| \leq C \, b^{n(r-1)} \, n^{\frac{1}{2}}. \]

\end{cor}

\begin{proof}
The assertion follows from Corollary \ref{cor_f} for $q = 2$.

\end{proof}

\begin{rem}

We recall that $S_p^0 H(\Q^2) = L_p(\Q^2)$, therefore, the point sets $\Rn$ have best possible $L_p$-discrepancy.

\end{rem}

%% file: BestDiscrepancySpqrB.tex
\subsection{Discrepancy of Chen-Skriganov type point sets}
We begin this subsection with the definition of point sets of Chen-Skriganov type, first suggested in \cite{CS02}. Those point sets are $b$-adic constructions and digital nets. Therefore, we are now perfectly prepared to work with them. The reader is referred to \cite{DP10} for more information than can be found in this work. Chen and Skriganov constructed those sets as an example for point sets with best possible $L_2$-discrepancy. In \cite{S06} Skriganov proved that they also have best possible $L_p$-discrepancy. Our goal is to analyze their discrepancy in spaces with dominating mixed smoothness with focus mainly on spaces $S_{pq}^r B (\Q^d)$.

We will repeat some notation that already has been introduced. An element $A$ from $\Fb_b^{dn}$ for $b$ prime will be given in the form $A = (a_1, \ldots, a_d)$ and for each $i$ we have $a_i = (a_{i 1}, \ldots, a_{i n}) \in \Fb_b^n$. The mapping $\Phi_n^d: \, \Fb_b^{dn} \rightarrow \Q^d$ is defined by Definition \ref{df_mapping_Phi} as $\Phi_n^d(A) = (\Phi_n(a_1), \ldots, \Phi_n(a_d))$ and
\[ \Phi_n(a_i) = \frac{a_{i 1}}{b} + \ldots + \frac{a_{i n}}{b^n}. \]
Finally, $v_n(a_i) = \max \left\{ \nu: \, a_{i \nu} \neq 0 \right\}$.

Let $b \geq 2d^2$ be a prime number and $n \in \N$ divisible by $2d$, i.e. $n = 2dw$ for some $w \in \N$. For some positive integer $h$ let
\[ f(z) = f_0 + f_1 z + \ldots + f_{h - 1} z^{h - 1} \]
be a polynomial in $\Fb_b[z]$. Its degree is $\dg(f) = h - 1$, assuming $f_{h - 1} \neq 0$ and $\dg(0) = 0$. For every $\lambda \in \N$ the $\lambda$-th hyper-derivative is
\[ \partial^{\lambda} f(z) = \sum_{i=0}^{h - 1} \binom{i}{\lambda} f_{\lambda} z^{i - \lambda}. \]

We use the usual convention for the binomial coefficient modulo $b$ that $\binom{i}{\lambda} = 0$ whenever $\lambda > i$. There are $2d^2$ distinct elements $\beta_{i,\nu} \in \Fb_b$, $1 \leq i \leq d$, and $1 \leq \nu \leq 2d$. For $1 \leq i \leq d$ let
\[ a_i(f) = \left( \left( \partial^{\lambda - 1} f(\beta_{i,\nu}) \right)_{\lambda = 1}^w \right)_{\nu=1}^{2d} \in \Fb_b^n. \]
We define $\C_n \subset \Fb_b^{dn}$ as
\[ \C_n = \left\{ A(f) = (a_1(f), \ldots, a_d(f)): \, f \in \Fb_b[z], \, \dg(f) < n \right\}. \]
Since there are $b^n$ polynomials with $\dg(f) < n$ in $\Fb_b[z]$ and $A(f) \neq A(g)$ if $f \neq g$, $\C_n$ has exactly $b^n$ elements. The set of polynomials in $\Fb_b[z]$ with $\dg(f) < n$ is closed under addition and scalar multiplication over $\Fb_b$ and $A \, : \, \Fb_b[z] \longrightarrow \Fb_b[z]$ is linear, hence $\C_n$ is an $\Fb_b$-linear subspace of $\Fb_b^{dn}$. Instead of working with point sets directly, we will work with such $\Fb_b$-linear subspaces of $\Fb_b^{dn}$ and use their duality properties. The point set which is a dual counterpart of such a subspace can be obtained through the mapping $\Phi_n^d$.
\begin{df}
Let $\C_n$ be as above. Then the Chen-Skriganov type point set is the set
\[ \CS_n = \Phi_n^d(\C_n). \]
\end{df}
The set $\CS_n$ contains exactly $b^n$ points. The following is the main result of this work and we also refer to \cite{M13b}.

\begin{thm} \label{thm_main}

Let $1 \leq p,q \leq \infty$. Let $0 < r < \frac{1}{p}$. Then there exists a constant $C > 0$ such that, for any integer $N \geq 2$, there exists a point set $\P$ in $\Q^d$ with $N$ points such that
\[ \left\| D_{\P} | S_{pq}^r B(\Q^d) \right\| \leq C \, N^{r-1} \, (\log N)^{\frac{d-1}{q}}. \]

\end{thm}

We will prove it later because we need to do some preliminary work first.

\begin{rem}

The point sets in the theorem are the Chen-Skriganov point sets. It was conjectured in \cite{Hi10} that they might satisfy the desired upper bound. The restrictions for the parameter $r$ are necessary. The upper bound $r < \frac{1}{p}$ is due to the fact that we need characteristic functions of intervals to belong to $S_{pq}^r B ([0,1)^d)$ and the condition given by \cite[Theorem 6.3]{T10a}. The restriction $r \geq 0$ comes from the point sets. Anyway, there is a restriction of $r > \frac{1}{p} - 1$ from the fact that we require $S_{pq}^r B ([0,1)^d)$ to have a $b$-adic Haar basis (see Theorem \ref{thm_besov_char}). We have an additional restriction $r > 0$ which is due to our estimations which might not be optimal.

\end{rem}

The proof of Theorem \ref{thm_main} will work as follows. The discrepancy function can be partitioned as $D_{\P} = \Theta_{\P} + R_{\P}$ where $\Theta_{\P}$ is obtained by truncating Walsh series expansions and $R_{\P}$ is the rest. Because of the special properties of $\CS_n$, $R_{\CS_n}$ is pointwise small enough. We can estimate the $b$-adic Haar coefficients of $\Theta_{\CS_n}$ and use the characterization of the norm of $S_{pq}^r B ([0,1)^d)$ in terms of $b$-adic Haar bases given by Theorem \ref{thm_besov_char}.

Again we conclude results for other spaces with dominating mixed smoothness.

\begin{cor} \label{cor_f_spaces}

Let $1 \leq p,q < \infty$. Let $0 < r < \frac{1}{\max(p,q)}$. Then there exists a constant $C > 0$ such that, for any integer $N \geq 2$, there exists a point set $\P$ in $\Q^d$ with $N$ points such that
\[ \left\| D_{\P} | S_{pq}^r F(\Q^d) \right\| \leq C \, N^{r-1} \, (\log N)^{\frac{d-1}{q}}. \]

\end{cor}

\begin{proof}
From Corollary \ref{cor_emb_BF} we have $S_{\max(p,q), q}^r B(\Q^d) \hookrightarrow S_{pq}^r F(\Q^d)$. Therefore, we get the assertion for $0 < r < \frac{1}{\max(p,q)}$ from Theorem \ref{thm_main}.

\end{proof}

\begin{cor} \label{cor_h_spaces}

Let $1 \leq p < \infty$. Let $0 \leq r < \frac{1}{\max(p,2)}$. Then there exists a constant $C > 0$ such that, for any integer $N \geq 2$, there exists a point set $\P$ in $\Q^d$ with $N$ points such that
\[ \left\| D_{\P} | S_p^r H(\Q^d) \right\| \leq C \, N^{r-1} \, (\log N)^{\frac{d-1}{2}}. \]

\end{cor}

\begin{proof}
Let $r > 0$. Then the assertion follows from Corollary \ref{cor_f_spaces} for $q = 2$. We recall that $S_p^0 H(\Q^d) = L_p(\Q^d)$, therefore, the assertion in the case $r = 0$ is Theorem \ref{thm_L1_upper} and Theorem \ref{thm_upper_chen}.

\end{proof}

The next result is \cite[Theorem 16.28]{DP10}.

\begin{prp} \label{CS_to_be_digital_net}

For every $w \in \N$ the set $\C_n$ is an $\Fb_b$-linear subspace of $\Fb_b^{dn}$ of dimension $n$. Its dual space $\C_n^{\perp}$ has dimension $nd - n$ and it satisfies
\[ \varkappa_n(\C_n^{\perp}) \geq 2d + 1 \text{ and } \delta_n(\C_n^{\perp}) \geq n + 1. \]

\end{prp}

\begin{thm} \label{CS_is_digital_net}

The Chen-Skriganov type point set $\CS_n$ is a digital $(0,n,d)$-net in base $b$.

\end{thm}

\begin{proof}
This result is a direct consequence of Propositions \ref{CS_to_be_digital_net} and \ref{prp_dig_net_duality}.

\end{proof}

Before we turn to the computation of the Haar coefficients of the discrepancy function of the points sets $\CS_n$, we give some very easy lemmas which already have been stated in less generality, see Lemmas \ref{lem_haar_coeff_besov_x_1}, \ref{lem_haar_coeff_besov_indicator_1}, \ref{lem_factor5}, \ref{lem_haar_coeff_besov_x_ham}. We recall the notation that has been given in the beginning of Subsection \ref{notation_eta}. By $0 \leq s \leq d$ we denote the number of coordinates of $j \in \N_{-1}^d$ which are not $-1$ and by $j_{\eta_i}$ for $1 \leq i \leq s$ we denote such coordinates of $j$ which are not $-1$. We write $|j| = j_{\eta_1} + \ldots + j_{\eta_s}$.

\begin{lem} \label{lem_haar_coeff_besov_x}

Let $f(x) = x_1 \cdot \ldots \cdot x_d$ for $x=(x_1,\ldots,x_d) \in \Q^d$. Let $j \in \N_{-1}^d, \, m \in \Dd_j, l \in \B_j$ and let $\mu_{jml}$ be the $b$-adic Haar coefficient of $f$. Then
\[ \mu_{jml} = \frac{b^{-2|j| - s}}{2^{d-s}(\e^{\frac{2\pi \im}{b} l_{\eta_1}} - 1) \cdot \ldots \cdot (\e^{\frac{2\pi \im}{b} l_{\eta_s}} - 1)}, \]
and therefore,
\[ |\mu_{jml}| \leq c \, b^{-2|j|} \]
with a constant $c > 0$.

\end{lem}

\begin{lem} \label{lem_haar_coeff_besov_indicator}

Let $z = (z_1,\ldots,z_d) \in \Q^d$ and $g(x) = \chi_{[0,x)}(z)$ for $x = (x_1, \ldots, x_d) \in \Q^d$. Let $j \in \N_{-1}^d, \, m \in \Dd_j, l \in \B_j$ and let $\mu_{jml}$ be the $b$-adic Haar coefficient of $g$. Then $\mu_{jml} = 0$ whenever $z$ is not contained in the interior of the $b$-adic interval $I_{jm}$ supporting the functions $h_{jml}$. If $z$ is contained in the interior of $I_{jm}$ then there is a unique $k = (k_1,\ldots,k_d)$ with $k_i \in \{0,1,\ldots,b-1\}$ if $j_i \neq -1$ or $k_i = -1$ if $j_i = -1$ such that, $z$ is contained in $I_{jm}^k$. Then 
	      \begin{multline*}
        \mu_{jml} = b^{-|j| - s} \prod_{1 \leq i \leq d; \, j_i = -1}(1-z_i) \times\\
        \times \prod_{\nu = 1}^s \left[ (bm_{\eta_{\nu}}+k_{\eta_{\nu}}+1-b^{j_{\eta_{\nu}}+1}z_{\eta_{\nu}}) \e^{\frac{2\pi \im}{b}k_{\eta_{\nu}} l_{\eta_{\nu}}} + \sum_{r_{\eta_{\nu}} = k_{\eta_{\nu}}+1}^{b-1} \e^{\frac{2\pi \im}{b}r_{\eta_{\nu}} l_{\eta_{\nu}}}, \right]
        \end{multline*}
and therefore,
\[ |\mu_{jml}| \leq c \, b^{-|j|} \]
with a constant $c > 0$.

\end{lem}

One easily calculates the one-dimensional case and concludes via tensor products.

\begin{lem} \label{lem_lambda_s_minus_1}

Let $\lambda \in \N_0$ and $s \in \N$. Then
\[ \# \left\{ (j_1, \ldots, j_s) \in \N_0^s: \, j_1 + \ldots + j_s = \lambda \right\} \leq (\lambda + 1)^{s-1}. \]

\end{lem}

\begin{proof}
For $s = 1$ the assertion is trivial. Inductively we get
\begin{align*}
& \# \left\{ (j_1, \ldots, j_{s+1}) \in \N_0^{s+1}: \, j_1 + \ldots + j_{s+1} = \lambda \right\} \\ 
& \qquad = \sum_{i = 0}^{\lambda} \# \left\{ (j_1, \ldots, j_s) \in \N_0^s: \, j_1 + \ldots + j_s = \lambda - i \right\} \\
& \qquad \leq \sum_{i = 0}^{\lambda} (\lambda - i + 1)^{s-1} \leq (\lambda + 1)^s
\end{align*}
\end{proof}

We consider the Walsh series expansion of the function $\chi_{[0,y)}$,
\begin{align}
\chi_{[0,y)}(x) = \sum_{t = 0}^\infty \hat{\chi}_{[0,y)}(t) \wal_t(x),
\end{align}
where for $t \in \N_0$ with $b$-adic expansion $t = \tau_0 + \tau_1 b + \ldots + \tau_{\varrho(t) - 1} b^{\varrho(t) - 1}$, the $t$-th Walsh coefficient is given by
\[ \hat{\chi}_{[0,y)}(t) = \int_0^1 \chi_{[0,y)}(x) \overline{\wal_t(x)} \dint x = \int_0^y \overline{\wal_t(x)} \dint x. \]
For $t > 0$ we put $t = t' + \tau_{\varrho(t) - 1} b^{\varrho(t) - 1}$.

\begin{lem} \label{chi_roof_lem}

Let $b \geq 2$ be an integer and $y \in \Q$. Then we have
\[ \hat{\chi}_{[0,y)}(0) = y = \frac{1}{2} + \sum_{a = 1}^\infty \sum_{z = 1}^{b-1} \frac{1}{b^a (\e^{-\frac{2 \pi \im}{b} z} - 1)} \wal_{z b^{a-1}}(y) \]
and for any integer $t > 0$ we have
\begin{multline*}
\hat{\chi}_{[0,y)}(t) = \frac{1}{b^{\varrho(t)}}\left( \frac{1}{1 - \e^{-\frac{2 \pi \im}{b} \tau_{\varrho(t) - 1}}} \overline{\wal_{t'}(y)} \right. +\\
+ \left( \frac{1}{\e^{-\frac{2 \pi \im}{b} \tau_{\varrho(t) - 1}} - 1} + \frac{1}{2} \right) \overline{\wal_t(y)} +\\
+ \left. \sum_{a = 1}^\infty \sum_{z = 1}^{b-1} \frac{1}{b^a (\e^{\frac{2 \pi \im}{b} z} - 1)} \overline{\wal_{z b^{\varrho(t)+a-1} + t}(y)} \right).
\end{multline*}

\end{lem}

This is called Fine-Price formulas and was first proved in \cite{F49} (dyadic case) and \cite{P57} ($b$-adic version). One often finds it in literature, e.g. see \cite[Lemma 14.8]{DP10} for an easy understandable proof.

For $n \in \N_0$ we consider the approximation of $\chi_{[0,y)}$ by the truncated series
\begin{align}
\chi_{[0,y)}^{(n)}(x) = \sum_{t = 0}^{b^n - 1} \hat{\chi}_{[0,y)}(t) \wal_t(x).
\end{align}
Now let $y = (y_1, \ldots, y_d) \in [0,1)^d$. Then we put
\[ \chi_{[0,y)}^{(n)}(x) = \prod_{i = 1}^d \chi_{[0,y_i)}^{(n)}(x_i) \]
where $x = (x_1, \ldots, x_d) \in [0,1)^d$ to approximate $\chi_{[0,y)}$. Let $N$ be a positive integer. Then we put for some point set $\P$ in $\Q^d$ with $N$ points
\begin{align}
\Theta_{\P}(y) = \frac{1}{N} \sum_{z \in \P} \chi_{[0,y)}^{(n)}(z) - y_1 \cdot \ldots \cdot y_d.
\end{align}
We partition the discrepancy function as
\begin{align} \label{split}
D_{\P}(y) = \Theta_{\P}(y) + R_{\P}(y)
\end{align}
with the main part $\Theta_{\P}$ and the rest $R_{\P}$ which will be handled separately. We now restrict ourselves again to the case where $b$ is prime. The reader might want to recall Definition \ref{df_d_prime}.

\begin{lem} \label{theta_lem}

Let $\C$ be an $\Fb_b$-linear subspace of $\Fb_b^{dn}$ of dimension $n$ and let $\P = \Phi_n^d(\C)$ denote the corresponding digital $(v,n,d)$-net in base $b$ with generating matrices $C_1, \ldots, C_d$. Then
\[ \Theta_{\P}(y) = \sum_{t \in \Dn'(C_1, \ldots, C_d)} \hat{\chi}_{[0,y)}(t). \]

\end{lem}

\begin{proof}
For $t = (t_1, \ldots, t_d) \in \N_0^d$ and $y = (y_1, \ldots, y_d) \in \Q^d$, we have
\[ \hat{\chi}_{[0,y)}(t) = \hat{\chi}_{[0,y_1)}(t_1) \cdot \ldots \cdot \hat{\chi}_{[0,y_d)}(t_d). \]
By Lemma \ref{lem_duality_into_disc} we get
\begin{align*}
\Theta_{\P}(y) & = \frac{1}{b^n} \sum_{z \in \P} \sum_{t_1, \ldots, t_d = 0}^{b^n - 1} \hat{\chi}_{[0,y)}(t) \wal_t(z) - \hat{\chi}_{[0,y)}((0, \ldots, 0)) \\
               & = \sum_{\substack{t_1, \ldots, t_d = 0 \\ (t_1, \ldots t_d) \neq (0, \ldots, 0)}}^{b^n - 1} \hat{\chi}_{[0,y)}(t) \frac{1}{b^n} \sum_{z \in \P} \wal_t(z) \\
               & = \sum_{t \in \Dn'(C_1, \ldots, C_d)} \hat{\chi}_{[0,y)}(t).
\end{align*}
\end{proof}

The next results give us pointwise estimates for the rest term, so that only the main term $\Theta_{\P}$ needs to be considered.

\begin{lem} \label{rest_lem}

There exists a constant $c > 0$ such that, for any $n \in \N_0$ and for any $\Fb_b$-linear subspace $\C$ of $\Fb_b^{dn}$ of dimension $n$ with dual space $\C^{\perp}$ satisfying $\delta_n(\C^{\perp}) \geq n + 1$ with the corresponding digital $(0,n,d)$-net $\P = \Phi_n^d(\C)$ and for every $y \in [0,1)^d$, we have
\[ |R_{\P}(y)| \leq c \, b^{-n}. \]

\end{lem}

The fact that $\P$ is a $(0,n,d)$-net is a consequence of Proposition \ref{prp_dig_net_duality}. For a proof of this lemma the interested reader is referred to \cite[Lemma 16.21]{DP10}.

We recall the following notation. For functions $f, g \in L_2(\Q^d)$ we write
\[ \left\langle f, g \right\rangle = \int_{\Q^d} f \, \bar{g}. \]

\begin{prp} \label{prp_min1}

Let $j = (-1, \ldots, -1), \, m = (0, \ldots, 0), \, l = (1, \ldots, 1)$. Then there exists a constant $c > 0$ independent of $n$ such that,
\[ |\mu_{jml}(D_{\CS_n})| \leq c \, b^{-n}. \]

\end{prp}

\begin{proof}
As in \eqref{split} we partition $D_{\CS_n}(y) = \Theta_{\CS_n}(y) + R_{\CS_n}(y)$ and we know from Proposition \ref{CS_to_be_digital_net} and Lemma \ref{rest_lem} that there exists a constant $c > 0$ such that, $|R_{\CS_n}(y)| \leq c \, b^{-n}$. Using Lemma \ref{theta_lem} we can calculate the Haar coefficient
\[ \mu_{jml}(D_{\CS_n}) = \left\langle \Theta_{\CS_n} + R_{\CS_n} , h_{jml} \right\rangle. \]
To do so we use the fact that $h_{jml} = \wal_{(0, \ldots, 0)}$. Now we consider the one-dimensional case first. From the first part of Lemma \ref{chi_roof_lem} we get
\[ \left\langle \hat{\chi}_{[0,\cdot)}(0),\wal_0 \right\rangle = \frac{1}{2}. \]
Now let $t > 0$. Then from the second part of Lemma \ref{chi_roof_lem} we have
\[ \left\langle \hat{\chi}_{[0,\cdot)}(t),\wal_0 \right\rangle = \begin{cases}
\frac{1}{b^{\varrho(t)}} \frac{1}{1 - \e^{-\frac{2 \pi \im}{b} \tau_{\varrho(t)-1}}} & t' = 0, \\
0 & t' \neq 0.
\end{cases} \]
This means that we can find a constant $c_1 > 0$ such that, for any integer $t \geq 0$ we have
\[ \left| \left\langle \hat{\chi}_{[0,\cdot)}(t),\wal_0 \right\rangle \right| \leq c_1 \, b^{-\varrho(t)} \]
and
\[ \left\langle \hat{\chi}_{[0,\cdot)}(t),\wal_0 \right\rangle = 0 \]
if $t > 0$ and $t' \neq 0$.

Now suppose, we have some $t \in \Dn'(C_1, \ldots, C_d)$ such that,
\[ \left\langle \hat{\chi}_{[0,\cdot)}(t), \wal_{(0, \ldots, 0)} \right\rangle \neq 0. \]
Then for all $1 \leq i \leq d$ we have
\[ \left\langle \hat{\chi}_{[0,\cdot)}(t_i), \wal_0 \right\rangle \neq 0. \]
Then necessarily $t_i = \tau_{\varrho(t_i)-1} \, b^{\varrho(t_i)-1}$ (since $t_i' = 0$) or $t_i = 0$ for any $i = 1, \ldots, d$ which means that either $\varkappa(t_i) = 1$ or $\varkappa(t_i) = 0$. In any case we have $\varkappa^d(t) \leq d$ which is a contradiction to $\varkappa_n(\mathcal{C}^{\bot}_n) \geq 2d + 1$ as must be the case according to Proposition \ref{CS_to_be_digital_net}. Therefore, for all $t \in \Dn'(C_1, \ldots, C_d)$ we have
\[ \left\langle \hat{\chi}_{[0,\cdot)}(t), \wal_{(0, \ldots, 0)} \right\rangle = 0 \]
and from Lemma \ref{theta_lem} follows $\left\langle \Theta_{\CS_n}, \wal_{(0, \ldots, 0)} \right\rangle = 0$.
Hence we have
\[ |\mu_{jml}(D_{\CS_n})| \leq |\left\langle \Theta_{\CS_n} , \wal_{(0, \ldots, 0)} \right\rangle| + |\left\langle R_{\CS_n} , \wal_{(0, \ldots, 0)} \right\rangle| \leq c \, b^{-n}. \]
\end{proof}

\begin{lem} \label{lem_scalprod_1}

Let $j \in \N_{-1}$, $m \in \Dd_j$, $l \in \B_j$ and $\alpha \in \N_0$. Then
\begin{enumerate}[(i)]
  \item if $j \in \N_0$ and $\varrho(\alpha) = j + 1$ and $\alpha_{j} = l$ then
\[ |\left\langle h_{jml} , \wal_{\alpha} \right\rangle| = b^{-j}, \]
  \item if $j = -1, \, m = l = 0$ and $\alpha = 0$ then
\[ |\left\langle h_{jml} , \wal_{\alpha} \right\rangle| = 1, \]
  \item if $\varrho(\alpha) \neq j + 1$ or $\alpha_{j} \neq l$ then
\[ |\left\langle h_{jml} , \wal_{\alpha} \right\rangle| = 0. \] \label{lem_haar_wal_3}
\end{enumerate}

\end{lem}

\begin{proof}
The second claim and the third for $j = -1$ are trivial so let $j \geq 0$. Let $y \in [0,1)$. We expand $\alpha$ and $y$ as
\[ \alpha = \alpha_0 + \alpha_1 b + \ldots + \alpha_{\varrho(\alpha)-1} b^{\varrho(\alpha)-1} \]
and
\[ y = y_1 b^{-1} + y_2 b^{-2} + \ldots. \]
Hence
\[ \wal_{\alpha} (y) = \e^{\frac{2 \pi \im}{b} (\alpha_0 y_1 + \ldots + \alpha_{\varrho(\alpha)-1} y_{\varrho(\alpha)})}. \]
The function $\wal_{\alpha}$ is constant on the intervals
\[ \left[\right.b^{-\varrho(\alpha)} \delta , b^{-\varrho(\alpha)} (\delta + 1)\left.\right) \]
for any integer $0 \leq \delta < b^{\varrho(\alpha)}$ according to Proposition \ref{prp_constant_walsh}. The function $h_{jml}$ is constant on the intervals
\[ I_{jm}^k = \left[\right.b^{-j-1} (bm+k) , b^{-j-1} (bm+k+1)\left.\right) \]
for any integer $0 \leq k < b$. Now suppose that either $j+1 > \varrho(\alpha)$ or $j+1 < \varrho(\alpha)$. This would mean that either
\[ I_{jm} = \left[\right. b^{-j} m , b^{-j} (m+1)\left.\right) \subseteq \left[\right.b^{-\varrho(\alpha)} \delta , b^{-\varrho(\alpha)} (\delta + 1)\left. \right) \]
in the first case or
\[ \left[\right. b^{-\varrho(\alpha)} \delta , b^{-\varrho(\alpha)} (\delta + 1)\left. \right) \subset I_{jm}^k \]
for some $k$ in the second case or in both cases
\[ \left[\right. b^{-j} m , b^{-j} (m+1)\left. \right) \cap \left[\right. b^{-\varrho(\alpha)} \delta , b^{-\varrho(\alpha)} (\delta + 1)\left. \right) = \emptyset \]
In any case
\[ \left\langle h_{jml} , \wal_{\alpha} \right\rangle = 0. \]
Hence, \eqref{lem_haar_wal_3} is proved and the only remaining case is $j+1 = \varrho(\alpha)$. Then either again
\[ \left[\right. b^{-j} m , b^{-j} (m+1)\left.\right) \cap \left[\right.b^{-\varrho(\alpha)} \delta , b^{-\varrho(\alpha)} (\delta + 1)\left. \right) = \emptyset \]
or
\[ \left[\right. b^{-\varrho(\alpha)} \delta , b^{-\varrho(\alpha)} (\delta + 1)\left. \right) = I_{jm}^k \]
for some $k$. We consider the last possibility.
The value of $h_{jml}$ on $I_{jm}^k$ is $\e^{\frac{2 \pi \im}{b} l k}$. To calculate the value of $\wal_{\alpha}$ we expand $m$ as
\[ m = m_1 + m_2 b + \ldots + m_j b^{j-1}. \]
Clearly, $0 \leq b m + k < b^{j+1}$. Hence,
\[ b^{-j-1} (b m + k) = m_j b^{-1} + \ldots + m_2 b^{-j+1} + m_1 b^{-j} + k b^{-j-1}. \]
So,
\[ \wal_{\alpha} (b^{-j-1} (b m + k)) = \e^{\frac{2 \pi \im}{b} (\alpha_0 m_j + \ldots + \alpha_{j-1} m_1 + \alpha_j k)}. \]
Now we can calculate
\begin{align*}
\overline{\left\langle h_{jml} , \wal_{\alpha} \right\rangle} & = \int_{I_{jm}} \overline{h_{jml}(y)} \wal_{\alpha}(y) \dint y\\
& = \sum_{k=0}^{b-1} \int_{I_{jm}^k} \overline{h_{jml}(y)} \wal_{\alpha}(y) \dint y\\
& = b^{-j-1} \sum_{k=0}^{b-1} \e^{\frac{2 \pi \im}{b} (\alpha_0 m_j + \ldots + \alpha_{j-1} m_1 + (\alpha_j-l) k)}\\
& = b^{-j-1} \e^{\frac{2 \pi \im}{b} (\alpha_0 m_j + \ldots + \alpha_{j-1} m_1)} \sum_{k=0}^{b-1} \e^{\frac{2 \pi \im}{b} (\alpha_j-l) k}\\
& = \begin{cases}
b^{-j} \e^{\frac{2 \pi \im}{b} (\alpha_0 m_j + \ldots + \alpha_{j-1} m_1)} & \alpha_j = l,\\
0 & \alpha_j \neq l
\end{cases}
\end{align*}
and the lemma follows.

\end{proof}

\begin{lem} \label{lem_scalprod_2}

There exists a constant $c > 0$ with the following property. Let $t, \alpha \in \N_0$. Then if $\alpha = t'$ or $\alpha = t + \tau \, b^{\varrho(t) + a - 1}$ for some integers $0 \leq \tau \leq b - 1$ and $a \geq 1$ then
\[ \left| \left\langle \hat{\chi}_{[0,\cdot)}(t) , \wal_{\alpha} \right\rangle \right| \leq c \, b^{-\max(\varrho(t), \varrho(\alpha))}. \]
If $\alpha \neq t'$ and there are no integers $0 \leq \tau \leq b - 1$ and $a \geq 1$ such that, $\alpha = t + \tau \, b^{\varrho(t) +a - 1}$, then
\[ \left\langle \hat{\chi}_{[0,\cdot)}(t) , \wal_{\alpha} \right\rangle = 0. \]

\end{lem}

\begin{proof}
We use Lemma \ref{chi_roof_lem}. First let $t > 0$. Suppose that $\alpha = t'$, so $\varrho(\alpha) < \varrho(t)$. Then
\[ \left| \left\langle \hat{\chi}_{[0,\cdot)}(t) , \wal_{\alpha} \right\rangle \right| = \left| \frac{1}{1 - \e^{\frac{-2 \pi \im}{b} \tau_{\varrho(t) - 1}}} \right| \, b^{-\varrho(t)} \leq c \, b^{-\varrho(t)}. \]
If $\alpha = t$ meaning that $\varrho(\alpha) = \varrho(t)$ then
\[ \left| \left\langle \hat{\chi}_{[0,\cdot)}(t) , \wal_{\alpha} \right\rangle \right| \leq \left| \frac{1}{\e^{\frac{-2 \pi \im}{b} \tau_{\varrho(t) - 1}} - 1} + \frac{1}{2} \right| \, b^{-\varrho(t)} \leq c \, b^{-\varrho(t)}. \]
Now let $\alpha = t + \tau \, b^{\varrho(t) +a - 1}$ for some $1 \leq \tau \leq b - 1$ and $a \geq 1$. Hence $\varrho(\alpha) = \varrho(t) + a$. Then
\[ \left| \left\langle \hat{\chi}_{[0,\cdot)}(t) , \wal_{\alpha} \right\rangle \right| = \left| \frac{1}{\e^{\frac{2 \pi \im}{b} \tau} - 1} \right| \, b^{-\varrho(t)} \, b^{-a} \leq c \, b^{-\varrho(\alpha)}. \]
For any other $\alpha$ clearly,
\[ \left\langle \hat{\chi}_{[0,\cdot)}(t) , \wal_{\alpha} \right\rangle = 0. \]
Now we consider the case $t = 0$. Then for $\alpha = 0$ (meaning $\varrho(\alpha) = 0$) we have
\[ \left| \left\langle \hat{\chi}_{[0,\cdot)}(t) , \wal_{\alpha} \right\rangle \right| = \frac{1}{2} \leq c \, b^{-\varrho(\alpha)}. \]
Let $\alpha = \tau \, b^{a - 1}$ for some $1 \leq \tau \leq b - 1$ and $a \geq 1$. Then $\varrho(\alpha) = a$ and
\[ \left| \left\langle \hat{\chi}_{[0,\cdot)}(t) , \wal_{\alpha} \right\rangle \right| = \left| \frac{1}{\e^{\frac{2 \pi \im}{b} \tau} - 1} \right| \, b^{-a} \leq c \, b^{-\varrho(\alpha)}. \]
For any other $\alpha$ again clearly,
\[ \left\langle \hat{\chi}_{[0,\cdot)}(t) , \wal_{\alpha} \right\rangle = 0. \]
\end{proof}

We now need an additional notation. For any function $f \, : \, \Fb_b^{dn} \longrightarrow \Cx$ we call $\hat{f}$ given by
\[ \hat{f}(B) = \sum_{A \in \Fb_b^{dn}} \e^{\frac{2 \pi \im}{b} A \cdot B} f(A) \]
for $B \in \Fb_b^{dn}$ the Walsh transform of $f$.

The following two facts can be found in \cite{DP10}. The first lemma is \cite[Lemma 16.9]{DP10} while the second is \cite[(16.3)]{DP10}.

\begin{lem} \label{lem_169dp10}

Let $\C$ and $\C^{\perp}$ be mutually dual $\Fb_b$-linear subspaces of $\Fb_b^{dn}$. Then for any function $f \, : \, \Fb_b^{dn} \longrightarrow \Cx$ we have
\[ \sum_{A \in \C} f(A) = \frac{\# \C}{b^{dn}} \sum_{B \in \C^{\perp}} \hat{f}(B). \]

\end{lem}

\begin{lem} \label{lem_walshduality}

Let $\C$ and $\C^{\perp}$ be mutually dual $\Fb_b$-linear subspaces of $\Fb_b^{dn}$. Let $B \in \Fb_b^{dn}$. Then we have
\[ \sum_{A \in \C} \e^{\frac{2 \pi \im}{b} A \cdot B} = \begin{cases} \# \C, & B \in \C^{\perp}, \\ 0, & B \notin \C^{\perp}. \end{cases} \]

\end{lem}

We will introduce some notation now, slightly changed from what can be found in \cite[16.2]{DP10}. Let $0 \leq \gamma_1, \ldots, \gamma_d \leq n$ be integers. We put $\gamma = (\gamma_1, \ldots, \gamma_d)$. Then we write
\[ \V_{\gamma} = \left\{ A \in \Fb_b^{dn} \, : \, \Phi_n(A) \in \prod_{i = 1}^d \left[ \left. 0, b^{-\gamma_i} \right) \right. \right\}. \]
Hence, $\V_{\gamma}$ consists of all such $A \in \Fb_b^{dn}$ that $a_i = (0, \ldots, 0, a_{i,\gamma_i + 1}, \ldots, a_{i n})$ for all $1 \leq i \leq d$. For all $1 \leq i \leq d$ let $0 \leq \lambda_i \leq \gamma_i$ be integers and let $\lambda = (\lambda_1, \ldots, \lambda_d)$. Then we write $\V_{\gamma, \lambda}$ for the set consisting of all such $A \in \Fb_b^{dn}$ that $a_i = (0, \ldots, 0, a_{i, \lambda_i + 1}, \ldots, a_{i, \gamma_i - 1}, 0, \linebreak a_{i, \gamma_i + 1}, \ldots, a_{i n})$. The case $\lambda_i = \gamma_i$ is to be understood in the obvious way as $a_i = (0, \ldots, 0, a_{i,\gamma_i + 1}, \ldots, a_{i n})$. Therefore, $\V_{\gamma}^{\perp}$ consists of such $A \in \Fb_b^{dn}$ that $a_i = (a_{i 1}, \ldots, a_{i, \gamma_i}, \linebreak 0, \ldots, 0)$ and $\V_{\gamma, \lambda}^{\perp}$ consists of such $A \in \Fb_b^{dn}$ that $a_i = (a_{i 1}, \ldots, a_{i, \lambda_i}, 0, \ldots, 0, a_{i, \gamma_i}, 0, \ldots, 0)$.

For a subset $V$ of $\Fb_b^{dn}$ we denote the characteristic function of $V$ by $\chi_V$. The next result is a slight generalization of the corresponding assertion from \cite[Lemma 16.11]{DP10}.

\begin{lem} \label{lem_1611dp10}

Let $\gamma_1, \ldots, \gamma_d, \lambda_1, \ldots, \lambda_d$ be as above. Let $\sigma$ be the number of such $i$ that $\lambda_i < \gamma_i$. For all $B \in \Fb_b^{dn}$ we have
\[ \hat{\chi}_{\V_{\gamma, \lambda}}(B) = b^{dn - |\lambda| - \sigma} \chi_{\V_{\gamma, \lambda}^{\perp}}(B). \]

\end{lem}

\begin{proof}
We use Lemma \ref{lem_walshduality} and obtain
\begin{align*}
\hat{\chi}_{\V_{\gamma, \lambda}}(B) & =\sum_{A \in \Fb_b^{dn}} \e^{\frac{2 \pi \im}{b} A \cdot B} \chi_{\V_{\gamma, \lambda}}(A) \\
& = \sum_{A \in \V_{\gamma, \lambda}} \e^{\frac{2 \pi \im}{b} A \cdot B} \\
& = \# \left( \V_{\gamma, \lambda} \right) \chi_{\V_{\gamma, \lambda}^{\perp}}(B) \\
& = b^{dn - |\lambda| - \sigma} \chi_{\V_{\gamma, \lambda}^{\perp}}(B).
\end{align*}
\end{proof}

The following fact is a generalization of \cite[Lemma 16.13]{DP10}.

\begin{lem} \label{1613dp10}

Let $\C$ and $\C^{\perp}$ be mutually dual $\Fb_b$-linear subspaces of $\Fb_b^{dn}$. Let $\gamma_1, \ldots, \gamma_d, \linebreak \lambda_1, \ldots, \lambda_d, \sigma$ be as above. Then we have
\[ \# \left( \C \cap \V_{\gamma, \lambda} \right) = \frac{\# \C}{b^{|\lambda| + \sigma}} \, \# \left( \C^{\perp} \cap \V_{\gamma, \lambda}^{\perp} \right). \]

\end{lem}

\begin{proof}
We use Lemmas \ref{lem_169dp10} and \ref{lem_1611dp10} and get
\begin{align*}
\# \left( \C \cap \V_{\gamma, \lambda} \right) & = \sum_{A \in \C} \chi_{\V_{\gamma, \lambda}}(A) \\
& = \frac{\# \C}{b^{dn}} \sum_{B \in \C^{\perp}} \hat{\chi}_{\V_{\gamma, \lambda}}(B) \\
& = \frac{\# \C}{b^{dn}} \sum_{B \in \C^{\perp}} b^{dn - |\lambda| - \sigma} \chi_{\V_{\gamma, \lambda}^{\perp}}(B) \\
& = \frac{\# \C}{b^{|\lambda| + \sigma}} \, \# \left( \C^{\perp} \cap \V_{\gamma, \lambda}^{\perp} \right).
\end{align*}
\end{proof}

The proof of the following fact is contained in the proof of \cite[Lemma 16.26]{DP10}.

\begin{lem} \label{main_est_lem_notenough}

Let $\C$ be an $\Fb_b$-linear subspace of $\Fb_b^{dn}$ of dimension $n$ with dual space of dimension $dn - n$ satisfying $\delta_n(\C^{\perp}) \geq n + 1$. Let $0 \leq \gamma_1, \ldots, \gamma_d \leq n$ be integer with $|\gamma| \geq n + 1$. Then we have
\[ \# \left\{ A = (a_1, \ldots, a_d) \in \C^{\perp}: \, v_n(a_i) = \gamma_i; \, 1 \leq i \leq d \right\} \leq b^{|\gamma| - n}. \]

\end{lem}

For our purposes this result is not strong enough. We use the following instead.

\begin{prp} \label{main_est_prp}

Let $\C$ be an $\Fb_b$-linear subspace of $\Fb_b^{dn}$ of dimension $n$ with dual space of dimension $dn - n$ satisfying $\delta_n(\C^{\perp}) \geq n + 1$. Let $0 \leq \lambda_i \leq \gamma_i \leq n$ be integers for all $1 \leq i \leq d$ with $|\gamma| \geq n + 1$ and $|\lambda| + d \leq n$. Then we have
\[ \# \left\{ A = (a_1, \ldots, a_d) \in \C^{\perp}: \, v_n(a_i) \leq \gamma_i; \, a_{ik} = 0 \; \forall \, \lambda_i < k < \gamma_i; \, 1 \leq i \leq d \right\} \leq b^d. \]

\end{prp}

\begin{proof}
Let $A \in \C^{\perp}$ with $v_n(a_i) \leq \gamma_i$ and for all $\lambda_i < k < \gamma_i$ with $a_{i k} = 0$ for all $1 \leq i \leq d$. Let $\gamma = (\gamma_1, \ldots, \gamma_d)$ and $\lambda = (\lambda_1, \ldots, \lambda_d)$. Then we have $A \in \V_{\gamma, \lambda}^{\perp}$. Let $\sigma$ be the number of such $i$ that $\lambda_i < \gamma_i$. Analogously to the proof of \cite[Lemma 16.26]{DP10} using Lemma \ref{1613dp10} we get
\begin{align}
\# & \left\{ A = (a_1, \ldots, a_d) \in \C^{\perp}: \, v_n(a_i) \leq \gamma_i; \, a_{ik} = 0 \; \forall \, \lambda_i < k < \gamma_i; \, 1 \leq i \leq d \right\} \notag \\
& \qquad \leq \# \left( \C^{\perp} \cap \V_{\gamma, \lambda}^{\perp} \right) \notag \\
& \qquad = b^{|\lambda| + \sigma - n} \, \# \left( \C \cap \V_{\gamma, \lambda} \right). \label{multline_cardinality}
\end{align}
Now suppose $A \in \V_{\gamma, \lambda}$. Then for all $1 \leq i \leq d$ we have
\[ \Phi_n(a_i) = \frac{a_{i, \lambda_i + 1}}{b^{\lambda_i + 1}} + \ldots + \frac{a_{i, \gamma_i - 1}}{b^{\gamma_i - 1}} + \frac{a_{i, \gamma_i + 1}}{b^{\gamma_i + 1}} + \ldots + \frac{a_{i n}}{b^n} < \frac{1}{b^{\lambda_i}} \]
in the case where $\lambda_i < \gamma_i$ and
\[ \Phi_n(a_i) = \frac{a_{i, \lambda_i + 1}}{b^{\lambda_i + 1}} + \ldots + \frac{a_{i n}}{b^n} < \frac{1}{b^{\lambda_i}} \]
elsewise. Hence, $\Phi_n^d(A)$ is contained in the $b$-adic interval
\[ \prod_{i = 1}^d \left[ \left. 0, b^{-\lambda_i} \right) \right. \]
of volume $b^{-|\lambda|}$. By Proposition \ref{prp_dig_net_duality}, $\Phi_n^d(\C)$ is a digital $(0,n,d)$-net in base $b$, and therefore, according to Remark \ref{rem_numberininterval} contains exactly $b^{n - |\lambda|}$ points which lie in a $b$-adic interval of volume $b^{-|\lambda|}$. Therefore, we have
\[ \# \left( \C \cap \V_{\gamma, \lambda} \right) \leq b^{n - |\lambda|} \]
and the result follows from \eqref{multline_cardinality} since $\sigma \leq d$.

\end{proof}

\begin{prp} \label{prp_haar_coeff_cs}

There exists a constant $c > 0$ with the following property. Let $\CS_n$ be a Chen-Skriganov type point set with $N = b^n$ points and let $\mu_{jml}$ be the $b$-adic Haar coefficient of the discrepancy function of $\CS_n$ for $j \in \N_{-1}^d, \, m \in \Dd_j$ and $l \in \B_j$. Then
\begin{enumerate}[(i)]
 	\item if $j = (-1, \ldots, -1)$ then
\[ \left| \mu_{jml} \right| \leq c \, b^{-n}, \] \label{prp_1_cs}
  \item if $j \neq (-1, \ldots, -1)$ and $|j| \leq n$ then
\[ \left| \mu_{jml} \right| \leq c \, b^{-|j| - n}, \] \label{prp_2_cs}
  \item if $j \neq (-1, \ldots, -1)$ and $|j| > n$ and $j_{\eta_1}, \ldots, j_{\eta_s} < n$ then
\[ \left| \mu_{jml} \right| \leq c \, b^{-|j| - n} \] \label{prp_3_cs}
and
\[ \left| \mu_{jml} \right| \leq c \, b^{-2|j|} \]
for all but $b^n$ coefficients $\mu_{jml}$,
  \item if $j \neq (-1, \ldots, -1)$ and $j_{\eta_1} \geq n$ or $\ldots$ or $j_{\eta_s} \geq n$ then
\[ \left| \mu_{jml} \right| \leq c \, b^{-2|j|}, \] \label{prp_4_cs}
\end{enumerate}

\end{prp}

\begin{proof}
Part \eqref{prp_1_cs} is actually Proposition \ref{prp_min1}.

To prove part \eqref{prp_2_cs} we use again the resolution of $D_{\CS_n}$ in 
\[ D_{\CS_n} = \Theta_{\CS_n} + R_{\CS_n}. \]
Let $j \in \N_{-1}^d, \, j \neq (-1, \ldots, -1), \, |j| \leq n, \, m \in \Dd_j, \, l \in \B_j$. The Walsh function series of $h_{jml}$ can be given as
\begin{align} \label{h_j_m_l_given_as}
h_{jml} = \sum_{\alpha \in \N_0^d} \left\langle h_{jml} , \wal_{\alpha} \right\rangle \wal_{\alpha}.
\end{align}
By Lemma \ref{rest_lem} we have
\[ \left| \left\langle R_{\CS_n} , h_{jml} \right\rangle \right| \leq c \, b^{-n} |I_{jm}| = c \, b^{-|j| - n}. \]
We recall that
\[ \left\langle h_{jml}, \wal_{\alpha} \right\rangle = \left\langle h_{j_1 m_1 l_1}, \wal_{\alpha_1} \right\rangle \cdot \ldots \cdot \left\langle h_{j_d m_d l_d}, \wal_{\alpha_d} \right\rangle \]
and
\[ \left\langle \hat{\chi}_{[0,y)}(t) , \wal_{\alpha} \right\rangle = \left\langle \hat{\chi}_{[0,y_1)}(t_1) , \wal_{\alpha_1} \right\rangle \cdot \ldots \cdot \left\langle \hat{\chi}_{[0,y_d)}(t_d) , \wal_{\alpha_d} \right\rangle. \]
We will use Lemmas \ref{lem_scalprod_1} and \ref{lem_scalprod_2} on each of the factors. Lemma \ref{lem_scalprod_1} gives us \linebreak $|\left\langle h_{j_i m_i l_i}, \wal_{\alpha_i} \right\rangle| \leq b^{-j_i}$ if $j_i \neq -1$ for all $i$. For all $\alpha$ with $\varrho(\alpha_i) \neq j_i + 1$ for some $i$ we have $\left\langle h_{j m l}, \wal_{\alpha} \right\rangle = 0$. We also always get $0$ if the leading digit in the $b$-adic expansion of $\alpha_i$ is not $l_i$ for some $i$. In the case where $j_i = -1$ we can get $b^{-j_i}$, by increasing the constant. From Lemma \ref{lem_scalprod_2} we have $|\left\langle \hat{\chi}_{[0,y_i)}(t_i) , \wal_{\alpha_i} \right\rangle| \leq c \, b^{-\max(\varrho(\alpha_i), \varrho(t_i))}$. Inserting Lemma \ref{theta_lem} and \eqref{h_j_m_l_given_as} we get
\begin{align*}
\left| \mu_{jml} (\Theta_{\CS_n}) \right| & = \left| \left\langle \Theta_{\CS_n} , h_{jml} \right\rangle \right| \\
& = \left| \left\langle \sum_{t \in \Dn'(C_1, \ldots, C_d)} \hat{\chi}_{[0,\cdot)}(t) , \sum_{\alpha \in \N_0^d} \left\langle h_{jml} , \wal_{\alpha} \right\rangle \wal_{\alpha} \right\rangle \right| \\
& \leq \sum_{t \in \Dn'(C_1, \ldots, C_d)} \sum_{\alpha \in \N_0^d} \left| \left\langle \hat{\chi}_{[0,\cdot)}(t) , \wal_{\alpha} \right\rangle \right| \left| \left\langle h_{jml}, \wal_{\alpha} \right\rangle \right| \\
& \leq c_1 \, b^{-j_1 - \ldots - j_d} \sum_{t \in \Dn'(C_1, \ldots, C_d)} b^{-\max(j_1, \varrho(t_1)) - \ldots - \max(j_d, \varrho(t_d))}.
\end{align*}
The summation in $\alpha$ disappears due to the following facts. The application of Lemma \ref{lem_scalprod_1} leaves only all such $\alpha$ with $\varrho(\alpha_i) = j_i + 1$ and with $l_i$ as leading digit in the $b$-adic expansion of $\alpha_i$ for all $i$. The application of Lemma \ref{lem_scalprod_2} leaves then at most one $\alpha$ per $t$, namely the one with either $\alpha_i = t_i'$ (if $\varrho(t_i) > j_i + 1$) or $\alpha_i = t_i + l_i \, b^{j_i}$ (if $\varrho(t_i) \leq j_i + 1$) for all $i$. In the cases where there is an $i$ with $\varrho(t_i) > j_i + 1$, it is possible that no $\alpha$ is left in the summation, since we still have the condition on $\alpha_i$ that the leading digit in the $b$-adic expansion is $l_i$, which cannot be guaranteed for $t_i'$.

Our next step is to break the sum above into sums where for every $t$ every coordinate either has bigger NRT weight than the corresponding coordinate of $j$ or a smaller NRT weight. Let $0 \leq r \leq d$ be the integer that is the cardinality of such $1 \leq i \leq d$ that the NRT weight is smaller. Without loss of generality we consider for every $r$ only the case where for $1 \leq i \leq r$ we have $\varrho(t_i) \leq j_i$ while for $r + 1 \leq i \leq d$ we have $\varrho(t_i) > j_i$. All the other cases follow from renaming the indices and we will just increase the constant. In the notation we split the sum
\[ \sum_{t \in \Dn'(C_1, \ldots, C_d)} \leq c_2 \, \sum_{r = 0}^d \, \sum_{t \in \Dn'_r(C_1, \ldots, C_d)} \]
where by $\Dn'_r(C_1, \ldots, C_d)$ we mean the subset of $\Dn'(C_1, \ldots, C_d)$ according to what we explained above (with ordered indices and other cases incorporated into the constant, $r$ coordinates have smaller NRT weight). So we have
\begin{align*}
\left| \mu_{jml} (\Theta_{\CS_n}) \right| & \leq c_3 \, b^{-j_1 - \ldots - j_d} \sum_{r = 0}^d \, \sum_{t \in \Dn'_r(C_1, \ldots, C_d)} b^{-j_1 - \ldots - j_r - \varrho(t_{r + 1}) - \ldots - \varrho(t_d)} \\
& = c_3 \, \sum_{r = 0}^d \, b^{-2j_1 - \ldots - 2j_r - j_{r + 1} - \ldots - j_d} \sum_{t \in \Dn'_r(C_1, \ldots, C_d)} b^{-\varrho(t_{r + 1}) - \ldots - \varrho(t_d)}.
\end{align*}
Instead of summing over $t$, we can sum over the values of $\varrho(t)$, considering the number of such $t$ that $\varrho(t_i) = \gamma_i$, $1 \leq i \leq d$.
We recall that $\CS_n = \Phi_n^d(\C_n)$. Then we denote
\[ \omega_{\gamma} = \# \left\{ A \in \C_n^{\perp}: \, v_n(a_i) = \gamma_i \, \forall \, i \, \wedge \, a_{ik} = 0 \; \forall \, j_i < k < \gamma_i; \, r + 1 \leq i \leq d \right\} \]
and 
\[ \tilde{\omega}_{\gamma} = \# \left\{ A \in \C_n^{\perp}: \, v_n(a_i) \leq \gamma_i \, \forall \, i \, \wedge \, a_{ik} = 0 \; \forall \, j_i < k < \gamma_i; \, r + 1 \leq i \leq d \right\}. \]
Let $\Gamma$ consist of all such $\gamma = (\gamma_1, \ldots, \gamma_d)$ that $0 \leq \gamma_i \leq j_i$ for $1 \leq i \leq r$, $j_i < \gamma_i \leq n$ for $r + 1 \leq i \leq d$ and $|\gamma| \geq n + 1$. Then we have
\[ \left| \mu_{jml} (\Theta_{\CS_n}) \right| \leq c_3 \, \sum_{r = 0}^d \, b^{-2j_1 - \ldots - 2j_r - j_{r+1} - \ldots - j_d} \sum_{\gamma \in \Gamma} b^{-\gamma_{r + 1} - \ldots - \gamma_d} \, \omega_{\gamma}. \]
We can apply Proposition \ref{main_est_prp} with $\lambda_i = \gamma_i, \, 1 \leq i \leq r$ and $\lambda_i = j_i, \, r + 1 \leq i \leq d$. Thereby, we get $\tilde{\omega}_{\gamma} \leq b^d$. An obvious observation is that
\[ \sum_{0 \leq \kappa_i \leq \gamma_i, \, 1 \leq i \leq d} \omega_{\kappa} \leq \tilde{\omega}_{\gamma} \]
with $\kappa = (\kappa_1, \ldots, \kappa_d)$. Recall the notation $\bar{n} = (n, \ldots, n)$. For all $\gamma \in \Gamma$ it holds that $-\gamma_{r + 1} - \ldots - \gamma_d \leq \gamma_1 + \ldots + \gamma_r - n - 1$ so,
\begin{align*}
& \left| \mu_{jml} (\Theta_{\CS_n}) \right| \leq c_3 \, \sum_{r = 0}^d \, b^{-2j_1 - \ldots - 2j_r - j_{r+1} - \ldots - j_d} \sum_{\gamma \in \Gamma} b^{-n - 1 + \gamma_1 + \ldots + \gamma_r} \, \omega_{\gamma} \\
& \leq c_4 \, \sum_{r = 0}^d \, b^{-2j_1 - \ldots - 2j_r - j_{r+1} - \ldots - j_d - n} \sum_{0 \leq \gamma_i \leq j_i, \, 1 \leq i \leq r} b^{\gamma_1 + \ldots + \gamma_r} \sum_{j_i < \gamma_i \leq n, \, r + 1 \leq i \leq d} \omega_{\gamma} \\
& \leq c_4 \, \sum_{r = 0}^d \, b^{-2j_1 - \ldots - 2j_r - j_{r+1} - \ldots - j_d - n} \, \prod_{i = 1}^r \sum_{\kappa_i = 0}^{j_i} b^{\kappa_i} \sum_{j_i < \gamma_i \leq n, \, r + 1 \leq i \leq d} \, \max_{0 \leq \gamma_i \leq j_i, \, 1 \leq i \leq r} \omega_{\gamma} \\
& \leq c_5 \, \sum_{r = 0}^d \, b^{-j_1 - \ldots - j_d - n} \sum_{0 \leq \gamma_i \leq n, \, 1 \leq i \leq d} \omega_{\gamma} \\
& \leq c_6 \, b^{-j_1 - \ldots - j_d - n} \, \tilde{\omega}_{\bar{n}} \\
& \leq c_6 \, b^{-j_1 - \ldots - j_d - n} \, b^d \\
& \leq c_7 \, b^{-|j| - n}.
\end{align*}

For the part \eqref{prp_3_cs} let $|j| > n$ and $j_{\eta_1}, \ldots, j_{\eta_s} < n$. We recall that $\CS_n$ contains exactly $N = b^n$ points and for fixed $j \in \N_{-1}^d$, the interiors of the $b$-adic intervals $I_{jm}$ are mutually disjoint. There are no more than $b^n$ such $b$-adic intervals which contain a point of $\CS_n$ meaning that all but $b^n$ intervals contain no points at all. This fact combined with Lemma \ref{lem_haar_coeff_besov_x} gives us the second statement of this part. The remaining boxes contain exactly one point of $\CS_n$ (Theorem \ref{CS_is_digital_net}). So from Lemmas \ref{lem_haar_coeff_besov_x} and \ref{lem_haar_coeff_besov_indicator} we get the first statement of this part.

Finally, let $j_{\eta_1} \geq n$ or $\ldots$ or $j_{\eta_s} \geq n$, then there is no point of $\CS_n$ which is contained in the interior of the $b$-adic interval $I_{jm}$. Thereby part \eqref{prp_4_cs} follows from Lemma \ref{lem_haar_coeff_besov_x}.

\end{proof}

\begin{rem}

Proposition \ref{prp_haar_coeff_cs} together with \eqref{parsevals_equation} gives us yet another alternative proof for Theorem \ref{thm_upper_chen}.

\end{rem}

We are ready to prove the main result of this work.

\begin{proof}[Proof of Theorem \ref{thm_main}]
At first, we assume that $N = b^n$ for $n = 2dw$ for some $w \in \N_0$. Then the point set satisfying the assertion is the Chen-Skriganov type point set $\CS_n$. Let $\mu_{jml}$ be the $b$-adic Haar coefficients of the discrepancy function of $\CS_n$. Theorem \ref{thm_besov_char} gave us an equivalent quasi-norm on $S_{pq}^r B([0,1)^d)$ so that the proof of the inequality
\[ \left( \sum_{j \in \N_{-1}^d} b^{|j|(r - \frac{1}{p} + 1) q} \left( \sum_{m \in \Dd_j, \, l \in \B_j} |\mu_{jml}|^p \right)^{\frac{q}{p}} \right)^{\frac{1}{q}} \leq C \, b^{n(r-1)} n^{\frac{d-1}{q}} \]
for some constant $C > 0$ establishes the proof of the theorem in this case. 

To estimate the expression on the left-hand side, we use Minkowski's inequality to split the sum into summands according to the cases of Proposition \ref{prp_haar_coeff_cs}. We denote
\[ \Xi_j = b^{|j|(r - \frac{1}{p} + 1) q} \left( \sum_{m \in \Dd_j, \, l \in \B_j} |\mu_{jml}|^p \right)^{\frac{1}{p}} \]
and get
\[ \left( \sum_{j \in \N_{-1}^d} \Xi_j^q \right)^{\frac{1}{q}} \leq \Xi_{(-1, \ldots, -1)} + \sum_{s = 1}^d \left[ \left( \sum_{j \in J_s^1} \Xi_j^q \right)^{\frac{1}{q}} + \left( \sum_{j \in J_s^2} \Xi_j^q \right)^{\frac{1}{q}} + \sum_{i = 1}^s \left( \sum_{j \in J_{s i}^3} \Xi_j^q \right)^{\frac{1}{q}} \right] \]
where $J_s^1$ is the set of all such $j \neq (-1, \ldots, -1)$ for which $|j| \leq n$, $J_s^2$ is the set of all such $j \neq (-1, \ldots, -1)$ for which $0 \leq j_{\eta_1}, \ldots, j_{\eta_s} \leq n-1$ and $|j| > n$ and $J_{s i}^3$ is the set of all such $j$ for which $j_{\eta_i} \geq n$.

We will show that each of the summands above can be bounded by $C \, b^{n(r-1)} n^{\frac{d-1}{q}}$ which finishes the proof.

Part \eqref{prp_1_cs} of Proposition \ref{prp_haar_coeff_cs} gives us for $j = (-1, \ldots, -1), \, m = (0,\ldots,0), \, l = (0,\ldots,0)$
\[ \Xi_j = |\mu_{jml}| \leq c_1 b^{-n} \leq c_2 b^{n(r-1)} n^{\frac{d-1}{q}}. \]
Let now $1 \leq s \leq d$. We will use \eqref{prp_2_cs} in Proposition \ref{prp_haar_coeff_cs} and Lemma \ref{lem_lambda_s_minus_1}. The summation over $l \in \B_j$ can be incorporated into the constant and we recall that $\# \Dd_j = b^{|j|}$. Hence (using the fact that $r < 0$) we have
\begin{align*}
\left( \sum_{j \in J_s^1} \Xi_j^q \right)^{\frac{1}{q}} & \leq c_3 \left( \sum_{j \in J_s^1} b^{|j|(r - \frac{1}{p} + 1) q} \left( \sum_{m \in \Dd_j} b^{(-|j| - n) p} \right)^{\frac{q}{p}} \right)^{\frac{1}{q}} \\
& = c_3 \left( \sum_{j \in J_s^1} b^{(|j| r - n ) q} \right)^{\frac{1}{q}} \\
& \leq c_4 \left( \sum_{\lambda = 0}^n b^{(\lambda r - n) q} (\lambda + 1)^{s - 1} \right)^{\frac{1}{q}} \\
& \leq c_5 \, n^{\frac{s - 1}{q}} \, b^{-n} \left( \sum_{\lambda = 0}^n b^{\lambda r q} \right)^{\frac{1}{q}} \\
& \leq c_6 \, n^{\frac{d - 1}{q}} \, b^{n (r-1)}.
\end{align*}
From \eqref{prp_3_cs} in the same proposition (using the fact that $r - \frac{1}{p} < 0$ and $r - 1 \leq 0$) we have
\begin{align*}
\left( \sum_{j \in J_s^2} \Xi_j^q \right)^{\frac{1}{q}} & \leq c_7 \left( \sum_{j \in J_s^2} b^{|j|(r - \frac{1}{p} + 1) q} \, b^{n \frac{q}{p}} \, b^{(-|j| - n) q} \right)^{\frac{1}{q}} \\
& \quad + c_8 \left( \sum_{j \in J_s^2} b^{|j|(r - \frac{1}{p} + 1)q} \, b^{|j| \frac{q}{p}} \, b^{-2|j| q} \right)^{\frac{1}{q}} \\
& = c_7 \left( \sum_{j \in J_s^2} b^{\left[ |j| (r - \frac{1}{p}) + \frac{n}{p} - n \right] q} \right)^{\frac{1}{q}} \\
& \quad + c_8 \left( \sum_{j \in J_s^2} b^{|j|(r - 1)q} \right)^{\frac{1}{q}}\\
& \leq c_7 \left( \sum_{\lambda = n+1}^{s(n - 1)} (\lambda + 1)^{s-1} b^{\left[ \lambda(r - \frac{1}{p}) + \frac{n}{p} - n \right] q} \right)^{\frac{1}{q}} \\
& \quad + c_8 \left( \sum_{\lambda = n+1}^{s(n - 1)} (\lambda + 1)^{s-1} b^{\lambda (r-1) q} \right)^{\frac{1}{q}} \\
& \leq c_9 \, n^{\frac{s - 1}{q}} b^{\frac{n}{p} - n} \left( \sum_{\lambda = n + 1}^{s(n - 1)} b^{\lambda (r - \frac{1}{p}) q} \right)^{\frac{1}{q}} + c_{10} \, n^{\frac{s - 1}{q}} \left( \sum_{\lambda = n + 1}^{s(n - 1)} b^{\lambda (r - 1) q} \right)^{\frac{1}{q}} \\
& \leq c_{11} \, n^{\frac{s - 1}{q}} b^{\frac{n}{p} - n} \, b^{n (r - \frac{1}{p})} + c_{12} \, n^{\frac{s - 1}{q}} b^{n (r - 1)} \\
& \leq c_{13} \, n^{\frac{d - 1}{q}} \, b^{n(r-1)}.
\end{align*}
Part \eqref{prp_4_cs} in Proposition \ref{prp_haar_coeff_cs} gives us for any $1 \leq i \leq s$
\begin{align*}
\left( \sum_{j \in J_{s i}^3} \Xi_j^q \right)^{\frac{1}{q}} & \leq c_{14} \left( \sum_{j \in J_{s i}^3} b^{|j|(r - \frac{1}{p} + 1) q} \, b^{|j| \frac{q}{p}} \, b^{-2|j| q} \right)^{\frac{1}{q}}\\
& \leq c_{15} \left( \sum_{\lambda = n}^\infty (\lambda + 1)^{s-1} b^{\lambda(r - 1) q} \right)^{\frac{1}{q}}\\
& \leq c_{16} n^{\frac{d-1}{q}} b^{n(r-1)}.
\end{align*}

The cases $p = \infty$ and $q = \infty$ have to be modified in the usual way.

Now let $N \geq 2$ be an arbitrary integer. Then we find $w \in \N_0$ such that,
\[ b^{2d(w-1)} < N \leq b^{2dw} \]
and for $n = 2dw$ we construct the point set $\CS_n$. Since $\CS_n$ is a digital $(0,n,d)$-net, the point set
\[ \tilde{\P} = \CS_n \cap \left( \left[ \left. 0, \frac{N}{b^n} \right) \right. \times \Q^{d-1} \right) \]
contains exactly $N$ points. We define the point set
\[ \P = \left\{ \left( \frac{b^n}{N} x_1, x_2, \ldots, x_d \right): \, (x_1, \ldots, x_d) \in \tilde{\P} \right\}. \]
Then
\[ D_{\P}(y) = \# \left( \left[ \left. 0, \frac{N}{b^n} y_1 \right) \right. \times [0,y_2) \times \ldots \times [0,y_d) \cap \tilde{\P} \right) - N y_1 \cdot \ldots \cdot y_d. \]
Therefore, by scaling the first coordinate with the factor $\frac{N}{b^n} \in \big( \frac{1}{2}, 1 \big]$, we estimate (with certain constant $c_1 > 0$)
\[ \left\| \mathcal{F}^{-1}(\varphi_k \mathcal{F} D_{\P}) | L_p(\Q^d) \right\| \leq c_1 \left\| \mathcal{F}^{-1}(\varphi_k \mathcal{F} D_{\CS_n}) | L_p(\Q^d) \right\|. \]
Finally, we get
\[ \left\| D_{\P} | S_{pq}^r B(\Q^d) \right\| \leq c_1 \left\| D_{\P} | S_{pq}^r B(\Q^d) \right\| \leq c_2 \, N^{r-1} \, (\log N)^{\frac{d-1}{q}}. \]
\end{proof}

\begin{rem}

We remind the reader of Lemma \ref{non_existence_for_nets} which already made a necessity for $b$ to be large to ensure the existence of $(0,n,d)$-nets. But this dependence was linear, namely $b > d - 2$. The construction of $\CS_n$ demands for $b$ to be even larger, namely $b \geq 2d^2$.

\end{rem}

%% file: ConclusionDominatingMixed.tex
\section{Conclusion}
We summarize the discrepancy results for spaces with dominating mixed smoothness. Especially we would like to give the cases where the lower and the upper bounds coincide. The following combines results from Theorems \ref{lower_to_main_B}, \ref{thm_main} and Corollaries \ref{lower_to_main_F}, \ref{cor_f_spaces}, \ref{lower_to_main_H}, \ref{cor_h_spaces}.

\begin{thm} \label{thm_main_conclusion}

\mbox{}
\begin{enumerate}[(i)]
	\item Let $1 \leq p,q \leq \infty$ and $q < \infty$ if $p = 1$ and $q > 1$ if $p = \infty$. Let $0 < r < \frac{1}{p}$. Then there exist constants $c_1, C_1 > 0$ such that, for any integer $N \geq 2$, we have
	\[ c_1 \, N^{r-1} \, (\log N)^{\frac{d-1}{q}} \leq D^{S_{pq}^r B}(N) \leq C_1 \, N^{r-1} \, (\log N)^{\frac{d-1}{q}}. \]
	\item Let $1 \leq p,q < \infty$. Let $0 < r < \frac{1}{\max(p,q)}$. Then there exist constants $c_2, C_2 > 0$ such that, for any integer $N \geq 2$, we have
	\[ c_2 \, N^{r-1} \, (\log N)^{\frac{d-1}{q}} \leq D^{S_{pq}^r F}(N) \leq C_2 \, N^{r-1} \, (\log N)^{\frac{d-1}{q}}. \]
	\item Let $1 \leq p < \infty$. Let $0 \leq r < \frac{1}{\max(p,2)}$. Then there exist constants $c_3, C_3 > 0$ such that, for any integer $N \geq 2$, we have
	\[ c_3 \, N^{r-1} \, (\log N)^{\frac{d-1}{2}} \leq D^{S_p^r H}(N) \leq C_3 \, N^{r-1} \, (\log N)^{\frac{d-1}{2}}. \]
\end{enumerate}

\end{thm}

\begin{rem}

The constants in Theorem \ref{thm_main_conclusion} depend only on the dimension and on the parameters $b, p, q, r$. In particular they do not depend on $N$. We point out that the conditions on $r$ in Theorem \ref{thm_main_conclusion} are better in the case of the Besov spaces. For the Triebel-Lizorkin spaces in the case $\frac{1}{q} \leq r < \frac{1}{p}$ for $p < q$ the upper bound is still an open problem. Therefore, for the Sobolev spaces in the case $\frac{1}{2} \leq r < \frac{1}{p}$ for $p < 2$ the upper bound is an open problem.

\end{rem}

%% file: IntegrationErrors.tex
In this chapter we are going to deal with the applications of discrepancy theory. We already have mentioned the connection between the discrepancy function and the integration errors. Now we will conclude concrete results. For spaces with dominating mixed smoothness this connection is given by \cite[Theorem 6.11, Remark 6.28]{T10a}. We start with the definition of the error.
\begin{df}
Let $N$ be a positive integer and $M(\Q^d)$ be some Banach space of functions on $\Q^d$. Let $M_0^1(\Q^d)$ be the subset of the unit ball of $M(\Q^d)$ with the property that the extensions of all elements of $M_0^1(\Q^d)$ vanish whenever one of the coordinates of the argument is $1$. The error of quadrature formulas in $M(\Q^d)$ with $N$ points is
\[ \err_N(M) = \inf_{\{ x_1, \ldots, x_N \} \subset [0,1)^d} \sup_{f \in M^1_0([0,1)^d)} \left| \int_{[0,1)^d} f(x) \, \dint x - \frac{1}{N}\sum_{i = 1}^N f(x_i) \right|. \]
\end{df}
We now quote the aforementioned result and apply it then to the results from Theorem \ref{thm_main_conclusion}. The reader might be confused by the different notation in \cite{T10a}. In \cite[(3.182),(3.187),(5.5),(5.88),(6.7),(6.32)]{T10a} the necessary definitions and facts for the understanding can be found.

Let
\[ \frac{1}{p} + \frac{1}{p'} = \frac{1}{q} + \frac{1}{q'} = 1. \]

\begin{prp} \label{prp_triebel611}

Let $1 \leq p,q \leq \infty$ and $\frac{1}{p} < r < \frac{1}{p} + 1$. Then there exist constants $c_1, c_2 > 0$ such that, for any integer $N \geq 2$, we have
\[ c_1 \, D^{S_{p'q'}^{1-r} B}(N) \leq \err_N(S_{pq}^r B) \leq c_2 \, D^{S_{p'q'}^{1-r} B}(N). \]

\end{prp}

Hence, we can conclude bounds for the integration error. We start with the lower bounds.

\begin{thm} \label{int_err_low_B}

Let $1 \leq p,q \leq \infty$ and $q < \infty$ if $p = 1$ and $q > 1$ if $p = \infty$. Let $\frac{1}{p} < r < \frac{1}{p} + 1$. Then there exists a constant $c > 0$ such that, for any integer $N \geq 2$, we have
\[ \err_N(S_{pq}^r B) \geq c \, \frac{(\log N)^{\frac{(q-1)(d-1)}{q}}}{N^r}. \]

\end{thm}

\begin{proof}
From Proposition \ref{prp_triebel611} we have
\[ \err_N(S_{pq}^r B) \geq c_1 \, D^{S_{p'q'}^{1-r} B}(N) \]
for $\frac{1}{p} < r < \frac{1}{p} + 1$. From Theorem \ref{lower_to_main_B} we have
\[ D^{S_{p'q'}^{1-r} B}(N) \geq c \, N^{1-r-1} \, \left( \log N \right)^{\frac{d-1}{q'}} \]
for $\frac{1}{p'} - 1 < 1 - r < \frac{1}{p'}$ which is equivalent to $\frac{1}{p} < r < \frac{1}{p} + 1$. So the lower bounds follow. The additional conditions for $p$ and $q$ also come from Theorem \ref{lower_to_main_B}.

\end{proof}

\begin{cor} \label{int_err_low_F}

Let $1 \leq p,q < \infty$. Let $\frac{1}{p} < r < \frac{1}{\max(p,q)} + 1$. Then there exists a constant $c > 0$ such that, for any integer $N \geq 2$, we have
\[ \err_N(S_{pq}^r F) \geq c \, \frac{(\log N)^{\frac{(q-1)(d-1)}{q}}}{N^r}. \]

\end{cor}

\begin{proof}
From Corollary \ref{cor_emb_BF} we have $S_{\max(p,q),q}^r B(\Q^d) \hookrightarrow S_{pq}^r F(\Q^d)$, therefore, we get
\[ \err_N(S_{pq}^r F) \geq \err_N(S_{\max(p,q),q}^r B) \geq c \, \frac{(\log N)^{\frac{(q-1)(d-1)}{q}}}{N^r} \]
for $\frac{1}{\max(p,q)} < r < \frac{1}{\max(p,q)} + 1$ from Theorem \ref{int_err_low_B}. But we also need to guarantee pointwise evaluation for the integration, therefore we get the restriction $r > \frac{1}{p}$ (see \cite[Section 4.2.1]{T10a}).

\end{proof}

We add results for the Sobolev spaces ($q = 2$).

\begin{cor} \label{int_err_low_H}

Let $1 \leq p < \infty$. Let $\frac{1}{p} < r < \frac{1}{\max(p,2)} + 1$. Then there exists a constant $c > 0$ such that, for any integer $N \geq 2$, we have
\[ \err_N(S_p^r H) \geq c \, \frac{(\log N)^{\frac{d-1}{2}}}{N^r}. \]

\end{cor}

Now we turn to the upper bounds.

\begin{thm}  \label{thm_int_err_up_B}

Let $1 \leq p,q \leq \infty$. Let $\frac{1}{p} < r < 1$. Then there exists a constant $C > 0$ such that, for any integer $N \geq 2$, we have
\[ \err_N(S_{pq}^r B) \leq C \, \frac{(\log N)^{\frac{(q-1)(d-1)}{q}}}{N^r}. \]

\end{thm}

\begin{proof}
From Proposition \ref{prp_triebel611} we have
\[ \err_N(S_{pq}^r B) \leq c_2 \, D^{S_{p'q'}^{1-r} B}(N) \]
for $\frac{1}{p} < r < \frac{1}{p} + 1$. From Theorem \ref{thm_main} we have
\[ D^{S_{p'q'}^{1-r} B}(N) \leq C \, N^{1-r-1} \, \left( \log N \right)^{\frac{d-1}{q'}} \]
for $0 < 1 - r < \frac{1}{p'}$ which is equivalent to $\frac{1}{p} < r < 1$. Hence, the bounds follow.

\end{proof}

\begin{rem}
 
Just recently Ullrich proved the same upper bound in the case $1 \leq r < 2$ for the plane. His approach in \cite{U13} were diadic Hammersley point sets.

\end{rem}

\begin{cor}  \label{cor_int_err_up_F}

Let $1 \leq p < \infty$. Let $1 \leq q \leq \infty$. Let $\frac{1}{\min(p,q)} < r < 1$. Then there exists a constant $C > 0$ such that, for any integer $N \geq 2$, we have
\[ \err_N(S_{pq}^r F) \leq C \, \frac{(\log N)^{\frac{(q-1)(d-1)}{q}}}{N^r}. \]

\end{cor}

\begin{proof}
From Corollary \ref{cor_emb_BF} we have $S_{pq}^r F(\Q^d) \hookrightarrow S_{\min(p,q),q}^r B(\Q^d)$, therefore, we get
\[ \err_N(S_{pq}^r F) \leq \err_N(S_{\min(p,q),q}^r B) \leq C \, \frac{(\log N)^{\frac{(q-1)(d-1)}{q}}}{N^r} \]
for $\frac{1}{\min(p,q)} < r < 1$ from Theorem \ref{thm_int_err_up_B}.

From the first part of Proposition \ref{embeddings_SpqrBF} we have $S_{p, \infty}^r F(\Q^d) \hookrightarrow S_{p, \infty}^r B(\Q^d)$, therefore, we get the assertion analogously to the case above for $\frac{1}{p} < r < 1$.

\end{proof}

\begin{cor}  \label{cor_int_err_up_H}

Let $1 \leq p < \infty$. Let $\frac{1}{\min(p,2)} < r < 1$. Then there exists a constant $C > 0$ such that, for any integer $N \geq 2$, we have
\[ \err_N(S_p^r H) \leq C \, \frac{(\log N)^{\frac{d-1}{2}}}{N^r}. \]

\end{cor}

\begin{proof}
The assertion from Corollary \ref{cor_int_err_up_F} for $q = 2$.

\end{proof}

Again we summarize the results for the cases where we have the same lower and upper bounds.

\begin{thm} \label{thm_int_err}

\mbox{}
\begin{enumerate}[(i)]
	\item Let $1 \leq p,q \leq \infty$ and $q < \infty$ if $p = 1$ and $q > 1$ if $p = \infty$. Let $\frac{1}{p} < r < 1$. Then there exist constants $c_1, C_1 > 0$ such that, for any integer $N \geq 2$, we have
	\[ c_1 \, \frac{(\log N)^{\frac{(q-1)(d-1)}{q}}}{N^r} \leq \err_N(S_{pq}^r B) \leq C_1 \, \frac{(\log N)^{\frac{(q-1)(d-1)}{q}}}{N^r}. \]
	\item Let $1 \leq p,q < \infty$. Let $\frac{1}{\min(p,q)} < r < 1$. Then there exist constants $c_2, C_2 > 0$ such that, for any integer $N \geq 2$, we have
	\[ c_2 \, \frac{(\log N)^{\frac{(q-1)(d-1)}{q}}}{N^r} \leq \err_N(S_{pq}^r F) \leq C_2 \, \frac{(\log N)^{\frac{(q-1)(d-1)}{q}}}{N^r}. \]
	\item Let $1 \leq p < \infty$. Let $\frac{1}{\min(p,2)} < r < 1$. Then there exist constants $c_3, C_3 > 0$ such that, for any integer $N \geq 2$, we have
	\[ c_3 \, \frac{(\log N)^{\frac{d-1}{2}}}{N^r} \leq \err_N(S_p^r H) \leq C_3 \, \frac{(\log N)^{\frac{d-1}{2}}}{N^r}. \] \label{BTY_quote}
\end{enumerate}

\end{thm}

\begin{rem}

The constants do not depend on $N$ but they do depend on $d, b, p, q, r$. The reason that we have parameters $p', q', 1 - r$ in Proposition \ref{prp_triebel611} is that those results come from duality arguments. Part \eqref{BTY_quote} in Theorem \ref{thm_int_err} is the $d$-dimensional counterpart (for $r \leq 1$) of the first part of \cite[Theorem 4.1]{T03}.

\end{rem}

For the Triebel-Lizorkin spaces in the case $\frac{1}{p} \leq r < \frac{1}{q}$ for $p > q$ the upper bound is still an open problem. Therefore, for the Sobolev spaces in the case $\frac{1}{p} \leq r < \frac{1}{2}$ for $p > 2$ the upper bound is an open problem.